\documentclass[preprint,11pt]{elsarticle}
\usepackage[toc,page]{appendix}

\usepackage{amsmath,amsthm,amsfonts,mathtools,bm,hyperref}
\usepackage{amssymb}
\usepackage{dsfont}
\usepackage{mathrsfs}
\usepackage{multirow}
\usepackage{tikz,pgfplots}
\usetikzlibrary{patterns}
\usepackage{subcaption}
\pgfplotsset{compat=newest}
\usepackage{float,indentfirst,algorithm}
\usepackage{pgfplots}
\usepackage{booktabs}
\usepackage{hyperref}
\usepackage{multicol}
\usepackage{nomencl}
\usepackage{ulem}
\usepackage[margin=1in]{geometry}
\makenomenclature

\usepackage{etoolbox}
\renewcommand\nomgroup[1]{%
  \item[\bfseries
  \ifstrequal{#1}{A}{Symbols}{%
  \ifstrequal{#1}{B}{Operators}{}}%
]}

\newtheorem{theorem}{Theorem}
\theoremstyle{plain}

\newtheorem{remark}{Remark}
\newtheorem{definition}{Definition}
\newtheorem{lemma}{Lemma}
\newtheorem{proposition}{Proposition}
\usepackage[capitalise]{cleveref}
\usepackage{appendix}
\usepackage{multicol}
\usepackage{nomencl}

\crefname{equation}{equation}{equations}
\Crefname{equation}{Equation}{Equations}

\newcommand{\mean}{m}
\newcommand{\pmean}{\widehat{m}}
\newcommand{\pp}{\widehat{\theta}}
\newcommand{\py}{\widehat{y}}
\newcommand{\ppy}{\widehat{y}}
\newcommand{\Cov}{C}
\newcommand{\pCov}{\widehat{C}}

\newcommand{\N}{\mathcal{N}}
\newcommand{\G}{\mathcal{G}}
\newcommand{\F}{\mathcal{F}}
\newcommand{\Fu}{\mathcal{F}_u}
\newcommand{\E}{\mathbb{E}}
\newcommand{\I}{\mathbb{I}}
\newcommand{\R}{\mathbb{R}}
\newcommand{\Z}{\mathbb{Z}}
\newcommand{\bbN}{\mathbb{N}}
\newcommand{\bbR}{\mathbb{R}}
\newcommand{\bbZ}{\mathbb{Z}}
\newcommand{\cB}{\mathcal{B}}
\newcommand{\cL}{\mathcal{L}}
\newcommand{\cM}{\mathcal{M}}

\newcommand{\bigO}{\mathcal{O}}

\DeclareMathOperator*{\argmin}{arg\,min}

\makeatletter
\def\ps@pprintTitle{%
   \let\@oddhead\@empty
   \let\@evenhead\@empty
   \def\@oddfoot{\reset@font\hfil\thepage\hfil}
   \let\@evenfoot\@oddfoot
}
\makeatother

\definecolor{darkred}{rgb}{.7,0,0}
\definecolor{darkblue}{rgb}{0,0,.7}
\definecolor{darkgreen}{rgb}{0,.7,0}

\begin{document}

\begin{abstract}

This paper is focused on the optimization approach to
the solution of inverse problems. We introduce a stochastic 
dynamical system  in which the parameter-to-data map is embedded,
with the goal of employing techniques
from nonlinear Kalman
filtering to estimate the parameter given the data.
The extended Kalman filter (which we refer to
as ExKI in the context of inverse problems) can be effective
for some inverse problems approached this way, but is impractical
when the forward map is not readily differentiable and is given as a black box, and also
for high dimensional parameter spaces because of the need to propagate large covariance
matrices. Application of ensemble Kalman filters, for example use of
the ensemble Kalman inversion (EKI) algorithm, has emerged as a useful tool which overcomes both of these issues: it is
derivative free and works with a low-rank covariance approximation formed from
the ensemble. In this paper, we work with the ExKI, EKI, and a variant on EKI which we term unscented Kalman inversion (UKI). 

The paper contains two main contributions. 
Firstly, we identify a novel stochastic dynamical system 
in which the parameter-to-data
map is embedded. We present
theory in the linear case to show exponential convergence
of the mean of the filtering distribution to the solution of a 
regularized least squares problem. This is in
contrast to previous work in which the EKI has been employed 
where the dynamical system used leads to algebraic convergence to an unregularized problem.
Secondly, we show that the application of the UKI to this novel 
stochastic dynamical system yields improved inversion results, 
in comparison with the application of EKI to the same  novel stochastic dynamical system.

The numerical experiments include proof-of-concept linear examples and various applied nonlinear inverse problems: learning of
permeability parameters in subsurface flow; learning the damage field from structure deformation;
learning the Navier-Stokes initial condition from solution data at positive times;
learning subgrid-scale
parameters in a general circulation model~(GCM) from time-averaged
statistics.
\end{abstract}


\begin{frontmatter}

  \title{Iterated Kalman Methodology For Inverse Problems}

  \author[rvt1]{Daniel~Zhengyu~Huang}
  \ead{dzhuang@caltech.edu}
  
  \author[rvt1]{Tapio Schneider}
  \ead{tapio@caltech.edu}

  \author[rvt1]{Andrew M. Stuart}
  \ead{astuart@caltech.edu}

  \address[rvt1]{California Institute of Technology, Pasadena, CA}

\end{frontmatter}

\nomenclature[A]{$\theta$}{unknown parameter vector} 
\nomenclature[A]{$y$}{observation vector} 
\nomenclature[A]{$\py$}{observation vector mean in prediction step} 
\nomenclature[A]{$Y$}{observation set} 
\nomenclature[A]{$\G$}{mapping between parameter space and observation space} 
\nomenclature[A]{$\F$}{filter} 
\nomenclature[A]{$t$}{time} 
\nomenclature[A]{$N_\theta$}{dimension of unknown parameter vector} 
\nomenclature[A]{$N_y$}{dimension of observation vector} 
\nomenclature[A]{$\eta$}{observation error} 
\nomenclature[A]{$\Sigma_{\eta}$}{observation error covariance}
\nomenclature[A]{$\omega$}{artificial evolution error} 
\nomenclature[A]{$\Sigma_{\omega}$}{artificial evolution error covariance}
\nomenclature[A]{$\nu$}{artificial observation error} 
\nomenclature[A]{$\Sigma_{\nu}$}{artificial observation error covariance}
\nomenclature[A]{$\mean$}{conditional mean} 
\nomenclature[A]{$\pmean$}{conditional mean in prediction step}
\nomenclature[A]{$\Cov$}{conditional covariance} 
\nomenclature[A]{$\pCov$}{conditional covariance in prediction step}
\nomenclature[A]{$\Phi$}{least-square function}
\nomenclature[B]{$c$}{unscented points weight}
\nomenclature[B]{$W$}{unscented points weight}
\nomenclature[B]{$\kappa$}{unscented points weight}
\nomenclature[B]{$a$}{unscented points weight}

\nomenclature[B]{$\beta$}{unscented points weight}
\nomenclature[B]{$\Box_{n}$}{$n$-th time step}
\nomenclature[B]{$\Box^{j}$}{$j$-th ensemble particle}
\nomenclature[B]{$\Box_{(i)}$}{$i$-th component}


\section{Introduction}
\label{sec:I}

\subsection{Overview}
\label{ssec:over}
This paper is devoted to optimization approaches to calibrating models with observational data. The basic problem is formulated as recovering unknown model parameters $\theta \in \bbR^{N_{\theta}}$ from
noisy observation $y \in \bbR^{N_y}$ given by
\begin{equation}
\label{eq:KI}
    y = \G(\theta) + \eta;
\end{equation}
here $\G$ denotes the parameter-to-data map which, for the applications
we have in mind, generally requires solving partial differential equations, and $\eta \sim \N(0,\Sigma_{\eta})$ denotes the Gaussian observation error.
Consider now the stochastic dynamical system 
\begin{subequations}
\label{eq:std}
  \begin{align}
  &\textrm{evolution:}    &&\theta_{n+1} = \alpha \theta_{n}  + (1 - \alpha) r_0 +  \omega_{n+1}, &&\omega_{n+1} \sim \N(0,\Sigma_{\omega}),\\
  &\textrm{observation:}  &&y_{n+1} = \G(\theta_{n+1}) + \nu_{n+1}, &&\nu_{n+1} \sim \N(0,\Sigma_{\nu}).
\end{align}
\end{subequations}
We assume that the artificial evolution error covariance $\Sigma_{\omega} \succ 0$, the artificial observation error covariance $\Sigma_{\nu} \succ 0$, and the regularization parameter $\alpha \in (0,1]$, whilst $r_0$ is an arbitrary vector.
\footnote{We write $A \succ 0$ when $A$ is strictly positive-definite, and will also write
$A \prec B$ when $B - A$ is strictly positive-definite and $A \preceq B$ when $B - A$ is positive semi-definite.}
We study methods to determine $\theta$ from $y$ given by \eqref{eq:KI}
by employing filtering methods to find 
$\theta_n$ given $Y_n:=\{y_\ell\}_{\ell=1}^n$, in the setting where $y_\ell \equiv y$ for all
$\ell \in \bbN.$

Note that dynamical system (\ref{eq:std}a) for $\theta_n$ has, for $\alpha \in (0,1)$,
statistical equilibrium given by the Gaussian $\N(r_0,(1-\alpha^2)^{-1}\Sigma_{\omega}).$
The output of this statistical model is then repeatedly exposed to the observations, expressed via (\ref{eq:std}b) with $y_{n+1}$ set to the data $y$, and hence it is intuitive that
filtering methods will deliver an estimate of $\theta$ solving \eqref{eq:KI} as $n
\to \infty.$ Such a method, in the special case $\alpha=1, \Sigma_{\omega}=0, \Sigma_{\nu} = \Sigma_{\eta}$, is the
basis of the ensemble Kalman inversion (EKI) algorithm as proposed in \cite{iglesias2013ensemble}.
The two main takeaway messages of this paper are firstly to highlight
the benefits of choosing $\alpha \in (0,1)$ and $\Sigma_{\omega} \succ 0$, and secondly to demonstrate
that application of the unscented Kalman filter improves on the ensemble Kalman filter, leading to unscented Kalman inversion (UKI). 

The primary issue with the choice $\alpha=1$ is that it leads to over-fitting for problems
in which $N_{\theta}>N_{y}$, as shown in \cite{iglesias2013ensemble}. One approach to
deal with this is to use an adaptive modification of the basic EKI algorithm, based
on an analogy with the Levenberg-Marquardt algorithm, as developed 
in \cite{iglesias2016regularizing}; however, this leads to a need for
stopping criterion and the area is still being developed \cite{iglesias2020adaptive}.   
Another approach is to build Tikhonov regularization
directly into the inverse problem, before applying a filtering algorithm to
\eqref{eq:std} with $\alpha=1, \Sigma_{\omega}=0$, an approach introduced 
in \cite{chada2020tikhonov}. However, this leads to an algorithm which requires the
inversion of covariance matrices on spaces of dimension $N_{\theta}+N_y$ which
is undesirable for many problems concerning inference about fields, where
$N_{\theta} \gg 1.$ This issue is removed if the continuum limit of the
algorithm is used \cite{chada2020tikhonov}. However, practical experience
with using time-steppers for continuum limits of ensemble Kalman filtering
algorithms is in its infancy and current implementations of the
methods in \cite{chada2020tikhonov,schillings2017analysis,garbuno2020interacting,garbuno2020affine,ding2020ensemble} are not competitive with algorithms which start directly from
a discrete time formulation.

Central to both the optimization and probabilistic approaches to inversion
is the regularized objective function $\Phi_R(\theta)$ defined by
\begin{subequations}
\label{eq:KI2}
\begin{align}
\Phi_R(\theta) &:= \Phi(\theta)+\frac{1}{2}\lVert\Sigma_{0}^{-\frac{1}{2}}(\theta - r_0) \rVert^2,\\
    \Phi(\theta) &:= \frac{1}{2}\lVert\Sigma_{\eta}^{-\frac{1}{2}}(y - \G(\theta)) \rVert^2,
\end{align}
\end{subequations}
where $\Sigma_{\eta} \succ 0$ normalizes the model-data misfit $\Phi$ by means of the known
error statistics of the noise, prior mean $r_0$ encodes prior information about $\theta$,
and prior covariance $\Sigma_{0} \succ 0$ normalizes the prior information.
We will connect the parameters of \eqref{eq:std} 
for $\theta$ to a form of regularization of the inverse
problem. In this context it is worth noticing that, for linear
problems, the implied Tikhonov regularization has implied mean $r_0$, whilst the implied covariance $\Sigma_0$ of the regularization term is defined implicitly via  limit of an iterative procedure. Parameter $\alpha \in (0,1)$ controls the size of the regularization effect;
and when $\alpha=1$ the regularization effect disappears,
along with dependence of (\ref{eq:std}a) on $r_0$. Thus $\alpha=1$ is useful primarily for over-determined problems.

\subsection{Our Contributions}
\label{ssec:cont}

We make the following contributions to the
study of the solution of inverse problems by means of filtering methods:
\begin{itemize}
    \item we introduce a filtering-based approach to solving the inverse problem \eqref{eq:KI}, based on the novel stochastic dynamical system formulation \eqref{eq:std};  
    
    \item by studying linear problems we demonstrate that the methodology induces a form
    of Tikhonov regularization  and we prove an exponential convergence of the algorithm 
    to the minimizer of the Tikhonov-regularized problem, in the linear case;
    
    \item we introduce a Gaussian approximation for the filtering distribution 
    defined by \eqref{eq:std} and, from it, derive extended Kalman, ensemble Kalman
    and unscented Kalman (ExKI, EKI and UKI respectively) algorithms for the inverse problem
    \eqref{eq:KI}, applicable in the general nonlinear case;
    
    \item the algorithms are tested on a wide range of problems, including linear test
    problems, inversion for spatial fields in a variety of continuum mechanics
    applications, and the learning of parameters in chaotic dynamical systems, using time-averaged data;
    
    \item we show that UKI outperforms EKI, with both employed
    in the context of the stochastic dynamical model \eqref{eq:std},
    for a wide range of inverse problems with unknown parameter space of moderate dimension.
     
    \end{itemize}

    Taken together, the theoretical framework we develop and the numerical results we present show that the UKI, applied to the stochastic dynamical system \eqref{eq:std}, is a competitive methodology for solving inverse problems and parameter estimation problems defined by an expensive  black-box forward model; indeed the UKI is shown to
    outperform the EKI in settings where the number 
    of parameters $N_{\theta}$ is of moderate size and the black-box is not
    readily differentiable so that ExKI methods are not applicable. Other 
    ensemble filters, such as the ensemble adjustment and ensemble transform
    Kalman filters could also be used in place of unscented Kalman filters, and
    similar performance is to be expected. This issue is explored in detail in
    \cite{huang2022efficient} where ideas introduced in this paper are developed further
    in order to approximate the Bayesian posterior distribution for inverse
    problem \eqref{eq:KI}. We note that, as with the use of most nonlinear variants 
    of the Kalman filter, rigorous justification beyond the linear setting is not
    currently available, but that our numerical results demonstrate effectiveness
    in a wide range of nonlinear inverse problems. The use of
    interacting particle systems to solve inverse problems with
    multimodal distributions, far from Gaussian, is considered
    in \cite{reich2021fokker} and the derivation of mean-field
    limits of ensemble Kalman methods for inversion, viewed
    as interacting particle systems is established in \cite{ding2021ensemble,ding2021ensembleb}.

We conclude this introductory section with a deeper literature review relating
to the contributions we make in this paper, in Subsection \ref{ssec:lit}.  Then, in Section \ref{sec:A} we
introduce a conceptual algorithm based on a Gaussian approximation of the filtering distribution
associated with \eqref{eq:std}; we then derive the ExKI, UKI, and EKI 
algorithms as approximations to this conceptual Gaussian algorithm.
 In Section \ref{sec:T} we study the methodology for linear problems,
obtaining insight into the regularization conferred by (\ref{eq:std}a);
we study the relationship of the methodology to other gradient-based optimization techniques; we derive continuous-time limits in the nonlinear setting. Section \ref{sec:V} describes variants on the basic conceptual algorithm that may be useful in some settings, and in Section \ref{sec:app} we present numerical
results demonstrating the performance of the inversion methodology
introduced in this paper.
The code relating to numerical experiments presented
in Section \ref{sec:app} is accessible online:
\begin{center}
  \url{https://github.com/Zhengyu-Huang/InverseProblems.jl}
\end{center}

\subsection{Literature Review}
\label{ssec:lit}

The focus of this paper is mainly on derivative-free inversion by means of iterative techniques aimed at
solving the optimization problem   
defined by minimization of $\Phi_R$, or
variants of this problem \cite{engl1996regularization}.
However, even in the optimization setting, the methods introduced in this paper
are closely related to iterative methods applied in Bayesian (probabilistic)
inversion. In the Bayesian approach to the inverse problem \eqref{eq:KI} \cite{kaipio2006statistical,dashti2013bayesian} the posterior distribution is given by
\begin{equation}
\label{eq:post}
    \mu(d\theta)=\frac{1}{Z}\exp\bigl(-\Phi(\theta)\bigr)\mu_0(d\theta),
\end{equation}
where $\mu_0=\N(r_0,\Sigma_{0})$ is the prior and $\mu$ is the posterior.
A commonly adopted iterative approach to
solving the problem of sampling from $\mu$ is the finite time 
approach known as sequential Monte Carlo
(SMC) -- see \cite{del2006sequential,chopin2020introduction}, and \cite{beskos2015sequential} for applications
to inverse problems. The basic idea, upon which there
are many variants, is to consider the sequence of
measures $\mu_n$ defined by
\begin{equation}
\label{eq:post-SMC}
    \mu_{n+1}(d\theta)=\frac{1}{Z_n}\exp\bigl(-h\Phi(\theta)\bigr) \mu_n(d\theta).
\end{equation}
Note, then, that if $Nh=1$ it follows that $\mu_N=\mu.$ Each step
$\mu_n \mapsto \mu_{n+1}$ may be approximated by a particle-based filtering algorithm, leading to a variety of algorithms used in practice, involving a fixed finite number of steps $N$. Furthermore, continuous-time limits of this methodology may also be derived
by taking $N \to \infty$ and $h \to 0$ with $Nh=1$, giving insight into the algorithms; see \cite{reich2011dynamical, ding2020ensemble}. 

On the other hand,
if $h=1$ is fixed and the measures $\mu_n$ are studied in the limit   $n \to \infty$, they
will tend to concentrate on minimizers of $\Phi$, restricted to the
support of $\mu_0$, as the following identity shows:
\begin{equation}
    \label{eq:iterate}
\mu_n(d\theta)=\frac{1}{\bigl(\Pi_{\ell=0}^{n-1}Z_{\ell}\bigr)}\exp\bigl(-n\Phi(\theta)\bigr)\mu_0(d\theta).
\end{equation}
This corresponds to an infinite time approach.

The finite time approach was developed for probabilistic problems;
the infinite time approach is focused  on optimization.
This paper will build on the latter, optimization, approach to the problem. However, we
note that, other than restriction of $\mu_n$ to the support of $\mu_0$, regularization is lost in this approach since it focuses on minimizing $\Phi(\cdot)$ and not $\Phi_R(\cdot)$.
To introduce regularization we consider the iteration
\begin{equation}
\label{eq:post-SMC2}
    \mu_{n+1}(d\theta)=\frac{1}{Z_n}\exp\bigl(-\Phi(\theta)\bigr)P_n \mu_n(d\theta).
\end{equation}
To address the issue of regularization, we will choose $P_n$ to be 
the Markov kernel
associated with a first-order autoregressive~(AR1) process as defined
by (\ref{eq:std}a); it is thus independent of $n: P_n\equiv P.$ 
The resulting dynamic on measures $\mu_n$ defined by \eqref{eq:post-SMC2} corresponds to the filtering
distribution for $\theta_n|Y_n$ defined by the stochastic
dynamical system \eqref{eq:std}. We note 
that within SMC $P_n$ is also introduced in a similar fashion in \eqref{eq:post-SMC}, but in that context it is chosen to be a $\mu_n-$invariant Markov kernel so that $P_n\mu_n=\mu_n$, typically
from MCMC; in this setting $P_n$ is indeed $n-$dependent.
Note that $\mu_n$ is not invariant with respect to $P$ with the AR1 choice we make: thus the introduction of $P_n$ in our setting differs
from its use in  SMC; this is because we are solving an optimization
problem via iteration over $n$, and not the sampling problem which
morphs the prior at time $n=0$ into the posterior at 
time $n=N.$ The specific choice of $P_n$ made in our work, namely
the Markov kernel $P$ defined by an AR1 process, is made in order to regularize the iterative optimization approach to inversion encapsulated in \eqref{eq:iterate}. 
Once we apply particle methods, the presence of
$P$ plays the role of avoiding ensemble collapse~\cite{schillings2017analysis,garbuno2020interacting,chada2020tikhonov}. We also note that, in contrast to SMC, the initial measure $\mu_0$ in \eqref{eq:post-SMC2} does not need to be 
the prior distribution -- it may be chosen arbitrarily, although
a natural choice is the stationary measure for the AR1 process.

In the case where $\G$ is linear, \eqref{eq:post-SMC2} delivers
a sequence of measures, which are defined through a Kalman filter.
Our analysis of the underlying filtering problem in Subsection
\ref{ssec:lin}, which considers the linear
Gaussian setting, thus constitutes an analysis of the Kalman filter
for a specific state-space model with a specific choice of
data. In order to deal with
a range of cases, including exponential convergence, algebraic
convergence and divergence of the mean/covariances of the
filter, we introduce an explicit unified analysis
of the Kalman filter in our setting. 
We note, however, that this is a well-trodden
field and that variants on some of our results can be obtained
from the existing literature \cite{lancaster1995algebraic,bougerol1993kalman}.

The method we introduce and study in this paper arises from the application of ideas 
from Kalman filtering to the problem of approximating the distribution 
of $\theta_n|Y_n$. The Kalman filter itself applies to the case
 of linear $\G$ \cite{kalman1960new,sorenson1985kalman}. When $\G$ is nonlinear the methods
 can be generalized by use of the extended 
 Kalman filter (ExKF) \cite{jazwinski2007stochastic} which is based on 
 linearization and application of Kalman methodology.
 However this method suffers from two drawbacks which hamper
 its application in many large-scale applications: (a) it
 requires a derivative of the forward map $\G(\cdot)$;  and (b) the approach scales poorly
 to high dimensional parameter spaces where $N_\theta \gg 1$, because of the
 need to sequentially update covariances in $\R^{N_\theta \times N_\theta}.$
Thus, despite an early realization that Kalman-based methods could
be useful for large-scale filtering problems arising in the
geosciences \cite{ghil1981applications},
the methods did not become practical in this context until the work
of Evensen~\cite{evensen1994sequential}. This revolutionary paper
introduced the ensemble Kalman filter (EnKF) the essence of
which is to avoid the linearization of the dynamics and sequential
updating of the covariance, and instead
use a low-rank approximation of the covariance found by maintaining
an ensemble of estimates for $\theta_n|Y_n$ at every step $n.$ These
ensemble Kalman methods have been widely adopted in the geosciences,
not only because they are effective for high dimensional parameter spaces,
but also because they are derivative-free, requiring only $\G$ as a black box.
Their use in the solution of inverse problems via iterative
methods was pioneered in subsurface inversion \cite{chen2012ensemble,emerick2013investigation} where the perspective of
fixing $h \ll 1$ and iterating until $n=N=1/h$ was used, so that $\mu_N$ is
viewed as an approximation of the posterior, provided $\mu_0$ is chosen
as the prior. These papers thus view the ensemble methodology as a way of sampling from
the posterior and have elements in common with SMC; this idea is also implicit in the
paper \cite{reich2011dynamical}, which is focused on data assimilation, and addresses
the solution of a Bayesian inverse problem each time new data is received.

In \cite{iglesias2013ensemble} the Kalman methodology for inversion was
revisited from the optimization perspective, based on fixing $h=1$
and iterating in $n$, leading to an algorithm we will refer to
as ensemble Kalman inversion (EKI). 
The paper \cite{iglesias2016regularizing} introduced a novel approach to regularizing the iterative method, by drawing an analogy
with the Levenberg-Marquardt algorithm (LMA) \cite{hanke1997regularizing}; see also \cite{iglesias2020adaptive}.
Subsequent variants on the iterative optimization approach demonstrate how to 
introduce Tikhonov regularization into the EKI algorithm \cite{chada2020tikhonov} and
the paper \cite{garbuno2020interacting} shows
that adding noise to the iteration can lead to approximate Bayesian inversion,
a method we will refer to as ensemble Kalman sampling (EKS) and which is further
analyzed in \cite{garbuno2020affine,nusken2019note}.  The EKS
provides a different approach to the problem of Bayesian inversion 
from the ones pioneered in \cite{chen2012ensemble,emerick2013investigation}
since it does not require starting with draws from the prior $\mu_0$, but instead
relies on ergodicity and iteration to large $n$; the methods in
\cite{chen2012ensemble,emerick2013investigation} must be started with draws from the prior $\mu_0$ and iterated for precisely $n=1/h$ steps, and are hence more rigid in their
requirements.
Since the ensemble methods do not, in general, accurately approximate the true posterior distribution~\cite{law2012evaluating,ernst2015analysis} outside Gaussian scenarios, the derivative-free optimization perspective is arguably a more
natural avenue within which to analyze ensemble inversion.
However recent work demonstrates how a derivative-free multiscale stochastic sampling
method can usefully take the output of EKS as a preconditioner for a method which
provably approximates the true posterior distribution \cite{pavliotis21derivative};
in that context, the EKS is central to making the method efficient.
Furthermore, in recent interesting work, it has been shown
how to reweight ensemble Kalman methods to recover statistical
consistency in the non-Gaussian setting \cite{ding2020ensemble};
however computation of the weights requires gradients of $\G$ and
hence is not practical for many of the problems where ensemble methods
are most useful.

Within the control theory literature, and parallel to the development of
the ensemble Kalman filter, the unscented Kalman filter (UKF) was introduced
\cite{julier1995new,wan2000unscented}. Like the ensemble Kalman methods, this method also 
sidesteps the need
to sequentially update the  derivative of the forward model as part of
the covariance update; but, in the primary difference from ensemble Kalman methods, particles (sigma points) are chosen deterministically, and a  quadrature rule is applied 
within a Gaussian approximation of the filter. 
This paper is to establish a framework for the development of
unscented Kalman methods for inverse problems, based on \eqref{eq:std}:
we formalize and demonstrate the power of unscented Kalman inversion (UKI)
techniques. We also formalize extended Kalman inversion (ExKI) as a general purpose methodology for parameter learning and derive ExKI, UKI, and UKI as different approximations of a conceptual Gaussian methodology for the
(in general non-Gaussian) filtering problem defined by \eqref{eq:std}.

Inverse and parameter estimation problems are ubiquitous in engineering and scientific applications. Applications that motivate this work include global climate model calibration~\cite{sen2013global,schneider2017earth,dunbar2020calibration}, material constitutive relation calibration~\cite{huang2020learning,xu2020learning,avery2020computationally}, seismic inversion in geophysics~\cite{russell1988introduction,bunks1995multiscale}, and medical tomography~\cite{toger2020blood,trigo2004electrical}.
These problems are generally highly nonlinear, may feature multiple scales, and may
include chaotic and turbulent phenomena. Moreover, the observational data is often noisy and the inverse problem may be ill-posed. We note, also, that a number
of inverse problems of interest may involve a moderate number of unknown
parameters $N_{\theta}$, yet may involve the solution of a very expensive
forward model $\G$ depending on those parameters; furthermore, $\G$ may not be differentiable with respect to the parameters, or may be complex
to differentiate as it is given as a black box.

In the nonlinear setting of state estimation, there are three primary types of Kalman filters~\cite{simon2006optimal, auger2013industrial,fang2018nonlinear}: the extended Kalman filter~(ExKF), the unscented Kalman filter~(UKF), and the ensemble Kalman filter~(EnKF).
The use of Kalman based methodology as a non-intrusive iterative method for 
parameter estimation originates in the papers~\cite{singhal1989training, puskorius1991decoupled} which were based on the ExKF, hence requiring
derivative $d\G$, and its adjoint, to propagate covariances; the use
of derivative-free ensemble methods was then developed systematically in the
papers \cite{chen2012ensemble,emerick2013investigation}, in the
SMC context, followed by the iterate for optimization EKI approach \cite{iglesias2013ensemble}.
Derivative-free ensemble inversion and parameter estimation
are particularly suitable for complex multiphysics problems requiring coupling of different solvers, such as fluid-structure interaction~\cite{huang2018simulation,huang2019high,huang2020modeling,huang2020high} and general circulation 
models \cite{adcroft2019gfdl} and methods containing discontinuities
such as the immersed/embedded boundary method~\cite{peskin1977numerical,berger2012progress,huang2018family,huang2020embedded} and 
adaptive mesh refinement~\cite{berger1989local,borker2019mesh}. 
Furthermore, derivative-free ensemble inversion and parameter estimation has
been demonstrated to be effective in the context of forward models defined
by chaotic dynamical systems~\cite{cleary2020calibrate} 
where adjoint-based methods fail to deliver meaningful sensitivities~\cite{lea2000sensitivity,wang2014least}.
These wide-ranging potential applications form motivation for developing
other derivative-free Kalman based inversion and parameter estimation techniques,
and in particular, the unscented Kalman methods developed here.

There is already some work in which unscented Kalman methods are used
for parameter inversion.
Extended, ensemble and unscented Kalman inversions have been applied to train neural networks~\cite{singhal1989training, puskorius1991decoupled, wan2000unscented, kovachki2019ensemble} and EKI has been applied in the oil industry~\cite{oliver2008inverse, chen2012ensemble, emerick2013investigation}. Dual and joint Kalman filters~\cite{wan1997neural, wan2000unscented} have been 
designed to simultaneously estimate the unknown states and the parameters~\cite{wan1997neural,parlos2001algorithmic,wan2000unscented,gove2006application,albers2017personalized} from noisy sequential observations.
 However, whilst the EKI has been systematically developed and analyzed as
 a general purpose methodology for the solution of inverse and parameter
 estimation problems, the same is not the case for UKI.

Continuous-time limits and gradient flow structure of the EKI have been introduced and studied in \cite{reich2011dynamical,bergemann2012ensemble,schillings2017analysis,schillings2018convergence,ding2019ensemble,ding2021ensemble,ding2021ensembleb}. This work led to the development of variants on the EKI, such as the  Tikhonov-EKI (TEKI)~\cite{chada2020tikhonov} and the EKS~\cite{garbuno2020interacting}. We will develop study of continuous-time limits
for the UKI, and variants including an unscented Kalman sampler (UKS), in this paper. There are interesting links to
the Levenberg–Marquardt Algorithm (LMA)~\cite{bell1993iterated,hanke1997regularizing}, as introduced in~\cite{iglesias2016regularizing} and developed further in~\cite{chada2019convergence,chada2020iterative,iglesias2020adaptive}. We  will further refine the idea, which provides insights into understanding and improving the nonlinear Kalman inversion
methodology as introduced here.
 
 Finally, we mention that there are other derivative-free optimization techniques which are
 based on interacting particle systems, but are not Kalman based. Rather these methods are
 based on consensus-forming mean-field models, and their particle approximations,
 leading to consensus-based optimization \cite{carrillo2018analytical} and consensus-based
 sampling \cite{carrillo2021}. The paper \cite{pavliotis21derivative} also provides an
 alternative derivative-free approach to optimization and sampling for inverse problems,
 using ideas from multiscale dynamical systems.

\section{Nonlinear Kalman Inversion Algorithms}
\label{sec:A}

Recall that the basic approach to inverse problems that we adopt in this paper
is to pair the parameter-to-data relationship encoded in \eqref{eq:KI}
with a stochastic dynamical system for the parameter, resulting in
\eqref{eq:std}. We then employ techniques
from filtering to approximate the distribution $\mu_n$ of $\theta_n|Y_n$. 
A useful way to think of updating $\mu_n$ is through the prediction and
analysis steps~\cite{reich2015probabilistic,law2015data}: $\mu_n 
\mapsto \hat{\mu}_{n+1}$, and then 
$\hat{\mu}_{n+1} \mapsto \mu_{n+1}$, where $\hat{\mu}_{n+1}$ is the distribution
of $\theta_{n+1}|Y_n$. In Subsection \ref{ssec:gau} we first introduce
a Gaussian approximation of the analysis step, leading to an algorithm which
maps the space of Gaussian measures into itself at each step of the iteration; 
it is not implementable in general, but it is a useful conceptual algorithm. Subsection \ref{ssec:exki}
shows how this algorithm can be made practical, for low to moderate
dimension $N_{\theta}$ and assuming that $d\G$ is available, by means of the
ExKF, a form of linearization of the conceptual algorithm; we refer to this as ExKI. 
In Subsection \ref{ssec:uki} we show how the UKI algorithm
may be derived by applying a quadrature rule to
evaluate certain integrals appearing in the conceptual Gaussian approximation.
Subsection \ref{ssec:eki} connects the conceptual algorithm with the EKI, an
approach in
which ensemble approximation of the integrals is used.

\subsection{Gaussian Approximation}
\label{ssec:gau}

This conceptual algorithm maps Gaussians into Gaussians, and henceforth it is referred to as the Gaussian Approximation Algorithm (GAA).
Assume that $\mu_n \approx \N(m_n,C_n)$. The GAA is a mapping
from $(m_n,C_n)$ into $(m_{n+1},C_{n+1})$ which reduces to the
Kalman filter in the linear setting. The
algorithm proceeds by determining the joint distribution
of $\theta_{n+1}, y_{n+1}|Y_n$, assuming that $\theta_{n}|Y_n$ is Gaussian $\N(m_{n},C_{n})$. We then project 
\footnote{We refer to this as ``projection'' because it corresponds
to finding the closest Gaussian $p$ to the joint distribution of
$\theta_{n+1}, y_{n+1}|Y_n$ with respect to variation in the
second argument of the (nonsymmetric) Kullback-Leibler divergence
\cite{sanz2018inverse}[Theorem 4.5].}
this joint distribution onto a Gaussian by computing
its mean and covariance. And finally, we compute the conditional distribution of this joint Gaussian on observed $y_{n+1}$ to obtain
a Gaussian approximation  $\N(m_{n+1},C_{n+1})$ to $\mu_{n+1},$
the distribution of $\theta_{n+1}|Y_{n+1}.$

The projection of the joint distribution of  $\{\theta_{n+1}, y_{n+1}\}|Y_{n}$ onto a Gaussian distribution has the form
\begin{equation}
\label{eq:KF_joint}
     \N\Bigl(
    \begin{bmatrix}
    \pmean_{n+1}\\
    \py_{n+1}
    \end{bmatrix}, 
    \begin{bmatrix}
   \pCov_{n+1} & \pCov_{n+1}^{\theta y}\\
    {{\pCov_{n+1}}^{\theta y}}{}^{T} & \pCov_{n+1}^{yy}
    \end{bmatrix}
    \Bigr);
\end{equation}
we now define all the components of the mean and covariance.
Note that, under (\ref{eq:std}a),
$\hat{\mu}_{n+1}$ is also Gaussian if $\mu_n$ is Gaussian.
The use of (\ref{eq:std}a) shows that
\begin{equation}
\label{eq:KF_pred_mean}
\begin{split}
    \pmean_{n+1} &= \E[\theta_{n+1}|Y_n] =  \alpha\mean_n + (1 - \alpha)r_0,\\
    \pCov_{n+1} &= \mathrm{Cov}[\theta_{n+1}|Y_n] =  \alpha^2 \Cov_{n} + \Sigma_{\omega}.
\end{split}
\end{equation}
Then, with $\E$ denoting expectation with respect to 
$\theta_{n+1}|Y_n \sim \N( \pmean_{n+1},\pCov_{n+1})$, we have
\begin{equation}
\label{eq:KF_joint2}
\begin{split}
    \py_{n+1} =     & \E[\G(\theta_{n+1})|Y_n], \\
     \pCov_{n+1}^{\theta y} =     &  \mathrm{Cov}[\theta_{n+1}, \G(\theta_{n+1})|Y_n],\\
    \pCov_{n+1}^{yy} = &  \mathrm{Cov}[\G(\theta_{n+1})|Y_n] + \Sigma_{\nu}.
\end{split}
\end{equation}
Computing the conditional distribution of the joint Gaussian in \eqref{eq:KF_joint} to find $\theta_{n+1}|\{Y_n,y_{n+1}\}=\theta_{n+1}|Y_{n+1}$ gives the following
expressions for the mean $\mean_{n+1}$ and covariance $\Cov_{n+1}$ of the
approximation to $\mu_{n+1}:$
\begin{equation}
\label{eq:KF_analysis}
    \begin{split}
        \mean_{n+1} &= \pmean_{n+1} + \pCov_{n+1}^{\theta y} (\pCov_{n+1}^{yy})^{-1} (y_{n+1} - \py_{n+1}),\\
         \Cov_{n+1} &= \pCov_{n+1} - \pCov_{n+1}^{\theta y}(\pCov_{n+1}^{yy})^{-1} {\pCov_{n+1}^{\theta y}}{}^{T}.
    \end{split}
\end{equation}%

Equations \eqref{eq:KF_pred_mean}, \eqref{eq:KF_joint2}
and \eqref{eq:KF_analysis}  define the GAA. 
As  a  method  for  solving  the  inverse  problem  \eqref{eq:KI},  the  GAA  is  implemented  by  assuming  all observations $\{y_n\}$ are  identical  to $y$ and  iterating  in $n.$ With this assumption, we may write the algorithm 
as 
\begin{equation}
\label{eq:themap}
(m_{n+1},C_{n+1})=F(m_n,C_n;\G,r_0,\Sigma_{\omega}),    
\end{equation}
noting that the mapping is dependent on $\G$ and on the
mean and covariance of the assumed auto-regressive dynamics
for $\{\theta_n\}$.
\footnote{
$F$ also depends on $\alpha$ and
$\Sigma_{\nu}$ but we suppress this dependence for economy
of notation; the highlighted dependence is what is relevant
in Proposition \ref{prop:affine-invariance}.
}

In the setting where $\G$ is linear, the Gaussian ansatz used in the derivation of the
conceptual algorithm is exact, the integrals appearing in \eqref{eq:KF_joint2} have closed form, and the algorithm
reduces to the Kalman filter applied to \eqref{eq:std}, with a particular assumption on the data stream $\{y_n\}$.
In Subsection \ref{ssec:lin} we will show, again in the setting
where $\G$ is linear,
that the mean of this iteration converges to a minimizer of
$\Phi_R$ given by \eqref{eq:KI2}, in which the prior covariance of the regularization  $\Sigma_0$ is defined by solution of a linear equation depending on the choices of $\alpha$, $\Sigma_\omega$, and $\Sigma_{\nu}$, as well as on $\G.$

In the nonlinear setting, to make an implementable algorithm from the GAA encapsulated in~\cref{eq:KF_pred_mean,eq:KF_joint2,eq:KF_analysis}, 
it is necessary to approximate the integrals appearing in \eqref{eq:KF_joint2}.
When extended, unscented and ensemble Kalman filters are applied, respectively, to make such approximation, we  obtain the ExKI, UKI, and EKI algorithms. 
The extended, unscented, and ensemble approaches to this are detailed in the
following three subsections. Underlying all of them
is the following property of the GAA encapsulated in
Proposition \ref{prop:affine-invariance}.

We recall the idea of affine invariance, introduced
for MCMC methods in \cite{goodman2010ensemble},  motivated by the attribution of the empirical success
of the Nelder-Mead algorithm \cite{nelder1965simplex}
for optimization to a similar property; further development of the method in the context of sampling
algorithms may be found in \cite{leimkuhler2018ensemble,garbuno2020affine}.
In words an
iteration is affine invariant if an invertible linear transformation of the variable being iterated
makes no difference to the algorithm and hence to the convergence properties of
the algorithm; this has the desirable consequence that
performance of the method is independent of the aspect ratio in highly anisotropic objective functions.

Consider the invertible mapping from $x \in \R^{N_{\theta}}$ to ${}^{*}x \in \R^{N_\theta}$ defined
by ${}^{*}x = Ax + b$. Then define
${}^{*}\G(\theta) =  \G\big(A^{-1}(\theta - b)\big)$,
${}^{*}r_0 = Ar_0 + b$ and  ${}^{*}\Sigma_{\omega} = A\Sigma_{\omega} A^T.$

\begin{proposition}
\label{prop:affine-invariance}
Define, for all $n \in \Z^{0+}$,
\begin{equation*}
{}^{*}\mean_n =  A\mean_n + b \qquad {}^{*}\Cov_{n} = A\Cov_{n} A^T.
\end{equation*}
Then
\begin{equation}
 ({}^{*}\mean_{n+1},{}^{*}\Cov_{n+1}) = 
F({}^{*}\mean_n,{}^{*}\Cov_{n};{}^{*}\G,{}^{*}r_0,{}^{*}\Sigma_{\omega}).
\end{equation}
\end{proposition}
\begin{proof}
The proof is in \ref{sec:appendix}.
\end{proof}

The key observation of the previous theorem is that the same
map $F$ applies in the new coordinates.
This establishes the property of affine invariance,
noting that only $\G,r_0,\Sigma_{\omega}$ need to be transformed as the affine map applies only on the
signal space for $\{\theta_n\}$ and not the observation space for $\{y_n\}.$

\subsection{Extended Kalman Inversion}
\label{ssec:exki}

Consider the GAA defined by~\cref{eq:KF_pred_mean,eq:KF_joint2,eq:KF_analysis}.
The ExKI algorithm follows from invoking the approximations
\begin{equation}
\label{eq:ExKI_Taylor}
    \G(\theta_{n+1}) \approx \G(\pmean_{n+1}) + d\G(\pmean_{n+1}) (\theta_{n+1} - \pmean_{n+1})
\end{equation}
in the analysis updates for the mean and covariance respectively. In
particular both the mean and the covariances in \eqref{eq:KF_joint2} can be evaluated in closed
form with the approximation (\ref{eq:ExKI_Taylor}). The 
approximations are valid if the fluctuations around the mean state are small,
say of $\mathcal{O}(\epsilon)\ll 1$, and all the covariances are $\mathcal{O}(\epsilon^2).$
This results in the following algorithm:

\begin{itemize}
    \item Prediction step :
    \begin{equation}
\label{eq:KF_pred_mean_ex}
\begin{split}
    \pmean_{n+1} = & \alpha\mean_n + (1 - \alpha)r_0,\\
    \pCov_{n+1} = & \alpha^2 \Cov_{n} + \Sigma_{\omega}.
\end{split}
\end{equation}
    \item Analysis step :
    \begin{equation}
    \label{eq:ExKI-1.2}
    \begin{split}
    &\py_{n+1} =     \G(\pmean_{n+1}), \\
    &\pCov_{n+1}^{\theta y} = \pCov_{n+1} d\mathcal{G}(\pmean_{n+1})^T,\\
    &\pCov_{n+1}^{yy} = d\mathcal{G}(\pmean_{n+1})\pCov_{n+1}d\mathcal{G}(\pmean_{n+1})^T + \Sigma_{\nu},\\
    &\mean_{n+1} = \pmean_{n+1} + \pCov_{n+1}^{\theta y}(\pCov_{n+1}^{yy})^{-1}\bigl(y - \py_{n+1}\bigr),\\
    &\Cov_{n+1} = \pCov_{n+1} - \pCov^{\theta y}_{n+1}(\pCov^{yy}_{n+1})^{-1}{\pCov^{\theta y}_{n+1}}{}^{T}.\\
    \end{split}
    \end{equation}
    \end{itemize}
    
This is a map of the form \eqref{eq:themap}, but with a different definition of $F$, now depending on $d\G$ as well as $\G$.

\subsection{Unscented Kalman Inversion}
\label{ssec:uki}

Like the ExKI, the UKI also approximates the GAA;
but it approximates the integrals appearing in Equations \eqref{eq:KF_joint2} by means of
deterministic quadrature rules which are exact when evaluating means and
covariances of variables defined as linear transformations of the
random variable in question. Both the ExKI and the UKI recover
the Kalman filter when $\G$ is linear.
We need the definition of the unscented transform~\cite{julier1995new,wan2000unscented}:

\begin{definition}[Modified Unscented Transform]
\label{def:unscented_tranform}
Consider Gaussian random variable $\theta \sim \N(\mean, \Cov) \in \R^{N_{\theta}}$. Define the $2N_{\theta}+1$ symmetric sigma points 
$\{\theta_j\}_{j=0}^{2N_\theta +1}$ by
\begin{equation}
\begin{split}
    \theta^0 &= \mean, \\
    \theta^j &= \mean + c_j [\sqrt{\Cov}]_j \quad (1\leq j\leq N_\theta),\\
    \theta^{j+N_\theta} &= \mean - c_j [\sqrt{\Cov}]_j\quad (1\leq j\leq N_\theta),
\end{split}
\end{equation}
where $[\sqrt{\Cov}]_j$ is the $j$th column of the Cholesky factor of $\Cov$. Let $\G_i, i=1,2$ denote any pair of real vector-valued
functions on $\R^{N_\theta}$. Then the quadrature rule approximating the mean and covariance of the transformed variables $\G_1(\theta)$ and $\G_2(\theta)$ is given by  
\begin{equation}
\label{eq:ukf}
    \E[\G_i(\theta)] \approx \G_i(\theta^0)\qquad 
    \mathrm{Cov}[\G_1(\theta),\G_2(\theta)]  \approx \sum_{j=1}^{2N_{\theta}} W_j^{c} (\G_1(\theta^j) - \E\G_1(\theta))(\G_2(\theta^j) - \E\G_2(\theta))^T. 
\end{equation}
Here these constant weights are, for any $a \in \R$, 
\begin{align*}
    &c_j = a\sqrt{N_\theta}~(j=1,\cdots,N_{\theta}) \quad  W_j^{c} =\frac{1}{2a^2 N_\theta}~(j=1,\cdots,2N_{\theta}).\\
\end{align*}
\end{definition}

\begin{lemma}
\label{lem:uki}
Let $\G_i, i=1,2$ denote any pair of real vector-valued
functions on $\R^{N_\theta}$.
If $\theta \sim \N(\mean, \Cov)$ then 
\begin{align*}
&\E[\G_i(\theta)] = \G_i(m) + \bigO(\|C\|),\\ 
&\mathrm{Cov}[\G_1(\theta),\G_2(\theta)] = \sum_{j=1}^{2N_{\theta}} W_j^{c} (\G_1(\theta^j) - \E\G_1(\theta))(\G_2(\theta^j) - \E\G_2(\theta))^T + \bigO(\|C\|^2);
\end{align*}
thus the modified unscented transform is first and second order accurate in approximating
means and covariances of  $\G_1(\theta)$ and $\G_2(\theta)$ with respect to small $\|C\|.$ 
Furthermore, if $\G_1$ and $\G_2$ are linear then the modified unscented transform is
exact for these quantities.
\end{lemma}

\begin{proof}
The proof is in \ref{sec:appendix}.
\end{proof}

\begin{remark}
The first and second order high order error terms, appearing
in the expressions for the mean and covariance respectively,
depend on derivatives of $\G_i$ at $m$ and hence, through
these derivatives and through $C$, on the 
parameter dimension $N_{\theta}$.
The original unscented transform leads to second order  accuracy in the mean as well as covariance~\cite{julier2000new}.  
The modification we employ here replaces the original second order approximation of the $\E[\G_i(\theta)]$ with its first order counterpart. We do this to
avoid negative weights; it also has ramifications for the optimization
process which we discuss in \cref{rem:UKI}.
In this paper, the hyper-parameter is chosen to be
$\displaystyle a=\min\{\sqrt{\frac{4}{N_\theta}},  1\}.$ We note that  the papers~\cite{julier2000new,wan2000unscented,fang2018nonlinear}, suggest 
using a small positive value of $a$. We find in the numerical examples considered in this paper that our proposed choice of $a$ outperforms the choice $a=\min\{\sqrt{\frac{4}{N_\theta}},  0.01\}$), which builds
in the idea of using a small positive value of $a$.
\end{remark}

Consider the algorithm defined by
\cref{eq:KF_pred_mean,eq:KF_joint2,eq:KF_analysis}.
By utilizing the aforementioned quadrature rule, we obtain the following UKI algorithm:
\begin{itemize}
\item Prediction step :
\begin{equation}
\label{eq:KF_pred_mean_u}
\begin{split}
    \pmean_{n+1} = & \alpha\mean_n +(1 - \alpha)r_0,\\
    \pCov_{n+1} = & \alpha^2 \Cov_{n} + \Sigma_{\omega}.
\end{split}
\end{equation}
    \item Generate sigma points :
    \begin{equation}
    \label{eq:uki_sigma-points}
    \begin{split}
    &\pp_{n+1}^0 = \pmean_{n+1}, \\
    &\pp_{n+1}^j = \pmean_{n+1} + c_j [\sqrt{\pCov_{n+1}}]_j \quad (1\leq j\leq N_\theta),\\ 
    &\pp_{n+1}^{j+N_\theta} = \pmean_{n+1} - c_j [\sqrt{\pCov_{n+1}}]_j\quad (1\leq j\leq N_\theta).
    \end{split}
    \end{equation}
\item Analysis step :
   \begin{equation}
   \label{eq:UKI-analysis}
   \begin{split}
        &\ppy^j_{n+1} = \mathcal{G}(\pp^j_{n+1}) \qquad \py_{n+1} = \ppy^0_{n+1},\\
         &\pCov^{\theta y}_{n+1} = \sum_{j=1}^{2N_\theta}W_j^{c}
        (\pp^j_{n+1} - \pmean_{n+1} )(\ppy^j_{n+1} - \py_{n+1})^T, \\
        &\pCov^{yy}_{n+1} = \sum_{j=1}^{2N_\theta}W_j^{c}
        (\ppy^j_{n+1} - \py_{n+1} )(\ppy^j_{n+1} - \py_{n+1})^T + \Sigma_{\nu},\\
        &\mean_{n+1} = \pmean_{n+1} + \pCov^{\theta y}_{n+1}(\pCov^{yy}_{n+1})^{-1}(y - \py_{n+1}),\\
        &\Cov_{n+1} = \pCov_{n+1} - \pCov^{\theta y}_{n+1}(\pCov^{yy}_{n+1})^{-1}{\pCov^{\theta y}_{n+1}}{}^{T}.\\
    \end{split}
    \end{equation}
\end{itemize}

This is again a map of the form \eqref{eq:themap}, but with a different definition of $F$; unlike the ExKF there is no dependence
on $d\G$, only on $\G$.

\subsection{Ensemble Kalman Inversion}
\label{ssec:eki}

This method differs fundamentally from the ExKI and UKI
in that it does not map the mean and covariance. Rather it works
with a set of particles whose dynamics at each step is predicted
using (\ref{eq:std}a) and then used to compute empirical
approximations of covariances. These in turn are used in the
analysis step. The entire algorithm maps the collection
$\{\theta_{n}^{j}\}_{j=1}^J$ into $\{\theta_{n+1}^{j}\}_{j=1}^J$.
However, in the large $J$ limit the mean and covariance updates
match those of the GAA.

Consider the algorithm defined by
\cref{eq:KF_pred_mean,eq:KF_joint2,eq:KF_analysis}.
The EKI approach to making this implementable is to work with an ensemble
of parameter estimates and approximate the covariances $\pCov_{n+1}^{\theta y}$
and $\pCov_{n+1}^{yy}$ empirically:
\begin{itemize}
  \item Prediction step :
  \begin{equation}
      \label{eq:EKI-prediction}
  \begin{split}
                      \pp_{n+1}^{j} &= \alpha\theta_{n}^{j}+ (1-\alpha)r_0  + \omega_{n+1}^{j},\\ 
                      \pmean_{n+1} &= \frac{1}{J}\sum_{j=1}^{J}\pp_{n+1}^{j}.
                   \end{split}
                   \end{equation}
  \item  Analysis step : 
  \begin{equation}
  \label{eq:EKI-analysis}
  \begin{split}
         &\ppy_{n+1}^{j} = \mathcal{G}(\pp_{n+1}^{j})  \qquad \py_{n+1} = \frac{1}{J}\sum_{j=1}^{J}\ppy_{n+1}^{j},\\
        &\pCov_{n+1}^{\theta y} = \frac{1}{J-1}\sum_{j=1}^{J}(\pp_{n+1}^{j} - \pmean_{n+1})(\ppy_{n+1}^{j} - \py_{n+1})^T,  \\
        &\pCov_{n+1}^{yy} = \frac{1}{J-1}\sum_{j=1}^{J}(\ppy_{n+1}^{j} - \py_{n+1})(\ppy_{n+1}^{j} - \py_{n+1})^T +\Sigma_{\nu}, \\
        &\theta_{n+1}^{j} = \pp_{n+1}^{j} + \pCov_{n+1}^{\theta y}\left(\pCov_{n+1}^{yy}\right)^{-1}(y - \ppy_{n+1}^{j} - \nu_{n+1}^{j}),\\
        &\mean_{n+1} = \frac{1}{J} \sum_{j=1}^{J} \theta_{n+1}^{j}.\\
  \end{split}
\end{equation}
\end{itemize}
Here the superscript $j = 1,\cdots,\ J$ is the ensemble particle index, $\omega_{n+1}^{j} \sim \N(0, \Sigma_{\omega})$ and  $\nu_{n+1}^{j} \sim \N(0, \Sigma_{\nu})$ are independent and identically distributed random variables with respect to both $j$ and $n$.

\begin{remark}
In~\cite{iglesias2013ensemble}, where the iterative EKI
was introduced, a slightly different stochastic dynamical formulation is used, extending the parameter space to include the image of the parameters under $\G$ and
then making a linear observation operator on the extended space. The resulting method
reduces to our setting with $\alpha=1$, $\Sigma_{\omega}=0$, and $\Sigma_{\nu}=\Sigma_{\eta}$ in the preceding algorithm.
In the next section we will demonstrate theoretically that choosing $\alpha \in (0,1)$ and
$\Sigma_{\omega} \succ 0$ is beneficial and hence that the version of EKI proposed in
this paper is superior to that in~\cite{iglesias2013ensemble}.
\end{remark}

\section{Theoretical Insights}
\label{sec:T}

Recall that we view the GAA as an underlying conceptual algorithm
which gives insight into the ExKI, UKI, and EKI algorithms. The ExKI is 
itself an approximation of the GAA, found by linearizing $\G$ around the predictive mean and the UKI and EKI algorithms are approximations
of the resulting ExKI. Thus study of the GAA and ExKI gives insights into
the UKI and EKI algorithms. This section is devoted to such studies.
In Subsection \ref{ssec:lin} we consider behaviour of the GAA in the linear
setting. In Subsection \ref{ssec:LMA},
we show that the ExKI may be viewed as a generalization of the LMA for optimization. Subsection \ref{ssec:UA}
exhibits an averaging property induced by the unscented approximation,
indicating how this may help in solving problems with rough energy
landscapes. And in Subsection \ref{ssec:cts} we study a continuous-time limit of the GAA, which may itself be approximated to obtain continuous-time limits of the ExKI, UKI, and EKI algorithms; this provides insight into
the discrete algorithms as implemented in practice.

\subsection{The Linear Setting}
\label{ssec:lin}

In the linear setting the stochastic dynamical system for state $\{\theta_n\}$
and observations $\{y_n\}$ is given by
\begin{subequations}
\label{eq:std99}
  \begin{align}
  &\textrm{evolution:}    &&\theta_{n+1} = \alpha\theta_{n} + (1-\alpha)r_0 +  \omega_{n+1}, &&\omega_{n+1} \sim \N(0,\Sigma_{\omega}),\\
  &\textrm{observation:}  &&y_{n+1} = G\theta_{n+1} + \nu_{n+1}, &&\nu_{n+1} \sim \N(0,\Sigma_{\nu}).
\end{align}
\end{subequations}
Thanks to the linearity, equations~(\ref{eq:KF_joint2}) reduce to 
\begin{align}
\label{eq:KF_joint2_linear}
    \py_{n+1} = G\mean_n, \quad \pCov_{n+1}^{\theta y} = \pCov_{n+1} G^T,\quad  \textrm{and} \quad \pCov_{n+1}^{yy} = G  \pCov_{n+1} G^T + \Sigma_{\nu}.
\end{align}
The update equations~(\ref{eq:KF_analysis}) become
\begin{equation}
\label{eq:KF_pred_mean2}
\begin{split}
    \pmean_{n+1} &= \alpha\mean_n + (1-\alpha)r_0,\\
    \pCov_{n+1} &=  \alpha^2 \Cov_{n} + \Sigma_{\omega},
\end{split}
\end{equation}
and
\begin{subequations}
\label{eq:Lin_KF_analysis}
  \begin{align}
        \mean_{n+1} &= \pmean_{n+1} + \pCov_{n+1} G^T (G  \pCov_{n+1} G^T + \Sigma_{\nu})^{-1} \Big(y - G\pmean_{n+1} \Big), \\
         \Cov_{n+1}&= \pCov_{n+1} - \pCov_{n+1} G^T(G  \pCov_{n+1} G^T + \Sigma_{\nu})^{-1} G \pCov_{n+1}. \label{eq:Lin_KF_analysis_cov}
\end{align}
\end{subequations}
We have the following theorem about the convergence of the 
GAA in the setting of the linear forward model:

\begin{theorem}
\label{th:lin_converge}
Assume that $\Sigma_{\omega}\succ 0$ and $\Sigma_{\nu}\succ 0.$ Consider the iteration
\eqref{eq:KF_pred_mean2}, \eqref{eq:Lin_KF_analysis} mapping
$(m_n,C_n)$ into $(m_{n+1},C_{n+1})$. Assume further 
that $\alpha \in (0,1)$ or that
$\alpha=1$ and $\text{Range}(G^T)=\R^{N_{\theta}}$.
Then the steady state equation of \cref{eq:Lin_KF_analysis_cov}
\begin{equation}
\label{eq:cov_steady}
\Cov_{\infty}^{-1} =  G^T\Sigma_{\nu}^{-1}G + (\alpha^2 \Cov_{\infty} + \Sigma_{\omega})^{-1}
\end{equation}
has a unique solution $\Cov_{\infty} \succ 0.$ 
The pair $(m_n,C_n)$
converges exponentially fast to  limit $(\mean_{\infty},\Cov_{\infty})$.
Furthermore the limiting mean $\mean_{\infty}$ is the minimizer
of the Tikhonov regularized least squares
functional $\Phi_R$ given by
\begin{equation}
\label{eq:KI299}
\Phi_R(\theta) := \frac{1}{2}\lVert\Sigma_{\nu}^{-\frac{1}{2}}(y - G\theta) \rVert^2 +
\frac{1 - \alpha}{2}\lVert \pCov_{\infty}^{-\frac{1}{2}}(\theta - r_0) \rVert^2,
\end{equation}
where 
\begin{equation}
\label{eq:add}
\pCov_{\infty} =\alpha^2 \Cov_{\infty} + \Sigma_{\omega}.
\end{equation}
\end{theorem}
\begin{proof}
The proof is in \ref{sec:appendix}.
\end{proof}

\begin{remark}
When $\alpha \in (0,1)$, the exponential convergence rates of the mean and covariance are independent of the condition number of $G^T\Sigma_{\nu}^{-1}G$. 
Furthermore, $\pCov_{\infty}$ is bounded above and below: 
$$\displaystyle \Sigma_{\omega} \preceq  \pCov_{\infty}  \preceq \frac{\Sigma_{\omega}}{1-\alpha^2},$$ since $\displaystyle 0 \preceq \Cov_{\infty} \preceq \alpha^2\Cov_{\infty} + \Sigma_{\omega}$. 
\end{remark}

\begin{remark}
\label{rem:1}
Despite
the clear parallels between \cref{eq:KI299} and Tikhonov regularization \cite{engl1996regularization}, there is an important
difference: the matrix $\pCov_{\infty}$ defining the implied
prior covariance in the regularization term
depends on the forward model. This may be seen by noting that
it is defined by \eqref{eq:add} in terms of the
steady state covariance $\Cov_{\infty}$ satisfying \eqref{eq:cov_steady}.
To get some insight into the implications of this, we consider the over-determined linear system in which $G^T\Sigma_{\eta}^{-1}G$ is invertible
and we may define
\begin{equation}
    \label{eq:cstar}
{C_{*}} = (G^T\Sigma_{\eta}^{-1}G)^{-1}.
\end{equation}
If we choose the artificial evolution and observation error covariances
\begin{subequations}
    \label{eq:cstar2}
    \begin{align}
    \Sigma_\nu &= 2 \Sigma_{\eta},\\
    \Sigma_{\omega} &= \bigl(2 - \alpha^{2}\bigl) \Cov_{*},
    \end{align}
    \end{subequations}
then straightforward calculation with  
\eqref{eq:cov_steady}, \eqref{eq:add} shows that
$$C_{\infty}=C_{*}, \quad \pCov_{\infty}=2\Cov_{*}.$$
From \eqref{eq:KI299} it follows that
\begin{equation}
\label{eq:cstar33}
    \Phi_R(\theta)= 
     \frac{1}{4} \left\lVert \Sigma_{\eta}^{-\frac{1}{2}}(y-G\theta) \right\rVert^2 +
    \frac{(1-\alpha)}{4} \left\lVert \Sigma_{\eta}^{-\frac{1}{2}}(Gr_0 - G \theta) \right\rVert^2.
    \end{equation}
This calculation clearly demonstrates the dependence of the second (regularization) term on the forward model and that choosing
$\alpha \in (0,1]$ allows different weights on the regularization term.
In contrast to Tikhonov regularization, the regularization term~(\ref{eq:cstar33}) scales similarly with respect to $G$
as does the data misfit, providing a regularization between 
the prior mean $r_0$ and an overfitted parameter  $\theta^* : y = G\theta^{*}$. Therefore, despite the differences from
standard Tikhonov regularization, the implied regularization
resulting from the proposed stochastic dynamical system
is both interpretable and controllable; in particular, the
single parameter $\alpha$ measures the balance between prior
and the overfitted solution.
\end{remark}

\begin{remark}
Theorem \ref{th:lin_converge} holds for any Kalman inversions that fulfill \cref{eq:KF_pred_mean,eq:KF_joint2_linear} exactly, which include ExKI, UKI, and these square root Kalman inversions~\cite{tippett2003ensemble,huang2022efficient}, but not the EKI.
\end{remark}

We contrast Theorem~\ref{th:lin_converge} with the behaviour of the filtering distribution
for the stochastic dynamical system used in the derivation of
the standard form of the EKI \cite{iglesias2013ensemble}, which
corresponds to the choices $\alpha=1$, $\Sigma_{\omega}=0$, and $\Sigma_{\nu}=\Sigma_{\eta}$.
To study this case we will assume that $C_0 \succ 0$ and define
\begin{subequations}
\label{eq:newvar}
\begin{align}
&C_n'=C_0^{-\frac12}C_nC_0^{-\frac12},\quad m_n'=C_0^{-\frac12}m_n,\quad G'=GC_0^{\frac12}, \quad S=(G')^T\Sigma_{\nu}^{-1}G'.
\end{align}
\end{subequations}
We note that the nullspace of $S$ is equal to the nullspace of $G'$,
and that the nullspace of $G'$ is found from the nullspace of $G$
by application of $C_0^{-\frac12}.$
Let $Q$ denote orthogonal projection onto the nullspace of $S$, 
and $P$ the orthogonal complement of $Q$. We then have the following
characterization of the filtering distribution for the stochastic
dynamical system underlying the form of the EKI introduced
in~\cite{iglesias2013ensemble}.

\begin{theorem}
\label{th:lin_converge2}
Assume that $\alpha=1, \Sigma_{\omega}=0$ and 
consider the iteration
\eqref{eq:KF_pred_mean2}, \eqref{eq:Lin_KF_analysis} mapping
$(m_n,C_n)$ into $(m_{n+1},C_{n+1})$. Assume further 
that $\Sigma_{\nu}\succ 0$ and that $C_0 \succ 0$. Then $C_n \succ 0$ for all $n \in \bbN$
and 
\begin{subequations}
\label{eq:newvar2}
\begin{align}
    (C_n')^{-1}&=I+nS,\\
    (I+nS)m_n'&=m_0'+n(G')^T\Sigma_{\nu}^{-1}y.
\end{align}
\end{subequations}
Thus, as $n \to \infty$, with $S^+$ denoting the Moore-Penrose pseudo-inverse of $S$,
\begin{subequations}
\begin{align}
 &n^{-1}P(C_n')^{-1} = S+\mathcal{O}(n^{-1}),   &Q(C_n')^{-1}=Q,\\
 &Pm_n'=S^+(G')^T\Sigma_{\nu}y+\mathcal{O}(n^{-1}),  &Qm_n'=Qm_0'.
 \end{align}
\end{subequations}

\end{theorem}
\begin{proof}
The proof is in \ref{sec:appendix}.
\end{proof}

\begin{remark}
\label{rem:2}
Consider Theorem \ref{th:lin_converge2}, in which
$\alpha=1$ 
and $\Sigma_{\omega}=0$, and note that $S \succ 0$ in $P\bbR^{N_{\theta}}$.
The theorem shows that the covariance of the filtering
distribution of the stochastic dynamical system underlying
the original implementation of EKI 
exhibits collapse to zero at algebraic rate
in the observed subspace $P\bbR^{N_{\theta}}$, and is unchanged in 
the unobserved subspace $Q\bbR^{N_{\theta}}$.
The mean converges algebraically
slowly at rate $\mathcal{O}(n^{-1})$ in $P\bbR^{N_{\theta}}$ 
and is unchanged in $Q\bbR^{N_{\theta}}$.
\end{remark}

\begin{remark} Theorem \ref{th:lin_converge}-\ref{th:lin_converge2} suggests
the importance of choosing $\alpha \in (0,1)$ 
and $\Sigma_{\omega} \succ 0$ in the stochastic
dynamical systems that we propose here, as this ensures exponential convergence of the filtering distribution to a regularized least squares problem. However, if the forward operator has empty null-space, the situation arising when the inversion problem is well-determined or over-determined,
then $\alpha=1$ may be chosen but it is again important to ensure
$\Sigma_{\omega} \succ 0$ to avoid the algebraic convergence
exhibited in Theorem \ref{th:lin_converge2}.
In the case $\alpha \in (0,1)$ Theorem \ref{th:lin_converge} demonstrates the regularization which underlies the
proposed iterative method.
In the case $\alpha =1$, the regularization term vanishes.
\end{remark}

\begin{remark}
The behaviour of the finite particle size EKI, in the case $\alpha=1$, $\Sigma_\omega=0$,
is fully analyzed in \cite{schillings2017analysis}. Theorem  \ref{th:lin_converge2} is a mean-field counterpart of that theory.
\end{remark}

The following proposition is relevant to understanding some
of the numerical experiments  presented later in the
paper and, taken together with Theorems \ref{th:lin_converge} and
\ref{th:lin_converge2}, it also completes our analysis
of the filtering distribution for the novel stochastic
dynamical system introduced in this paper.

\begin{proposition}
\label{th:lin_converge3}
Assume that $\alpha=1$ and
$\Sigma_{\omega} \succ 0$ and consider the setting where the forward operator $G$ has non-trivial null space 
(thus violating the assumption 
$\text{Range}(G^T)=\R^{N_{\theta}}$ in
Theorem \ref{th:lin_converge}).  Assume further 
that $\Sigma_{\nu}\succ 0$ and that $C_0 \succ 0$. Then
$\Cov_n \succ 0$ for all $n\in\bbN$ and
$\mean_n$ converges to a minimizer of 
$\frac{1}{2}\lVert \Sigma_{\nu}^{-\frac{1}{2}} (y - G \theta)\rVert^2$ exponentially fast.  However $\Cov_n^{-1}$ converges to a singular matrix and hence
$\|\Cov_n\|$ diverges to $+\infty$; the rate of divergence
is bounded by
\begin{equation}
    \label{eq:ntb}
\Cov_n \preceq \Cov_0 + n \Sigma_{\omega}.
\end{equation}
\end{proposition}

\begin{proof}
The proof is in \ref{sec:appendix}.
\end{proof}

\subsection{ExKI: Levenberg–Marquardt Connection}
\label{ssec:LMA}
In the nonlinear setting, our numerical results will demonstrate
the implicit regularization and linear (sometimes superlinear) 
convergence of ExKI and UKI. This desirable feature can
be understood by the analogy with the Levenberg–Marquardt Algorithm~(LMA). 
We focus this discussion on the particular case $\alpha=1$ as we find
that, for over-determined problems, this choice often produces the best results.

Consider the non-regularized nonlinear least-squares objective function
$\Phi$, defined in (\ref{eq:KI2}b). 
The key step in the Levenberg–Marquardt Algorithm~(LMA) is to solve the minimization problem 
for~(\ref{eq:KI2}b) by a preconditioned gradient descent procedure which
maps $\theta_n$ to $\theta_n+\delta \theta_n$ and where $\delta \theta_n$ solves
 \begin{equation}
    \begin{split}
    \label{eq:LMA-a1}
        (d\mathcal{G}(\theta_n)^T\Sigma_{\nu}^{-1}d\mathcal{G}(\theta_n)  + \lambda_n \I)\delta \theta_n = d\mathcal{G}(\theta_n)^T\Sigma_{\nu}^{-1}(y - \mathcal{G}(\theta_n)).
\end{split}
\end{equation}
Here $\I$ is the identity matrix on $\R^{N_\theta}$ and $\lambda_n$ is the (non-negative) damping factor, often chosen adaptively.  Because
of the damping matrix $\lambda_n \I$, the LMA is 
found to be more robust than the Gauss–Newton Algorithm and exhibits
linear (or even superlinear) convergence in practice. 
The use of LMA for inverse problems is discussed
in \cite{hanke1997regularizing}.

The ExKI procedure solves the optimization problem for (\ref{eq:KI2}b)
by a different preconditioned gradient descent procedure, defined by the
update
 \begin{equation}
    \begin{split}
    \label{eq:ExKI-theorem-a1}
        \Big(d\mathcal{G}(\theta_n)^T\Sigma_{\nu}^{-1}d\mathcal{G}(\theta_n) + \left(\Cov_{n}  + \Sigma_{\omega}\right)^{-1}\Big)\delta \theta_n &= d\G(\theta_n)^T \Sigma_{\nu}^{-1}\left(y -  \G(\theta_n)\right).
\end{split}
\end{equation}
This may be viewed as a generalization of the LMA in which the 
adaptive damping term is now a matrix $\Cov_{n}  + \Sigma_{\omega}$
and the adaptation is automated through the covariance updates;
furthermore this matrix is lower bounded (in the sense of
quadratic forms) by $\Sigma_{\omega}$, regardless of the adaptation
through the covariance, ensuring some damping of the Gauss-Newton
approximate Hessian. We may expect that the UKI and EKI, which
approximate the linearization $d\G$ in the ExKI, to benefit from this
generalized LMA. Connections between the LMA and EKI were first 
systematically explored in \cite{iglesias2016regularizing} and more
recently in \cite{chada2020iterative}.

\subsection{UKI: Unscented Approximation and Averaging}
\label{ssec:UA}

Here we explain that the unscented transform may be viewed as
smoothing the energy landscape of UKI, in comparison with ExKI;
this helps to explain the improved behaviour of UKI over ExKI
on rough landscapes, such as those
we will show in section \ref{sec:app} when performing parameter estimation for
chaotic differential equations.
To understand this smoothing effect we first introduce a useful averaging property~\cite[Theorem~1]{csato2002sparse}.
\footnote{In
what follows, the suffix $_{\succeq 0}$ denotes positive semi-definite matrix
and $\frac{\partial}{\partial m}$ denotes gradient with respect to $m.$}
\begin{lemma}
\label{th:filter}
Let $\theta$ denote Gaussian random vector $\displaystyle \theta \sim \N(\mean, \Cov) \in \R^{N_{\theta}}$. For any nonlinear function $\G:\R^{N_\theta} \to \R^{N_y}$, we define the associated averaged function $\F\G:\R^{N_\theta} \times \R^{N_\theta \times N_\theta}_{\succeq 0}
\to \R^{N_y}$ and averaged gradient function $\F d \G: 
\R^{N_\theta} \times \R^{N_\theta \times N_\theta}_{\succeq 0} \to 
\R^{N_y \times N_\theta}$ as follows:
\begin{equation}
\label{eq:aGAA}
    \F\G(\mean, \Cov) := \E[\G(\theta)]  \qquad 
    \F d \G(\mean, \Cov) := \mathrm{Cov}[\G(\theta), \theta] \cdot \Cov^{-1}. 
\end{equation}
Then we have $\displaystyle \frac{\partial\F\G(\mean,\Cov)}{\partial\mean} = \F d \G(\mean,\Cov)$. 
\end{lemma}
\begin{proof}
The proof is in \ref{sec:appendix}.
\end{proof}

Note that in the linear case 
$\F\G(\mean, \Cov)=\G(\mean)$ and $\F d \G(\mean, \Cov)=\G$; the averaged derivative is exact.
This averaging procedure is useful to understand the conceptual GAA precisely because \eqref{eq:aGAA}
may be used to express $\textrm{Cov}[\G(\theta), \theta]$, which appears in the conceptual GAA,
in terms of the averaged derivative $\F d \G(\mean, \Cov).$ In order to
use this idea in the context of the UKI it is useful to understand related averaging operations
when the modified unscented transform (Definition \ref{def:unscented_tranform})
is employed to approximate Gaussian expectations. To this end we define, using 
\eqref{eq:KF_pred_mean_u}--\eqref{eq:UKI-analysis},
\begin{equation}
\begin{split}
\label{eq:UKI_smooth_0}
    \Fu\G_{n} & := \py_{n},\\
    \Fu d \G_{n} & := {\pCov_{n}^{\theta y}}{}^T {\pCov_{n}}^{-1},
\end{split}
\end{equation}
noting that $\Fu\G_n$ and $\Fu d\G_n$ then correspond to approximation of \eqref{eq:aGAA} at step $n$ of the
algorithm, using the modified unscented transform from Definition \ref{def:unscented_tranform}.

\begin{proposition}
\label{th:uki}
The UKI algorithm \eqref{eq:KF_pred_mean_u}--\eqref{eq:UKI-analysis}
may be written in the following form:
\begin{itemize}
    \item Prediction step :
    \begin{equation}
\label{eq:KF_pred_mean_ex_smooth}
\begin{split}
    \pmean_{n+1} = & \alpha\mean_n + (1 - \alpha)r_0,\\
    \pCov_{n+1} = & \alpha^2 \Cov_{n} + \Sigma_{\omega}.
\end{split}
\end{equation}
    \item Analysis step :
    \begin{equation}
    \label{eq:UKI_smooth:approx}
    \begin{split}
    &\py_{n+1} = \Fu\G_{n+1},\\
    &\pCov^{\theta y}_{n+1} = \pCov_{n+1} \Fu d\G_{n+1}^{T},\\
    &\pCov^{yy}_{n+1} =
        \Fu d\G_{n+1}\pCov_{n+1} \Fu d\G_{n+1}^{T} + \Sigma_{\nu} +
        \widetilde{\Sigma}_{\nu,n+1},\\
    &\mean_{n+1} = \pmean_{n+1} + \pCov_{n+1}^{\theta y}(\pCov_{n+1}^{yy})^{-1}\bigl(y - \py_{n+1}\bigr),\\
    &\Cov_{n+1} = \pCov_{n+1} - \pCov^{\theta y}_{n+1}(\pCov^{yy}_{n+1})^{-1}{\pCov^{\theta y}_{n+1}}{}^{T}.\\
    \end{split}
    \end{equation}
    \end{itemize}
Here $\widetilde{\Sigma}_{\nu,n+1}\succeq 0.$ Furthermore,  $\|\widetilde{\Sigma}_{\nu,n+1}\| = \bigO(\|\pCov_{n+1}\|^2)$ and $\widetilde{\Sigma}_{\nu,n+1} = 0$ when $\G$ is linear. 
\end{proposition}

\begin{proof}
The proof is in \ref{sec:appendix}.
\end{proof}

\begin{remark}
\label{rem:UKI}
Comparison of the original UKI algorithm \eqref{eq:KF_pred_mean_u}--\eqref{eq:UKI-analysis} with its rewritten form
\eqref{eq:KF_pred_mean_ex_smooth}-\eqref{eq:UKI_smooth:approx}
demonstrates that, in the regime where the covariance is small, or
the forward model is linear, the UKI algorithm behaves like the ExKI
algorithm \eqref{eq:KF_pred_mean_ex}-\eqref{eq:ExKI-1.2} but with
the nonlinear function $\G$ and its associated gradient $d\G$ having been
averaged according to unscented approximations of the averaging operations
defined in Lemma \ref{th:filter}. From the preceding subsection, it follows that the UKI is also related to a modified LMA applied to an averaged objective function.
Note that, by using the unscented approximation of the averaging procedure defined in
Lemma \ref{th:filter}, we essentially remove the averaging of $\G$ and retain it only on $d\G.$
Averaging of the gradient $d\G$ alone will be demonstrated to have an important positive effect
on parameter estimation for  chaotic dynamical systems in Subsections \ref{sec:app:Lorenz63}, \ref{sec:app:Lorenz96}, and \ref{sec:app:GCM}.
\end{remark}

\subsection{Continuous Time Limit}
\label{ssec:cts}

To derive a continuous-time limit we set $\alpha=1-\alpha_0 h$, $\Sigma_\omega
\mapsto h\Sigma_\omega$, and $\Sigma_\nu \mapsto h^{-1}\Sigma_\nu.$  The
algorithm defined by \Cref{eq:KF_pred_mean,eq:KF_joint2,eq:KF_analysis} then
has the form of a first order accurate (in $h$) approximation of the dynamical system
\begin{subequations}
\label{eq:cts}
\begin{align}
\dot{m}&=-\alpha_0(m-r_0)+\Cov^{\theta y} \Sigma_{\nu}^{-1} \bigl(y - \E \G(\theta)\bigr), \label{eq:cts-a}\\
\dot{C}&=-2\alpha_0 C+\Sigma_{\omega}-\Cov^{\theta y} {\Sigma_{\nu}}^{-1}{\Cov^{\theta y}}^T, \label{eq:cts-b}
\end{align}
\end{subequations}
where $\theta \sim \N(m,C)$, expectation $\E$ is with respect to this distribution and
$$\Cov^{\theta y}=\E\Bigl(\bigl(\theta-m\bigr)\otimes\bigl(\G(\theta)-\E \G(\theta)\bigr)\Bigr).$$
This continuous-time dynamical system may be used as the basis for practical
algorithms by discretizing in time, for example, using forward Euler with an
adaptive time-step as in \cite{kovachki2019ensemble}, and applying the same ideas
used in the ExKF, UKI or EKI to approximate the expectations.

The steady state $m_{\infty}, C_{\infty}$ of the differential equations \eqref{eq:cts} 
are implicitly defined in a somewhat complicated fashion. However, any such steady state always has
non-singular covariance as we now state and prove.

\begin{lemma}
\label{lemma:UKI-Continuous}
For any  steady state $(m_{\infty}, C_{\infty})$  of \cref{eq:cts}, the steady covariance
$C_{\infty}$ is non-singular.
\end{lemma}
\begin{proof}
The proof is in \ref{sec:appendix}.
\end{proof}

\section{Variants on the Basic Algorithm}
\label{sec:V}

\subsection{Enforcing Constraints}
\label{ssec:constraints}

Kalman inversion requires solving forward problems at every iteration. Failure of the forward problem to deliver physically meaningful solutions can lead to failure of the inverse problem. 
Adding constraints to the parameters~(for example, dissipation is non-negative) significantly improves the robustness of Kalman inversion.
Within the EKI there is a natural way to impose constraints, using the
fact that each iteration of the algorithm may be interpreted as solving
a set of coupled quadratic optimization problems, with coupling arising from
empirical covariances. These optimization problems are readily appended
with convex constraints, such as box (inequality) 
constraints~\cite{albers2019ensemble}; see also \cite{iglesias2016regularizing,chada2020tikhonov}. 
The UKI does not have this optimization
interpretation and so we adopt a different approach to enforcing box
constraints.

In this paper there are occasions where
we impose element-wise box constraints of the form
\begin{equation*}
    0 \leq \theta  \quad \textrm{ or }\quad \theta_{min} \leq \theta \leq \theta_{max}.
\end{equation*}
These are enforced by change of variables writing $\theta = \varphi(\tilde{\theta})$
where, for example, respectively,
\begin{equation*}
   \varphi(\tilde{\theta}) = |\tilde{\theta}| \quad \textrm{ or } \varphi(\tilde{\theta}) = \theta_{min} +  \frac{\theta_{max} - \theta_{min}}{1 + |\tilde\theta|}.
\end{equation*}
The inverse problem is then reformulated as 
\begin{equation*}
    y = \G(\varphi(\tilde\theta)) + \eta,
\end{equation*}
and the UKI methods and variants are employed with $\G \mapsto \G \circ \varphi.$

\subsection{Unscented Kalman Sampler}
\label{ssec:vir}

Consider the following stochastic dynamical system, in which $W$ is a standard
unit Brownian motion in $\R^{N_\theta}:$
\begin{subequations}
\label{eq:uks0}
\begin{align}
\dot{\theta}&=\Cov^{\theta y} \Sigma_{\eta}^{-1} \bigl(y - \G(\theta)\bigr)-C\Sigma_0^{-1}(\theta-r_0)+\sqrt{2}C^{\frac{1}{2}}\dot{W},\\
\Cov^{\theta y}&=\E\Bigl(\bigl(\theta-m\bigr)\otimes\bigl(\G(\theta)-\E \G(\theta)\bigr)\Bigr)
\end{align}
\end{subequations}
and all expectations are computed under the law of $\theta$, with respect to which
the mean and covariance are denoted as $\mean$ and $\Cov$ respectively. 
This It\`o-McKean diffusion process can be
approximated by an interacting particle system,
and the law of $\theta$ approximated using the
resulting empirical Gaussian approximation, leading to
the EKS \cite{garbuno2020interacting}; we now generalize this to an unscented version.
First consider the following evolution equations
for the mean and covariance of the Gaussian approximation
to the law of $\theta:$
\begin{subequations}
\label{eq:uks}
\begin{align}
\dot{m}&=\Cov^{\theta y} \Sigma_{\eta}^{-1} \bigl(y - \E \G(\theta)\bigr)-C\Sigma_0^{-1}(m-r_0),\\
\dot{C}&=-2\Cov^{\theta y} \Sigma_{\eta}^{-1}{\Cov^{\theta y}}^T-2C\Sigma_0^{-1}C+2C.
\end{align}
\end{subequations}
Note that the expectations are computed under the law of (\ref{eq:uks0}) and so
this is not, in general, a closed system for $(m,C).$

To obtain a closed system for $(m,C)$, we consider a Gaussian evolving according 
to the equations \eqref{eq:uks}, with matrix $\Cov^{\theta y}$ again given by
(\ref{eq:uks0}b), but now 
expectation $\E$ is computed with respect to the distribution $\N(m,C)$ so that a closed
system for $(m,C)$ is obtained.
The UKS is defined by approximating the expectations in this system 
by use of an unscented transform.

In the case where $\G$ is linear and the solution is initialized at a
Gaussian then the system \eqref{eq:uks} with expectations computed
under $\N(m,C)$ is consistent with the solution of the It\`o-McKean diffusion (\ref{eq:uks0}) governing $\theta$ -- that latter has Gaussian distribution evolving according to 
 \eqref{eq:uks}.
Furthermore, the analysis in \cite{garbuno2020interacting} shows that then
the system converges to the posterior distribution \eqref{eq:post}
at a rate $\exp(-t)$ independent of the problem being solved; this
independence of the rate on the problem conditioning may be viewed
as a consequence of affine invariance. We also mention that
the analysis in \cite{carrillo2019wasserstein} shows that, when initialized at a non-Gaussian,
the Gaussian dynamics is an attractor.
It is thus natural to consider using numerical simulations of 
\eqref{eq:uks}  to generate
approximate samples from the posterior distribution.
Illustrative examples are presented in \ref{sec:app:UKS}.

\section{Numerical Results}
\label{sec:app}
In this section, we present numerical results for Kalman-based inversion
using the proposed stochastic dynamical system~\cref{eq:std}. 

\subsection{Choice of Hyperparameters}
\label{ssec:params}
We make choices of $\Sigma_{\omega}$ and $\Sigma_{\nu}$
guided by the discussion in Remark \ref{rem:1}. However, for general nonlinear problems $\Cov_{*}$ is not explicitly defined. 
Thus we modify
the prescription given in \eqref{eq:cstar2} and instead choose 
\begin{subequations}
    \label{eq:cstar3}
    \begin{align}
    \Sigma_\nu &= 2 \Sigma_{\eta}\\
    \Sigma_{\omega} &= \bigl(2 - \alpha^{2}\bigl) \gamma \I
    \end{align}
\end{subequations}
for some $\gamma>0.$ 
For over-determined problems, when the observational noise is absent or negligible,  we take $\alpha=1.$ For under-determined problems, to avoid overfitting in the presence of noise,  we generally choose $\alpha \in (0,1)$; but we also present some under-determined problems 
with choice $\alpha=1$ to demonstrate undesirable effects from doing so. In general, cross-validation
should be invoked to determine an optimal choice of $\alpha.$ However in this paper,
 we have simply used the values $0.0, 0.5, 0.9, 1.0$ for illustrative purposes.
 To be concrete we initialize with $m_0=r_0$ and
$C_0=\gamma \I.$ Specific choices of $r_0$ and $\gamma$ will differ
between examples and will be spelled out in each case.

\subsection{Classes of Problems Studied}

For all applications, we focus mainly on the UKI; some comparisons between 
the UKI and EKI (specifically, as applied to the novel stochastic dynamical
system \eqref{eq:std} proposed here) are also presented; and computational difficulties
inherent in the rough misfit landscape experienced
by the ExKI for chaotic dynamical systems are illustrated,
and are demonstrably overcome by deploying the UKI. The applications cover a wide range of problems. They include three categories: 
\begin{enumerate}
    \item Noiseless linear problems, where over-determined, under-determined, and well-determined systems are considered.
\begin{itemize}
    \item Linear 2-parameter model problem: this problem serves as a proof-of-concept example, which demonstrates the convergence of the mean and the covariance matrix discussed in Subsection~\ref{ssec:lin}. In this case, the UKI is exact,
    as a consequence of Lemma \ref{lem:uki}; numerics are performed using only the UKI.
    \item Hilbert matrix problem: this problem illustrates
    the performance of the EKI and UKI when solving ill-conditioned inverse problems.  The EKI suffers from divergence as it is iterated. However the 
    UKI behaves well, again reflecting the exactness for linear
    problems, highlighted in Lemma \ref{lem:uki}, and the theory
    of Subsection~\ref{ssec:lin} characterizing the behaviour of
    the filtering distribution in the linear setting.
\end{itemize}
\item Noisy field recovery problems, in which we add $0\%$, $1\%$, and $5\%$ Gaussian random noise to the observation, as follows:
\begin{equation}
\label{eq:add-noise}
    y_{obs} = y_{ref} + \epsilon \odot\xi, \quad \xi \sim \N(0, \I),
\end{equation}
where $y_{ref} = \G(\theta_{ref})$,\, $\epsilon = 0\% y_{ref}, 1\% y_{ref}, \textrm{ and } 5\% y_{ref}$, and $\odot$ denotes element-wise multiplication.
It is important to distinguish between the added Gaussian random noise appearing in the data and the observation error model $\eta \sim \N(0, \Sigma_{\eta})$ used in the
development of the inversion algorithm; in essence we assume imperfect
knowledge of the noise model.\footnote{See section 7.1 of \cite{roininen2014whittle} for an example with a similar set-up; see also discussion around equation (55) in
\cite{iglesias2013evaluation} where the additive Gaussian noise used in the data
is carefully constructed to scale relative to the truth underlying it.}
Comparison of UKI and EKI is presented. EKI is shown to suffer from finite ensemble size effects, and in some cases diverges; in contrast, UKI behaves well. Thus we observe that what we have learned from the linear setting carries across to the setting of nonlinear inverse problems. This category of inversion for fields also serves to demonstrate the value of the Tikhonov regularization parameter $\alpha \in (0,1)$ in the prevention of overfitting. We consider three examples, now listed.
\begin{itemize}
    \item Darcy flow problem: to find permeability parameters in subsurface flow from measurements of pressure (or piezometric head).
    
     \item Damage detection problem:  determining the damage field in
     an elastic body from displacement observations on the surface of the structure. 
 
    \item Navier-Stokes problem: we study a two dimensional incompressible
    fluid, using the vorticity-streamfunction formulation, and recover the initial vorticity from noisy observations of the vorticity field at 
    later times.
     
\end{itemize}
\item Chaotic problems, in which the parameters are learned from  time-averaged statistics. For these problems, which are over-determined, we demonstrate that
choosing $\alpha=1$ is satisfactory, relying on the implicit regularization inherent in the approximate LMA interpretation of ExKI and UKI, as discussed in Subsection~\ref{ssec:LMA}. The three examples considered are now listed.
\begin{itemize}
    \item Lorenz63 model problem: we present a discussion of why adjoint based methods including ExKI, fail; we then demonstrate that the UKI succeeds. We attribute the success of the UKI to the averaging effect
    induced by the unscented transform and discussed in Subsection~\ref{ssec:UA}.
    \item Multiscale Lorenz96 problem: we study a scale-separated setting,
    in which the closure for the fast dynamics is learned from time-averaged statistics. 
    \item Idealized general circulation model problem: this is a 3D Navier-Stokes problem with a hydrostatic assumption,  and simple
    parameterized subgrid-scale models; we learn the parameters of the
    subgrid-scale model from time-averaged data. This problem demonstrates the potential of applying the UKI for large scale chaotic inverse problems. 
\end{itemize}
\end{enumerate}


\subsection{Linear 2-Parameter Model Problem}
\label{sec:app-lin}
Consider the 2-parameter linear inverse problem to
find $\theta\in \R^2$ from $y \in \R^{N_y}$ where
$y = G\theta$ with $G \in \R^{N_y \times 2}$ and no noise
is present in the data.
We explore the following three scenarios corresponding to
$N_y=3,2$ and $1$:
\begin{itemize}
    \item non-singular~(well-determined) system~(NS) $N_y=2$
   \begin{align*}
   y = \begin{bmatrix}
        3 \\
        7 
       \end{bmatrix}~~~
       G = \begin{bmatrix}
        1 & 2\\
        3 & 4
    \end{bmatrix}~~~
    \theta_{ref} = \begin{bmatrix}
        1 \\
        1 
       \end{bmatrix};
   \end{align*} 
\item over-determined system~(OD) $N_y=3$
    \begin{align*}
    y = \begin{bmatrix}
        3 \\
        7 \\
        10
       \end{bmatrix}~~~
       G = \begin{bmatrix}
        1 & 2\\
        3 & 4 \\
        5 & 6
    \end{bmatrix}~~~
    \theta_{ref} = \begin{bmatrix}
        1/3 \\
        17/12
       \end{bmatrix};
   \end{align*} 
    \item under-determined system~(UD) $N_y=1$
    \begin{align*}
    y = \begin{bmatrix}
        3
       \end{bmatrix}~~~
       G = \begin{bmatrix}
        1 & 2\\
    \end{bmatrix}~~~
       \theta_{ref} = \begin{bmatrix}
        1 \\
        1 
       \end{bmatrix}
       + c\begin{bmatrix}
        2 \\
        -1 
       \end{bmatrix}, \,\,c \in \R.
   \end{align*} 
\end{itemize}
Since there is no noise in the data we
have $\Sigma_\eta=0$ and $\Phi$ is undefined. 
To proceed we apply our methodology as if 
$\Sigma_\eta=0.1^2 \I$, corresponding to a misspecified
model. Then we may set $$\displaystyle \theta_{ref} = \argmin_{\theta}\Phi(\theta)=\argmin_{\theta}\frac{1}{2}\lVert(y - G\theta)\rVert^2.$$  
Note that for the OD and NS cases $\theta_{ref}$ is a single point, whereas in the UD case, $\theta_{ref}$ comprises a one-parameter ($c \in \R$) family of possible solutions.

We choose $r_0 = 0$, $\gamma=0.5^2$ and also initialize the UKI at $\theta_0 \sim \N(0, \gamma\I).$
In both the NS and OD cases $Range(G^T)=\R^{N_{\theta}}$ and so
we set $\alpha=1$, guided by Theorem \ref{th:lin_converge}. 
In the UD case $Range(G^T)\ne\R^{N_{\theta}}$ and we consider both $\alpha=1$  and $\alpha=0.5$, illustrating Proposition \ref{th:lin_converge3} and Theorem \ref{th:lin_converge} respectively.
The convergence of the parameter vectors $\{\mean_n\}$ is depicted in \cref{fig:Lin-theta}.
In all scenarios, the mean vectors converge 
to a limiting value exponentially fast. In the cases of NS and OD this is as
predicted by Theorem \ref{th:lin_converge} and, since $\alpha=1$, $\Phi_R$ and $\Phi$
coincide so that $m_{\infty}=\theta_{ref}.$ 
For UD with $\alpha=1$ and $\alpha=0.5$, the mean vectors converge to $[0.6\quad 1.2]^T$ and $[0.597\quad 1.195]^T$ respectively, following Proposition~\ref{th:lin_converge3} and Theorem~\ref{th:lin_converge}. However, the limiting mean for $\alpha=1$ depends on the initial conditions for the algorithm, whereas for $\alpha=0.5$ it is
uniquely determined.
The convergence of the covariance matrices $\{\Cov_n\}$ to $\Cov_{\infty}$ is depicted in \cref{fig:Lin-Sigma}, with NS, OD, and UD~($\alpha = 0.5$) on the left and UD~($\alpha = 1.0$) on the right. 
In the cases NS, OD, and UD~($\alpha = 0.5$), the estimated covariance matrices converge to the desired values~(the steady state of \cref{eq:cov_steady}, as predicted by Theorem \ref{th:lin_converge}).
In the case UD, the covariance matrices~$\{\Cov_n\}$ diverge to $+\infty$ (see Proposition \ref{th:lin_converge3}); nonetheless, this divergence of the covariance matrix does not affect the exponential convergence of the mean vector. 
In general, we advocate the use of $\alpha\in(0,1)$ for
under-determined problems and have set $\alpha=1$ for problem UD here only to illustrate some
of the issues that arise from doing so.

\begin{figure}[ht]
\centering
    \includegraphics[width=0.6\textwidth]{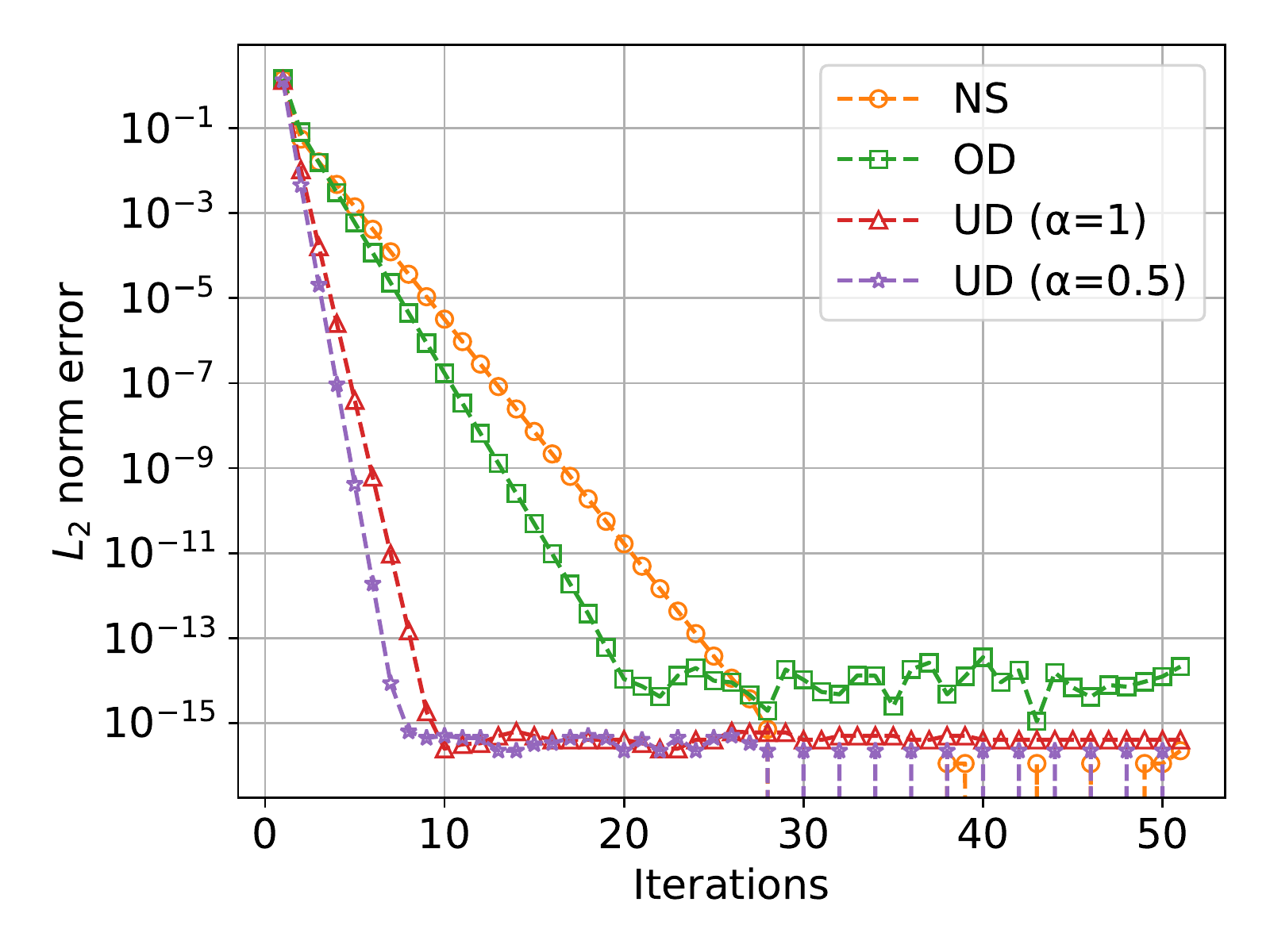}
    \caption{$L_2$ error $\lVert\mean_n - \theta_{ref}\rVert_2$ of the linear 2-parameter model problem. NS: non-singular system,  OD: over-determined system, UD: under-determined system.}\label{fig:Lin-theta}
    \end{figure}
    
\begin{figure}[ht]
\centering
    \includegraphics[width=0.49\textwidth]{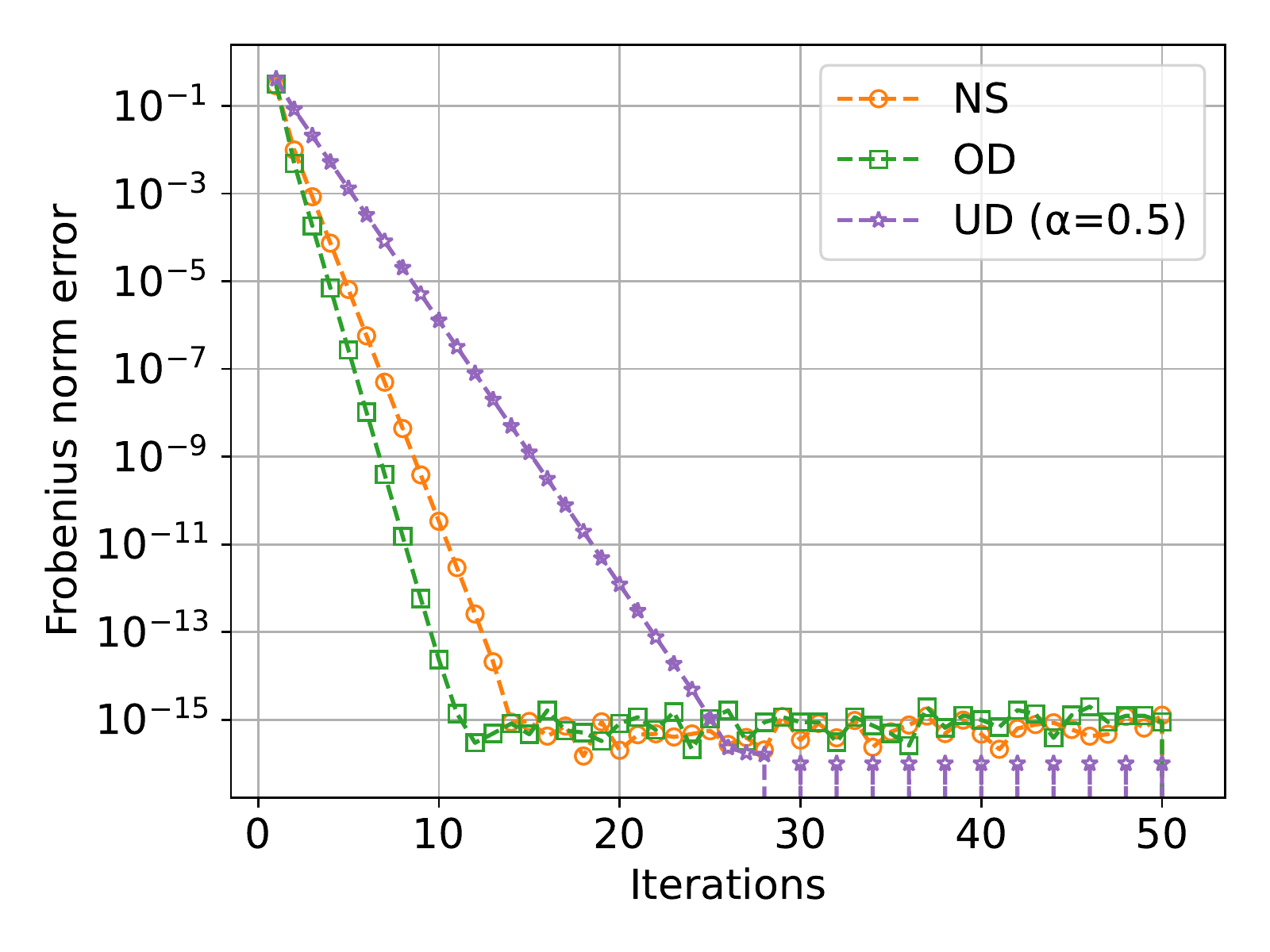}
    \includegraphics[width=0.49\textwidth]{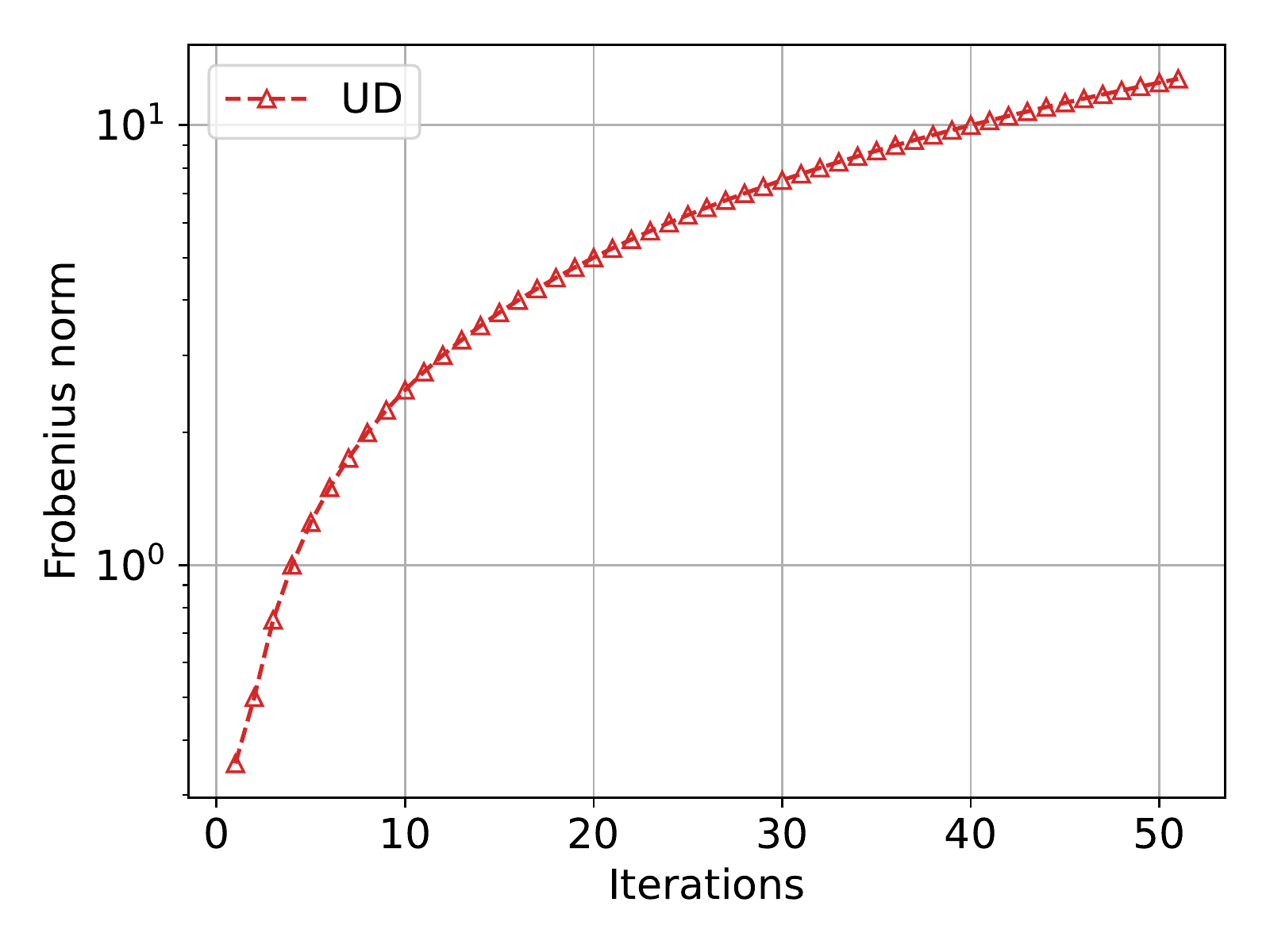}
    \caption{Frobenius norm $\lVert\Cov_n - \Cov_{\infty}\rVert_F$~(left) for non-singular~(NS) and over-determined~(OD) systems, and $\lVert\Cov_n\rVert_F$~(right) for the under-determined~(UD) system of the linear 2-parameter model problem.}
    \label{fig:Lin-Sigma}
\end{figure}

\subsection{Hilbert Matrix Problem}

\label{sec:app-Hilbert}

In this example $N_y=N_\theta.$ We define the Hilbert matrix $G \in R^{N_{\theta}\times N_{\theta}}$ by its entries
$$G_{i,j} = \frac{1}{i+j-1}.$$
The condition number of $G$ grows as $\bigO\Big((1+\sqrt{2})^{4N_{\theta}}/\sqrt{N_{\theta}}\Big)$ \cite{beckermann2000condition}. 
We consider the inverse problem of finding $\theta \in \R^{N_{\theta}}$ from
$y \in \R^{N_y}$ where $y = G\theta_{ref}$ and we define
$\theta_{ref}:= \mathds{1}.$ 
The ill-conditioning of $G$ makes the determination of $\theta$ from $y$ difficult.
Traditional linear solvers fail for such a problem.\footnote{$G\backslash y$ in Julia leads to an $L_2$ error of $4250.142$ for $N_\theta = 100$.} 

We consider two scenarios: $N_{\theta} = 10$ and $N_{\theta} = 100$. 
As in the previous linear case study we assume a model misspecification setting in
which $\Sigma_\eta=0.1^2 \I$, even though the data itself contains no noise, and we take
$\alpha=1.$ We set $r_0 = 0$ and $\gamma=0.5^2$. Thus  $\theta_0 \sim \N(0, 0.5^2\I)$.
Both UKI and EKI are applied. For the EKI, the ensemble sizes are set to $J = 2N_\theta+1$ 
and $J = 100N_\theta+1$. The convergence of the parameter vector $\mean_n$ is depicted in \cref{fig:Hilbert}. The UKI converges, but the convergence rate depends on the condition 
number of $G$, slowing as it grows. The EKI converges to a certain accuracy as fast as the 
UKI and then diverges. This divergence is related to the finite ensemble size, and is delayed by
use of the larger $J$. Indeed in the mean-field limit $J=\infty$ the EKI will coincide with
the UKI. This example clearly demonstrates the benefits of the UKI over the EKI.
\begin{figure}[ht]
\centering
    \includegraphics[width=0.49\textwidth]{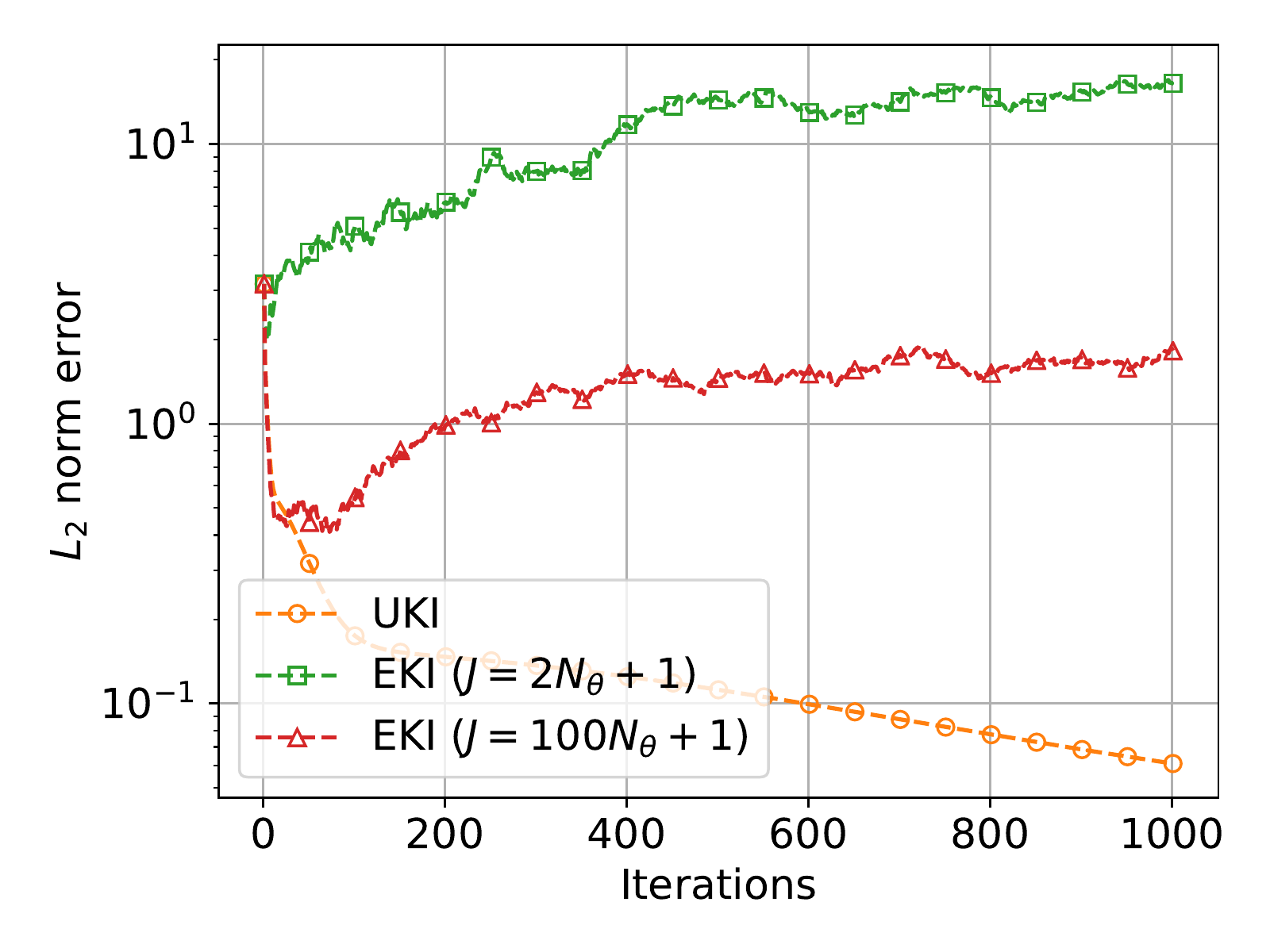}
    \includegraphics[width=0.49\textwidth]{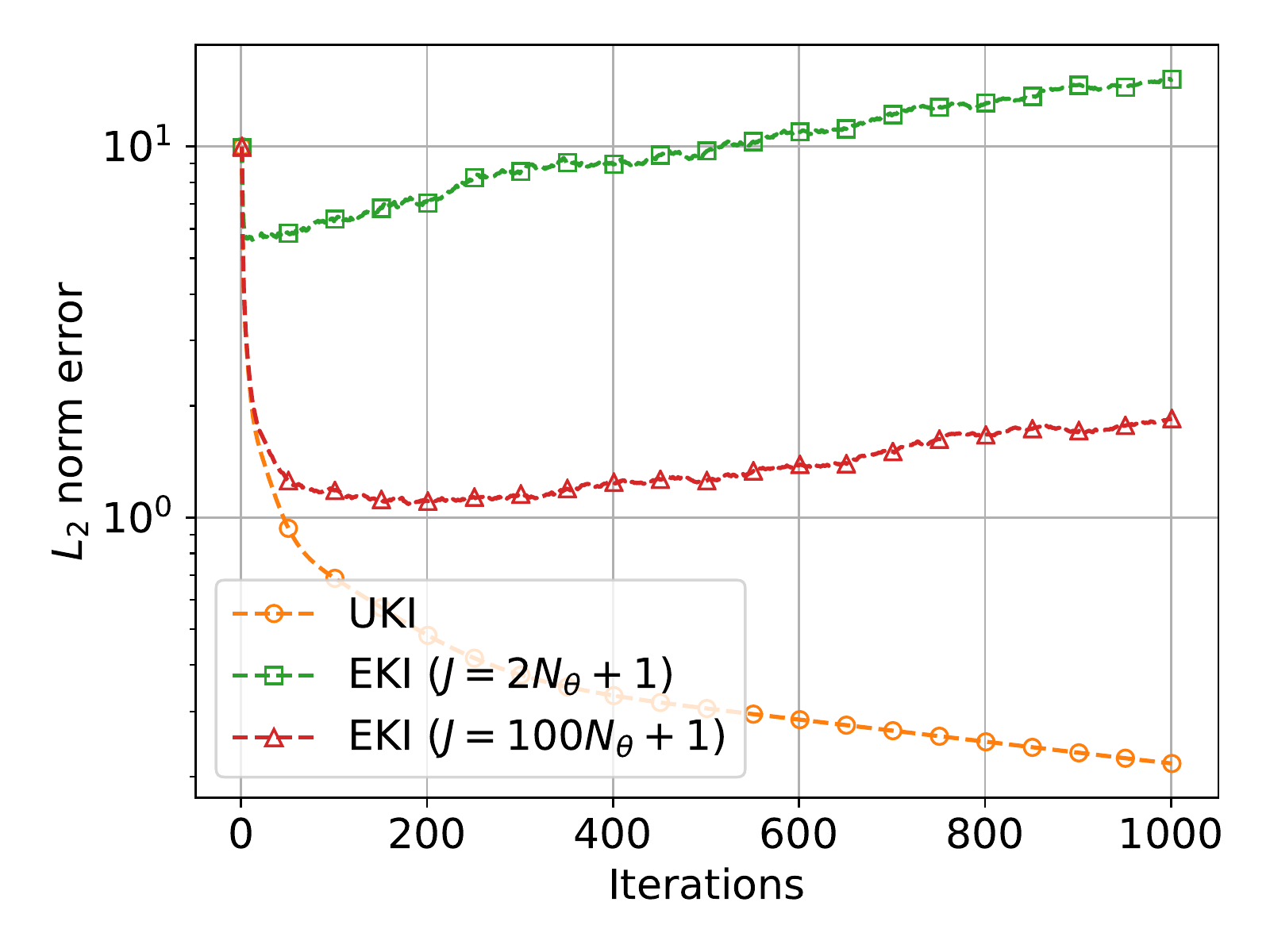}
    \caption{
    $L_2$ error $\lVert\mean_n - \theta_{ref}\rVert_2$ of the Hilbert inverse problem with $N_\theta = 10$~(left) and $N_\theta = 100$~(right).}
    \label{fig:Hilbert}
\end{figure}

\subsection{Darcy Flow Problem}
\label{sec:app:Darcy}
Consider the Darcy flow equation on the two-dimensional spatial domain $D=[0,1]^2$. The
forward model is to find the pressure field $p(x)$ in a porous medium defined by a positive permeability field $a(x,\theta)$:
\begin{align*}
    -\nabla \cdot (a(x, \theta) \nabla p(x)) &= f(x), \quad &&x\in D,\\
    p(x) &= 0, \quad &&x\in \partial D.
\end{align*}
For simplicity, we have imposed homogeneous Dirichlet boundary conditions on the pressure  at the boundary $\partial D$. The fluid source field $f$ is defined as
\begin{align*}
    f(x_1, x_2) = \begin{cases}
               1000 & 0 \leq x_2 \leq \frac{4}{6}\\
               2000 & \frac{4}{6} < x_2 \leq \frac{5}{6}\\
               3000 & \frac{5}{6} < x_2 \leq 1\\
            \end{cases}. 
\end{align*}
We study the inverse problems of finding $a$ from noisy measurements of $p$.
We place a prior on the permeability field $a(x,\theta)$ by
assuming that $\log a(x, \theta)$ is a centred Gaussian with covariance 
$$\mathsf{C} = (-\Delta + \tau^2 )^{-d};$$
here $-\Delta$ denotes the Laplacian on $D$ subject to homogeneous Neumann boundary conditions on the space of spatial-mean zero functions, 
$\tau > 0$ denotes the inverse length scale of the random field and $d  > 0$ determines its regularity ($\tau=3$ and $d=2$ in the present study).
See~\cite{chada2020tikhonov,dunlop2017hierarchical,garbuno2020interacting,nelsen2020random} for examples. The parameter $\theta$ represents the countable set of coefficients in the Karhunen-Lo\`{e}ve~(KL) expansion of
the Gaussian random field:
\begin{equation}
\label{eq:KL-2d}
    \log a(x,\theta) = \sum_{l\in K} \theta_{(l)}\sqrt{\lambda_l} \psi_l(x),
\end{equation}
where $K = \Z^{0+}\times\Z^{0+} \setminus \{0,0\}$, 
$\theta_{(l)} \sim \N(0,1)$ i.i.d. and the eigenpairs are of the form
\begin{equation*}
    \psi_l(x) = \begin{cases}
                 \sqrt{2}\cos(\pi l_1 x_1)              & l_2 = 0\\
                 \sqrt{2}\cos(\pi l_2 x_2)              & l_1 = 0\\
                 2\cos(\pi l_1 x_1)\cos(\pi l_2 x_2)    & \textrm{otherwise}\\
                 \end{cases},
                 \qquad \lambda_l = (\pi^2 |l|^2 + \tau^2)^{-d}.
\end{equation*}
The KL expansion~\cref{eq:KL-2d} can be rewritten as a sum over $\Z^{0+}$ rather than a lattice: 
\begin{equation}
\label{eq:KL-1d}
    \log a(x,\theta) = \sum_{k\in \Z^{0+}} \theta_{(k)}\sqrt{\lambda_k} \psi_k(x),
\end{equation}
where the eigenvalues $\lambda_k$ are in descending order.
In practice, we truncate this sum to $N_\theta$ terms, based on the
largest $N_\theta$ eigenvalues, and hence $\theta\in\R^{N_\theta}$.
The forward problem is solved by a finite difference method on an $80 \times 80$ grid. 

\begin{figure}[ht]
\centering
\includegraphics[width=0.5\textwidth]{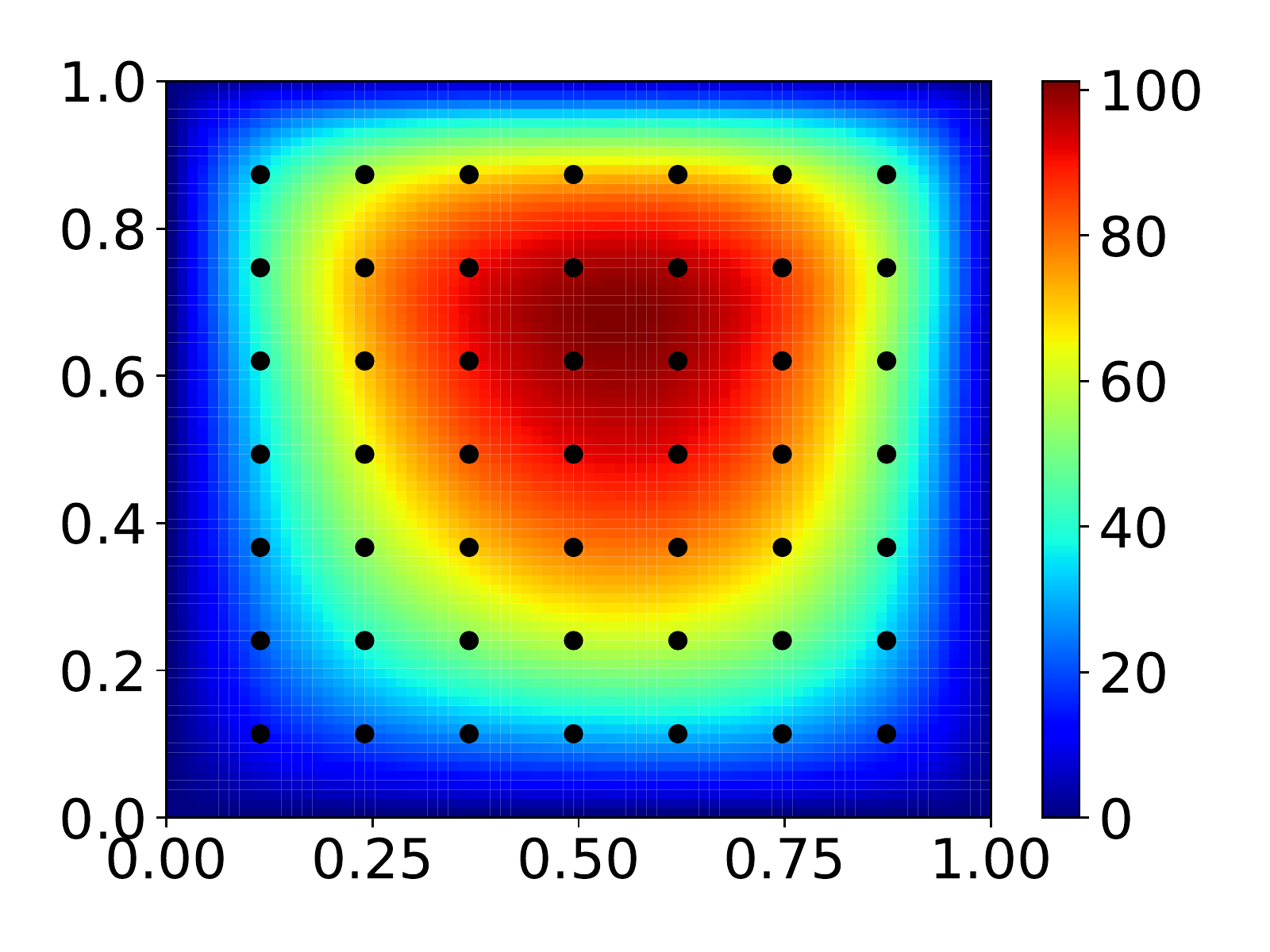}
\caption{The pressure field of the Darcy flow problem and the $49$ equidistant pointwise measurements~(black dots).}
\label{fig:Darcy-obs}
\end{figure}

For the inverse problem, the observation $y_{ref}$ consists of pointwise measurements of the pressure value $p(x)$ at $49$ equidistant points in the domain~(See~\cref{fig:Darcy-obs}). 
We generate a truth random field $\log a_{{ref}}(x)$ with $\theta \sim \N(0, \I)$ in $\R^{256}$ (i.e. we use the first $256$ KL modes) to construct the observation $y_{ref}$; different levels of noise are added to make
data $y_{obs}$ as explained in \eqref{eq:add-noise}.
Using this data, we consider two incomplete parameterization scenarios:  solving for the first $32$ KL modes~($N_\theta=32$) and for the first $8$ KL modes~($N_\theta=8$).  EKI and UKI are both applied. We take $r_0 = 0$ and $\gamma=1$ so that $\theta_0 \sim \N(0, \I)$. The observation error satisfies $\eta \sim \N(0, \I)$. For the EKI, the ensemble size is set to be $J = 100$, which is larger than the number of $\sigma-$points used in UKI~($2N_\theta+1$). 

For the $N_\theta = 32$ case, the convergence of the log-permeability fields $\log a(x, \mean_n)$ and the optimization errors~\eqref{eq:KI2} at each iteration for different noise levels are depicted in \cref{fig:Darcy-32-converge}; the top row shows the relative $L_2$ errors in the estimate of $\log a$ and the bottom row shows the optimization errors~(data-misfit), left to right corresponds to different noise levels in the data.
Without explicit regularization~($\alpha=1.0$), both UKI and EKI suffer from overfitting for noisy scenarios: the optimization errors keep decreasing, but the parameter errors show the ``U-shape'' characteristic of overfitting. Adding regularization~($\alpha=0.5$) relieves the overfitting. 
The estimated log-permeability fields $\log a(x,m_n)$ at the 50th iteration and the truth random field are depicted in \cref{fig:Darcy-32}.  Both UKI and EKI deliver similar results and these estimated log-permeability fields capture main features of the truth random field.

\begin{figure}[ht]
\centering
    \includegraphics[width=0.32\textwidth]{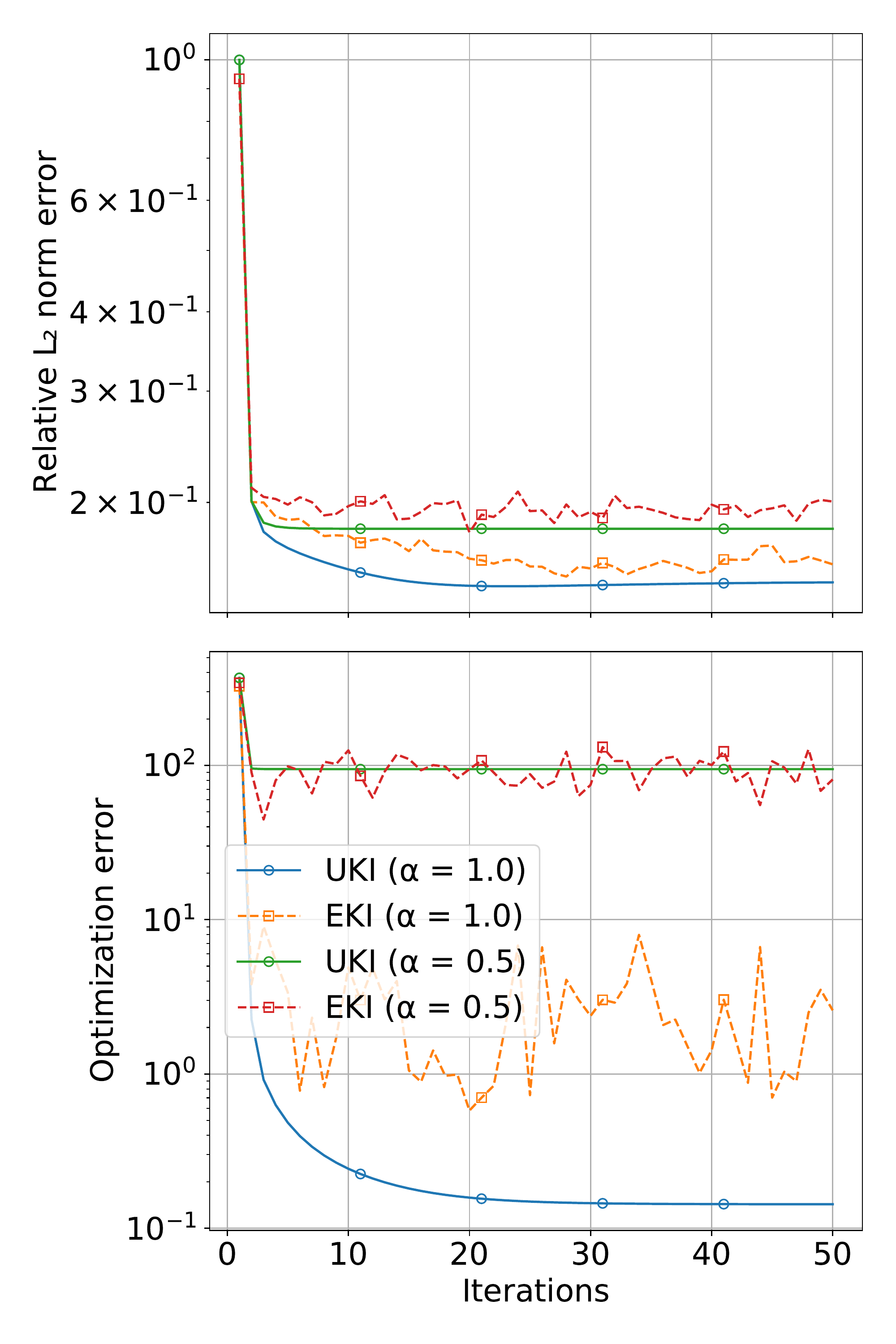}
    \includegraphics[width=0.32\textwidth]{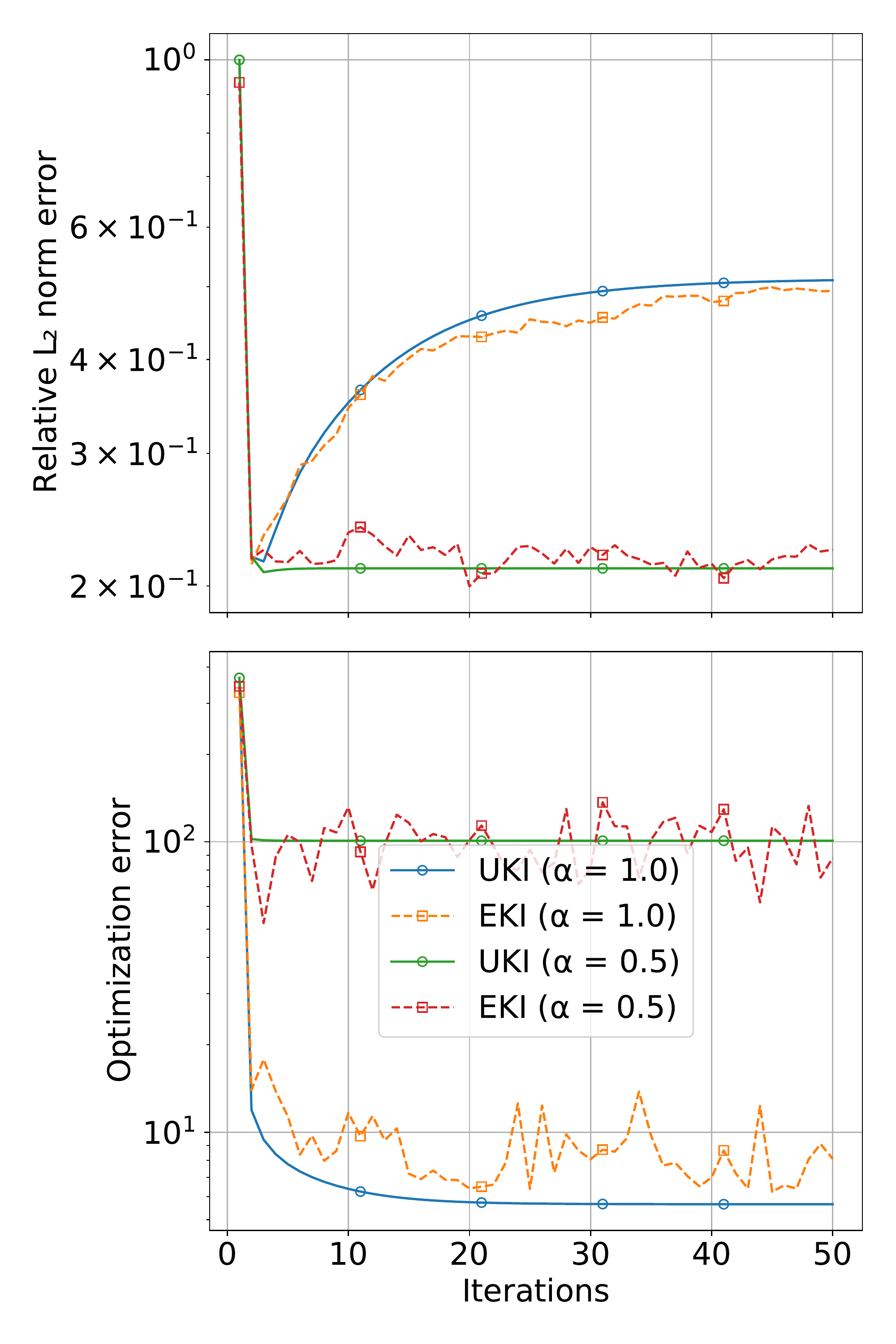}
    \includegraphics[width=0.32\textwidth]{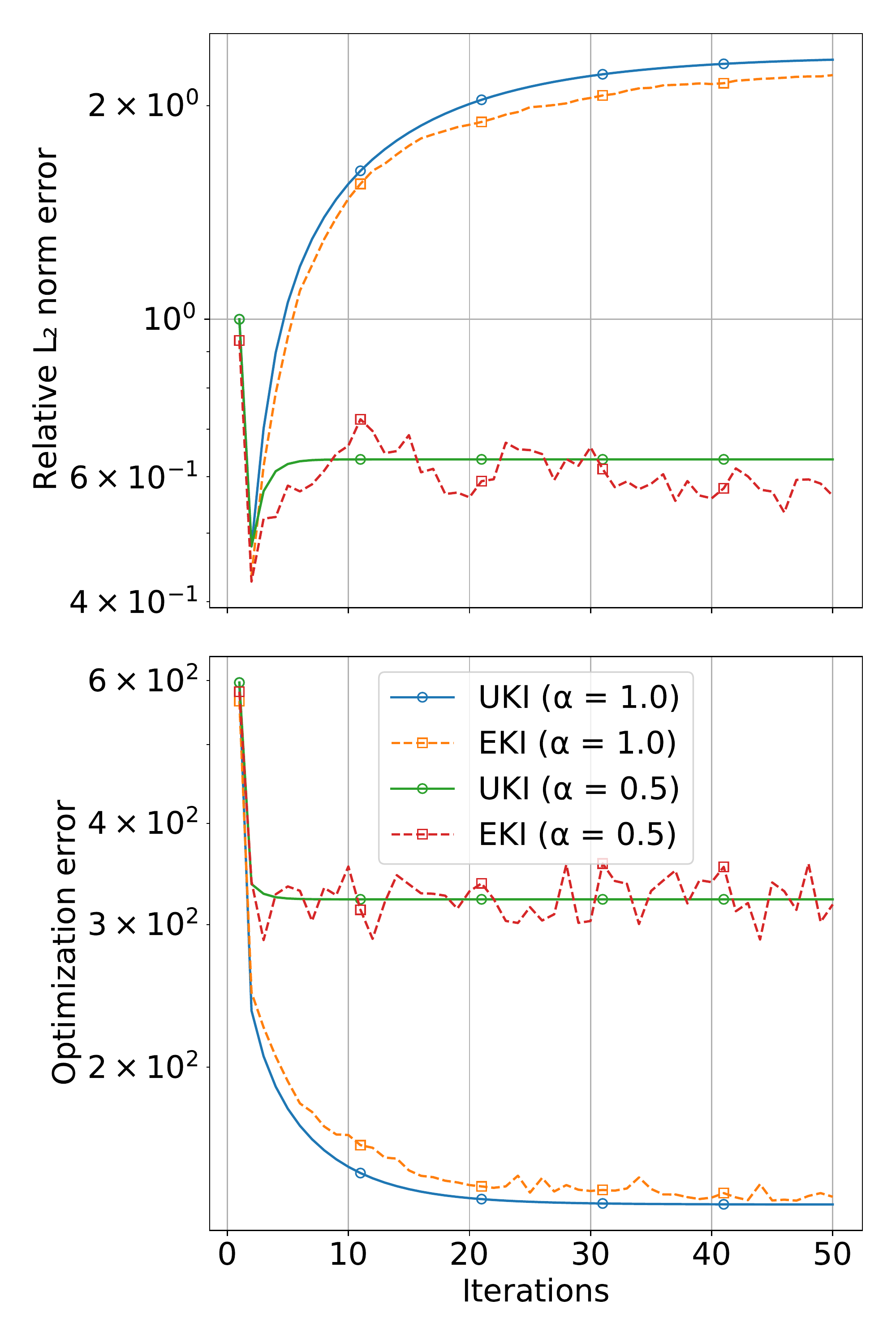}
    \caption{
    Relative error $\displaystyle \frac{\lVert\log a(x, \mean_n) - \log a_{ref}(x)\rVert_2}{\lVert\log a_{ref}(x)\rVert_2}$ (top) and the optimization error $\displaystyle \frac{1}{2}\lVert\Sigma_{\eta}^{-\frac{1}{2}} (y_{obs} - \py_n)\rVert^2$~(bottom) of the Darcy problem~($N_{\theta}=32$) with different noise levels: noiseless (left), $1\%$ error (middle), and $5\%$ error (right).}
    \label{fig:Darcy-32-converge}
\end{figure}

\begin{figure}[ht]
\centering
    \includegraphics[width=0.72\textwidth]{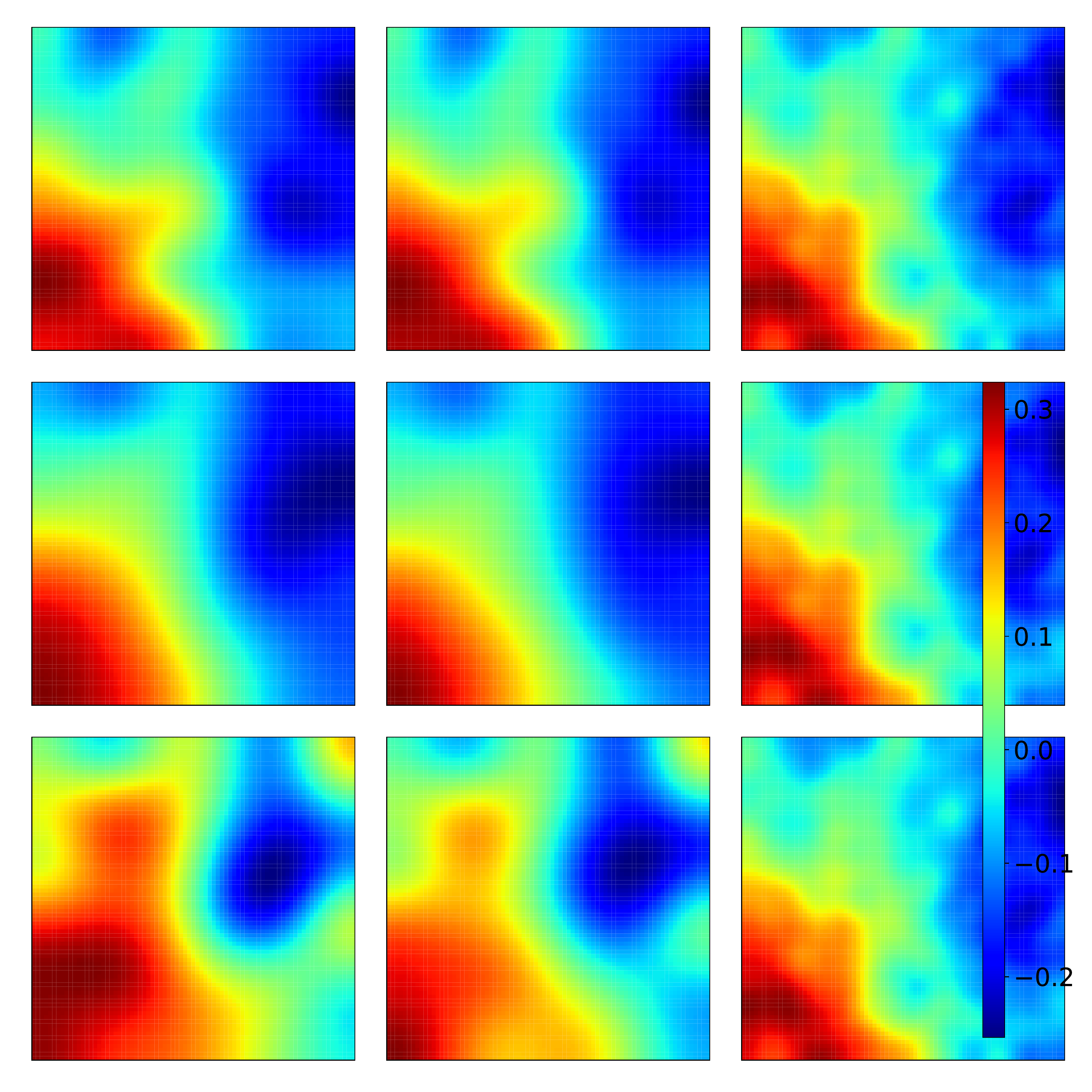}
    \caption{
    Log-permeability fields $\log a(x, \mean_n)$ with $N_{\theta}=32$ obtained by UKI, EKI, and the truth~(left to right) for different noise levels: noiseless~$\alpha=1$~(top), $1\%$ noise~$\alpha=0.5$~(middle), $5\%$ noise~$\alpha=0.5$~(bottom).}
    \label{fig:Darcy-32}
\end{figure}

For the $N_\theta = 8$ case, the convergence of the log-permeability fields $\log a(x, \mean_n)$ and the optimization errors at each iteration for different noise levels are depicted in \cref{fig:Darcy-8-converge}. 
Even without explicit regularization~($\alpha=1.0$), none of these Kalman inversions suffer from overfitting. Both UKI and EKI lead to similar parameter errors and optimization errors. 
The estimated log-permeability fields $\log a(x,m_n)$ at the 50th iteration for different noise levels, obtained by the UKI and the truth random field, are depicted in \cref{fig:Darcy-8}. Comparing with the $N_\theta = 32$ case, all Kalman inversions with $N_{\theta} = 8$ perform better for the $5\%$ noise scenario. This indicates the possibility of regularizing the inverse problem by reducing the parameter dimensionality.

Finally we observe the smoothness, as a function of the iteration number, of the UKI in comparison to EKI. This may be seen in all the experiments undertaken in the Darcy flow example.

\begin{figure}[ht]
\centering
    \includegraphics[width=0.32\textwidth]{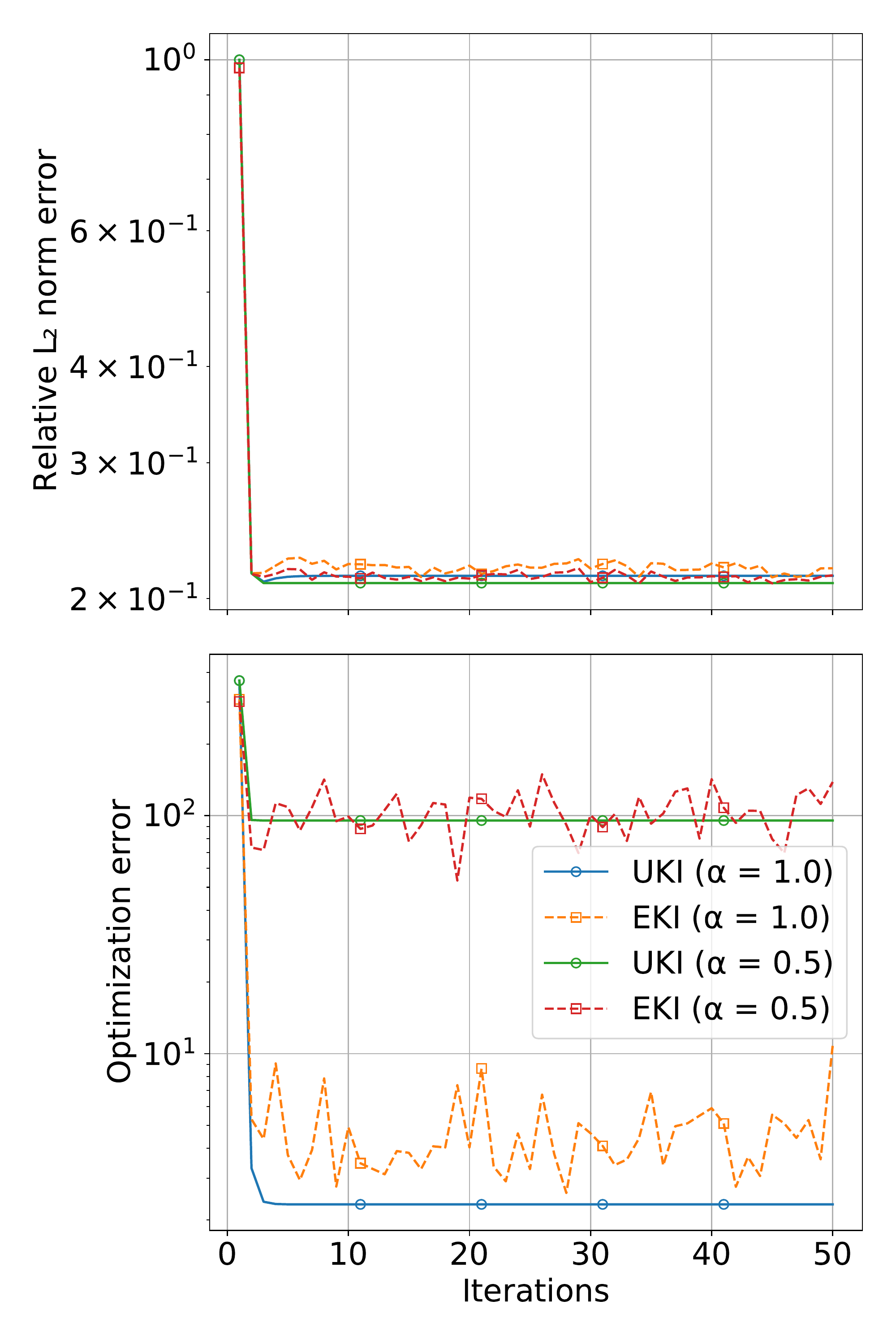}
    \includegraphics[width=0.32\textwidth]{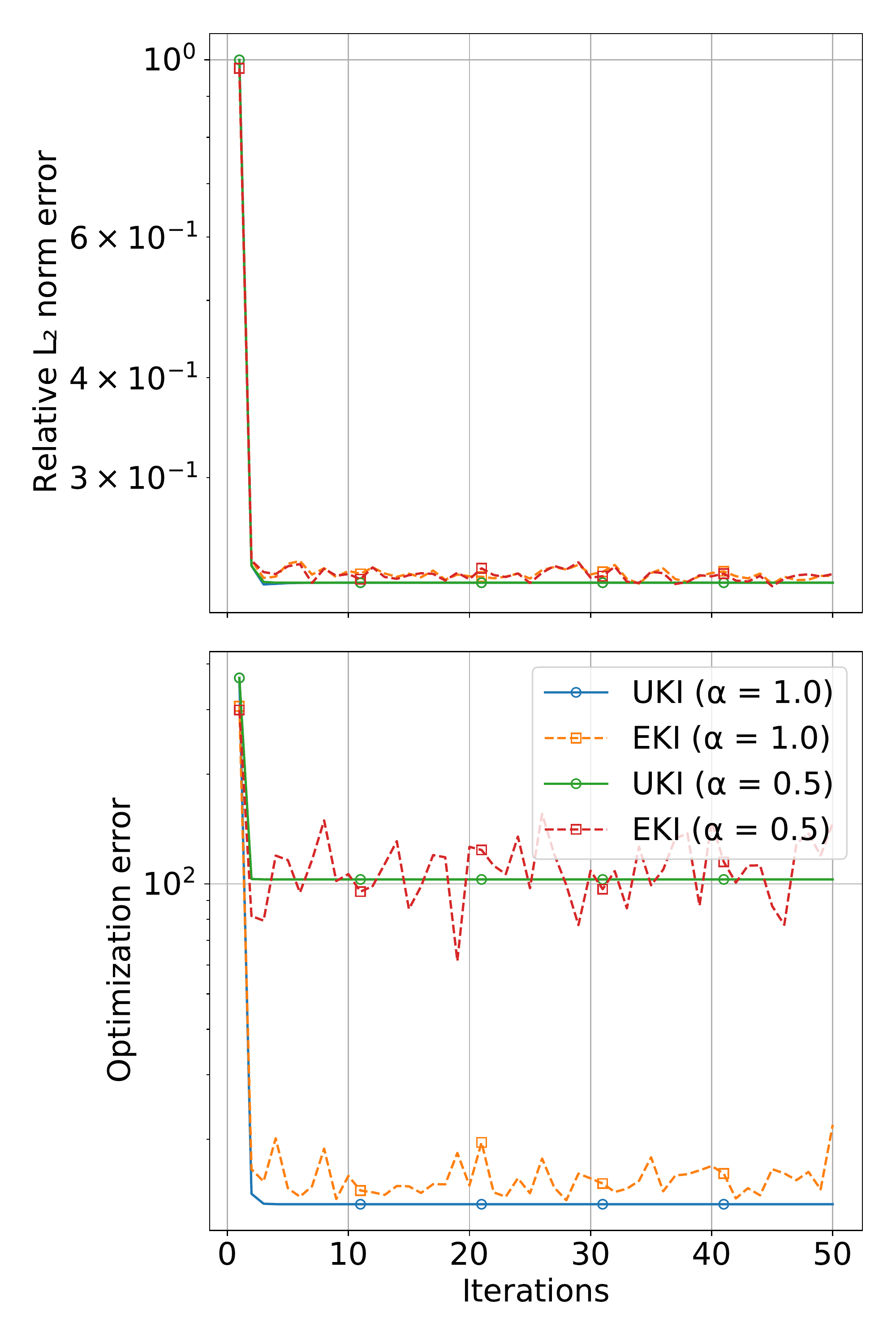}
    \includegraphics[width=0.32\textwidth]{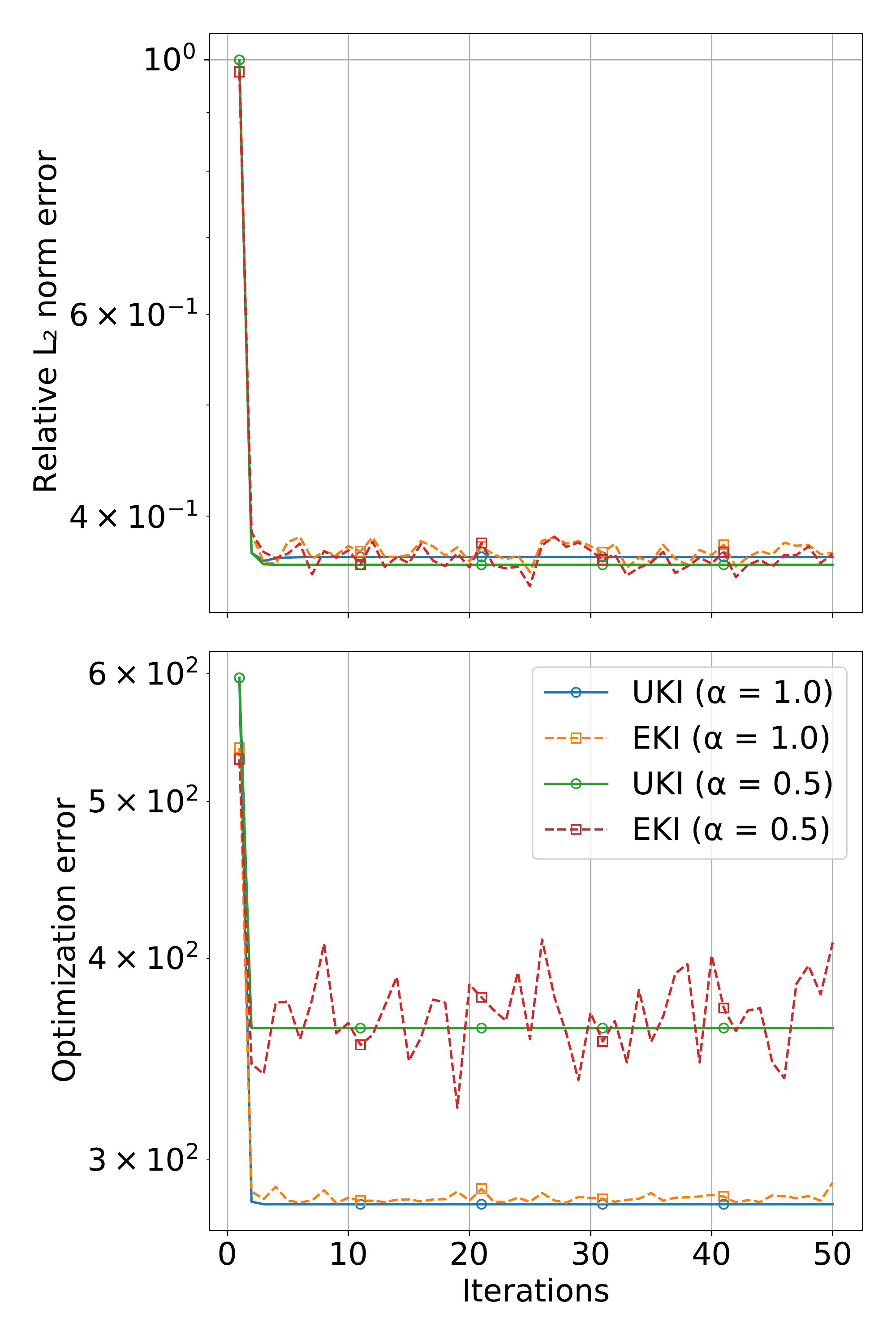}
    \caption{
     Relative error $\displaystyle \frac{\lVert\log a(x, \mean_n) - \log a_{ref}(x)\rVert_2}{\lVert\log a_{ref}(x)\rVert_2}$ (top) and the optimization error $\displaystyle \frac{1}{2}\lVert\Sigma_{\eta}^{-\frac{1}{2}} (y_{obs} - \py_n)\rVert^2$~(bottom) of the Darcy problem~($N_{\theta}=8$) with different noise levels: noiseless (left), $1\%$ error (middle), and $5\%$ error (right).}
    \label{fig:Darcy-8-converge}
\end{figure}

\begin{figure}[ht]
\centering
    \includegraphics[width=0.96\textwidth]{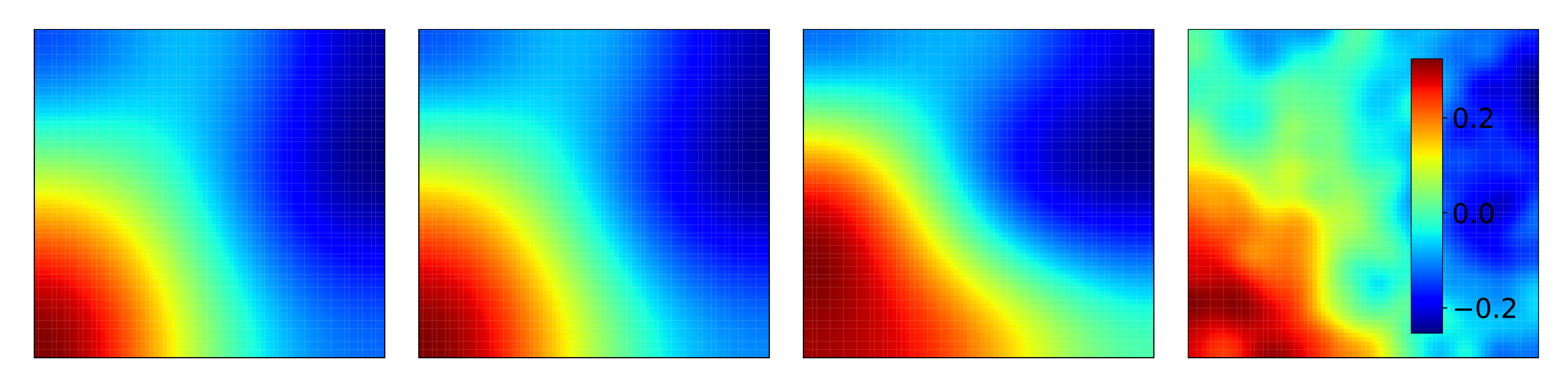}
    \caption{
    Log-permeability fields $\log a(x, \mean_n)$ with $N_{\theta}=8$ obtained by the UKI and the truth~(right) for different noise levels: noiseless~$\alpha=1$~(left), $1\%$ noise~$\alpha=1$~(middle-left), $5\%$ noise~$\alpha=1$~(middle-right).}
    \label{fig:Darcy-8}
\end{figure}

\subsection{Damage Detection Problem}
\label{sec:app:Damage}

Consider a thin linear elastic arch-like plate,  which is fixed on the bottom edges $\Gamma_u$. A traction boundary condition is applied on the top edge $\Gamma_{t_1}$, with distributed load $\bar{t} = (2, -20)$, and a traction free boundary condition is applied on the remaining edges $\Gamma_{t_2}$. See \cref{fig:Damage-obs}
The equations of linear elastostatics with plane stress assumptions are expressed in terms of the (Cauchy) stress tensor $\sigma$ and take the form
\begin{equation}
\label{eq:pde}
    \begin{split}
    \nabla \cdot \sigma + b &= 0 \textrm{ in } \Omega, \\
    u &= 0 \textrm{ on } \Gamma_{u}, \\
    \sigma \cdot n &= \bar{t} \textrm{ on } \Gamma_{t_1},\\
    \sigma \cdot n &= 0 \textrm{ on } \Gamma_{t_2}.
    \end{split}
\end{equation}
Here $u$ is the displacement vector, $b = 0$ is the body force vector, $\Omega\in \R^2$ is the bounded domain occupied by the plate. The strain tensor is
\begin{equation}
\varepsilon_{mn} = \frac{1}{2}\Big(\frac{\partial u_n}{\partial x_m} +  \frac{\partial u_m}{\partial x_n}\Big).
\end{equation}
The linear constitutive relation between strain and stress is written as
\begin{equation}
\sigma_{ij} = \mathsf{C}_{ijmn}(E, \nu) \varepsilon_{mn}.
\end{equation}
Here $\mathsf{C}_{ijmn} $ are the constitutive tensor components, which depend on the Young's modulus $E$ and Poisson's ratio $\nu$; throughout
this study, we fix $\nu=0.4$ and focus on learning the spatially-dependent
damage information present in the field $E.$
The damage is assumed to be isotropic elasticity-based damage with 
\begin{equation*}
    E(x, \theta) = \big(1 - \omega (x, \theta)\big) E_0.
\end{equation*}
Throughout this study, we fix $E_0 = 1000$, and $\omega (x, \theta)$ is the scalar-valued damage variable, which varies between zero (no damage)
to one (complete damage). 
The truth damage field~(See~\cref{fig:Damage-obs}-left) is 
\begin{equation*}
\begin{split}
&\omega_{ref} (x) = a_1e^{-\frac{1}{2}(x - x_1)\Sigma^{-1}_1(x - x_1)} + a_2e^{-\frac{1}{2}(x - x_2)\Sigma^{-1}_2(x - x_2)} + a_3e^{-\frac{1}{2}(x - x_3)\Sigma^{-1}_3(x - x_3)},\\
&a_1=0.8,\, a_2=0.6,\, a_3=0.5,\,  \\
&x_1 = \begin{bmatrix} 50\\50\end{bmatrix},\,
x_2 = \begin{bmatrix} 250\\160\end{bmatrix},\,
x_3 = \begin{bmatrix} 380\\100\end{bmatrix},\,
\Sigma_1 = \begin{bmatrix} 200&0\\0&200\end{bmatrix},\, 
\Sigma_2 = \begin{bmatrix} 800&0\\0&400\end{bmatrix},\, 
\Sigma_3 = \begin{bmatrix} 100&0\\0&400\end{bmatrix},\, 
\end{split}
\end{equation*} 
and may be seen to exhibit three flaws. Noise is added to the observations on the
boundary as in \eqref{eq:add-noise}. 
The forward equation is solved by the finite element method with $384$ quadratic quadrilateral elements~($1649$ nodes) using
the \texttt{NNFEM} library~\cite{huang2020learning,xu2020learning}.

\begin{figure}[ht]
\centering
\includegraphics[width=0.45\textwidth]{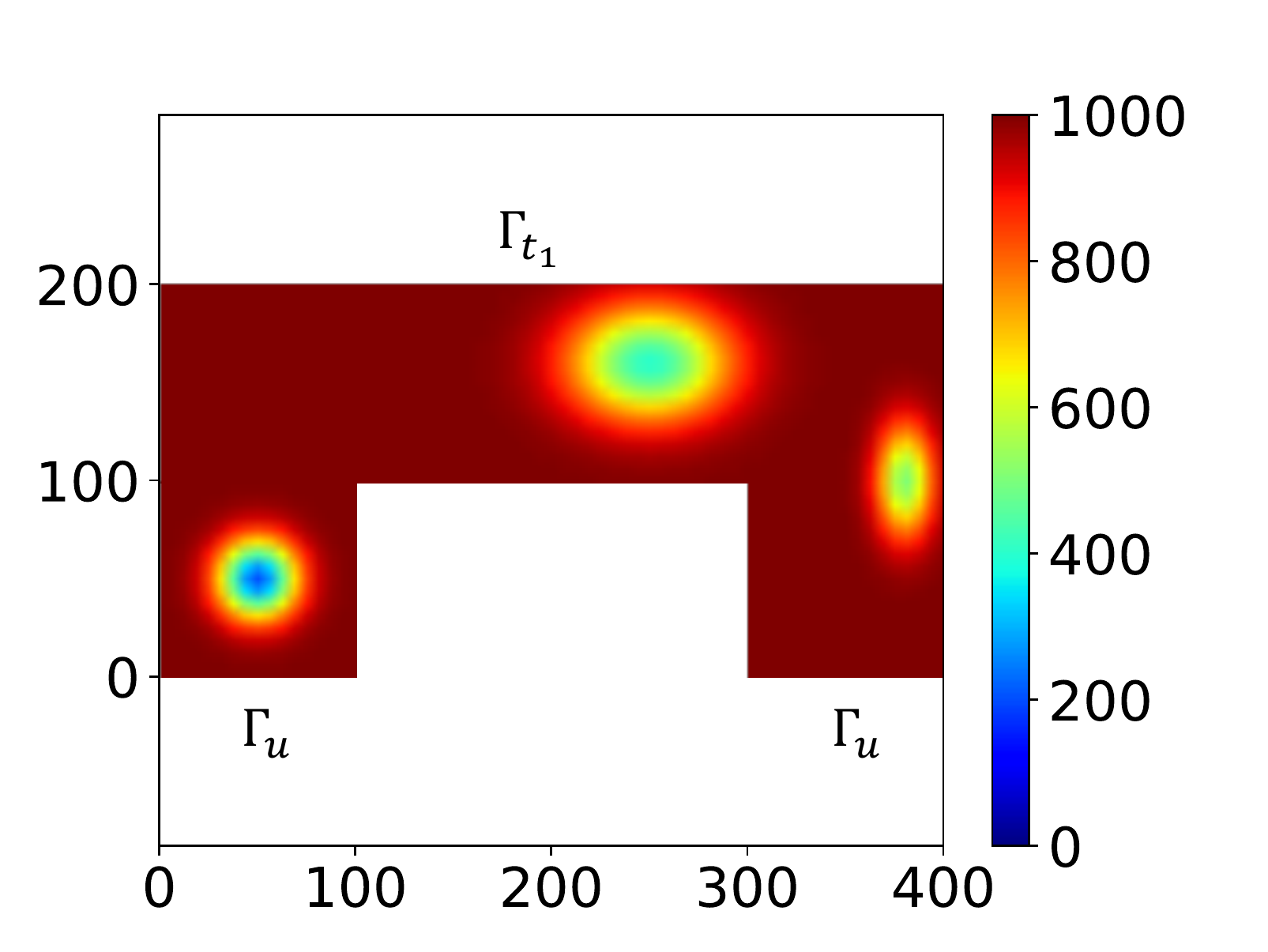}
\includegraphics[width=0.45\textwidth]{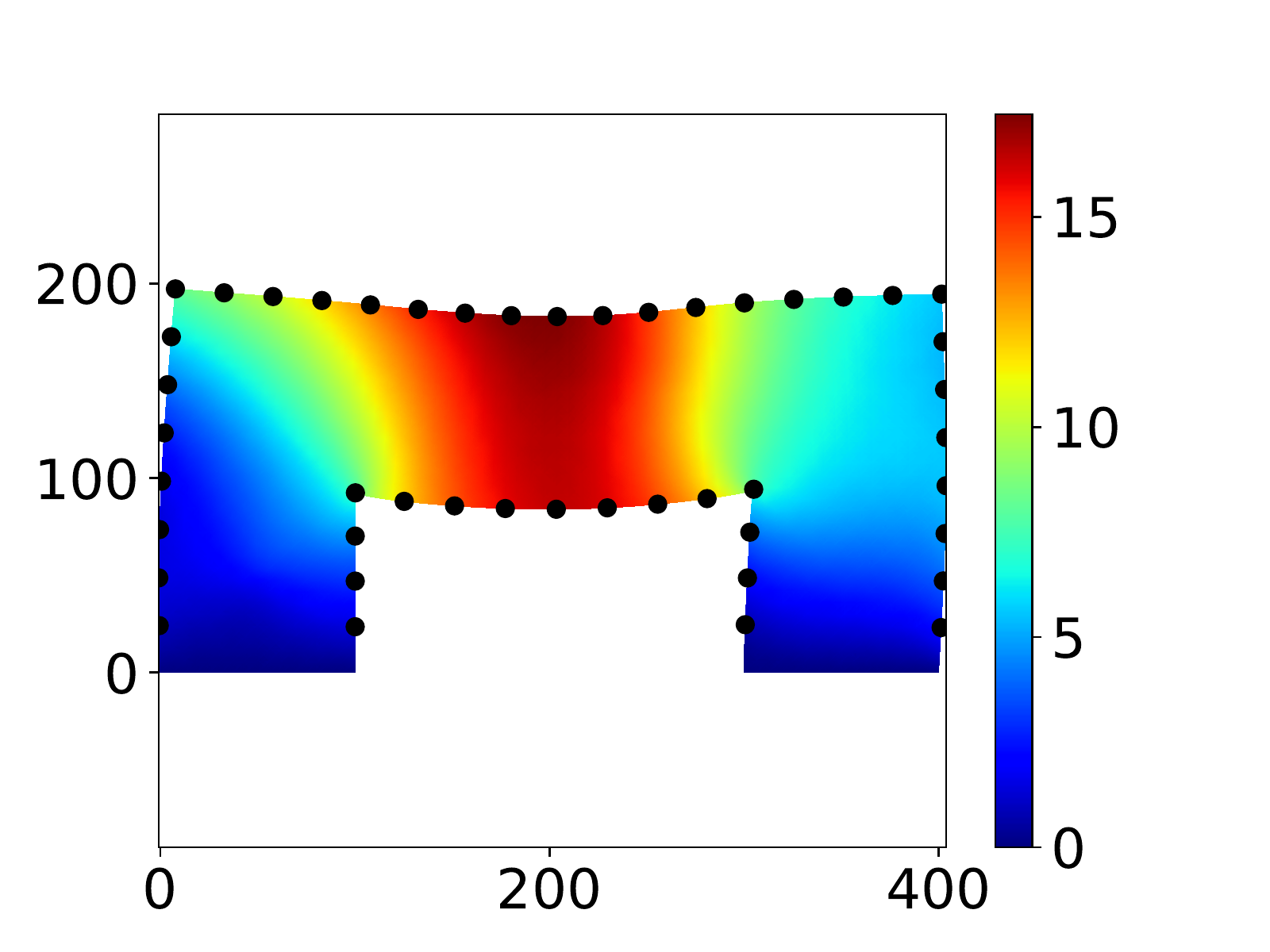}
\caption{The damaged Young's modulus~(left) and the displacement magnitude field~(right) with $46$ measurement locations on the surface of the boundaries~(black dots).  The five unlabelled edges comprise $\Gamma_{t_2}$; see equation \eqref{eq:pde}.}
\label{fig:Damage-obs}
\end{figure}

For the inverse problem, the damage field is parameterized in terms of field $\theta(x)$
as follows
\begin{equation*}
    \omega(\theta(x)) = 0.9 \frac{1 - e^{-\theta(x)}}{1 + 9e^{-\theta(x)}} \in (-0.1, 0.9).
\end{equation*}
Field $\theta(x)$ is itself discretized and represented by  
24 quadratic quadrilateral elements~($N_{\theta} = 125$)\footnote{It is worth mentioning that increasing the parameter dimensionality by refining the parameter mesh exacerbates the
ill-posedness and, therefore, deteriorates the performance of both Kalman inversions.}. 
The observations are $x_1$ and $x_2$ displacements measured at 46~($N_y=92$) locations on the surface boundaries~(see \cref{fig:Damage-obs}-right). We consider both $\alpha=0.5$ and
$\alpha=1.0$, and we set $r_0 = 0$ and $\gamma=1.$
The UKI and EKI are both applied, initialized with  $\theta_0 \sim \N(0, \I)$. The observation error model used in the algorithm is $\eta \sim \N(0, 0.1^2\I)$. For this problem the prior information  $\omega(\theta=0) = 0$ corresponds to an undamaged plate,
and is expected to be reasonable for most of the domain. For the EKI, the ensemble size is set to $J = 500$, which is larger than the 
number of $\sigma-$points used in UKI~($2N_\theta+1$). 

The convergence of the damage field $\omega(\theta(x, m_n))$ and the optimization errors
at each iteration are depicted in \cref{fig:Damage-converge}; the organization of the
information is the same as in the Darcy flow example.
In the noiseless scenario, the EKI exhibits divergence without regularization~($\alpha=1.0$) due to the ill-posedness, however, the UKI converges 
\footnote{We will see the same phenomenon in Subsection~\ref{sec:app:NS}.}.
For noisy scenarios, the effect of overfitting is significant. 
At $1\%$ noise level, setting $\alpha=0.5$ eliminates overfitting; however at $5\%$ noise level, setting $\alpha=0.5$ does not eliminate overfitting. Therefore, the results obtained with $\alpha=0.0$ are also reported for the  $5\%$ noise scenario. The estimated damaged Young's modulus fields $E(x, \theta)$ and the truth are depicted in \cref{fig:Damage}. Both Kalman inversion methods perform comparably, and these three flaw areas are captured;
however at $5\%$ noise level noticeable bias is visible in the flaws to the left and
right of the domain. As in the Darcy flow case, the convergence histories of the UKI
are smoother than for the EKI.

\begin{figure}[ht]
\centering
    \includegraphics[width=0.32\textwidth]{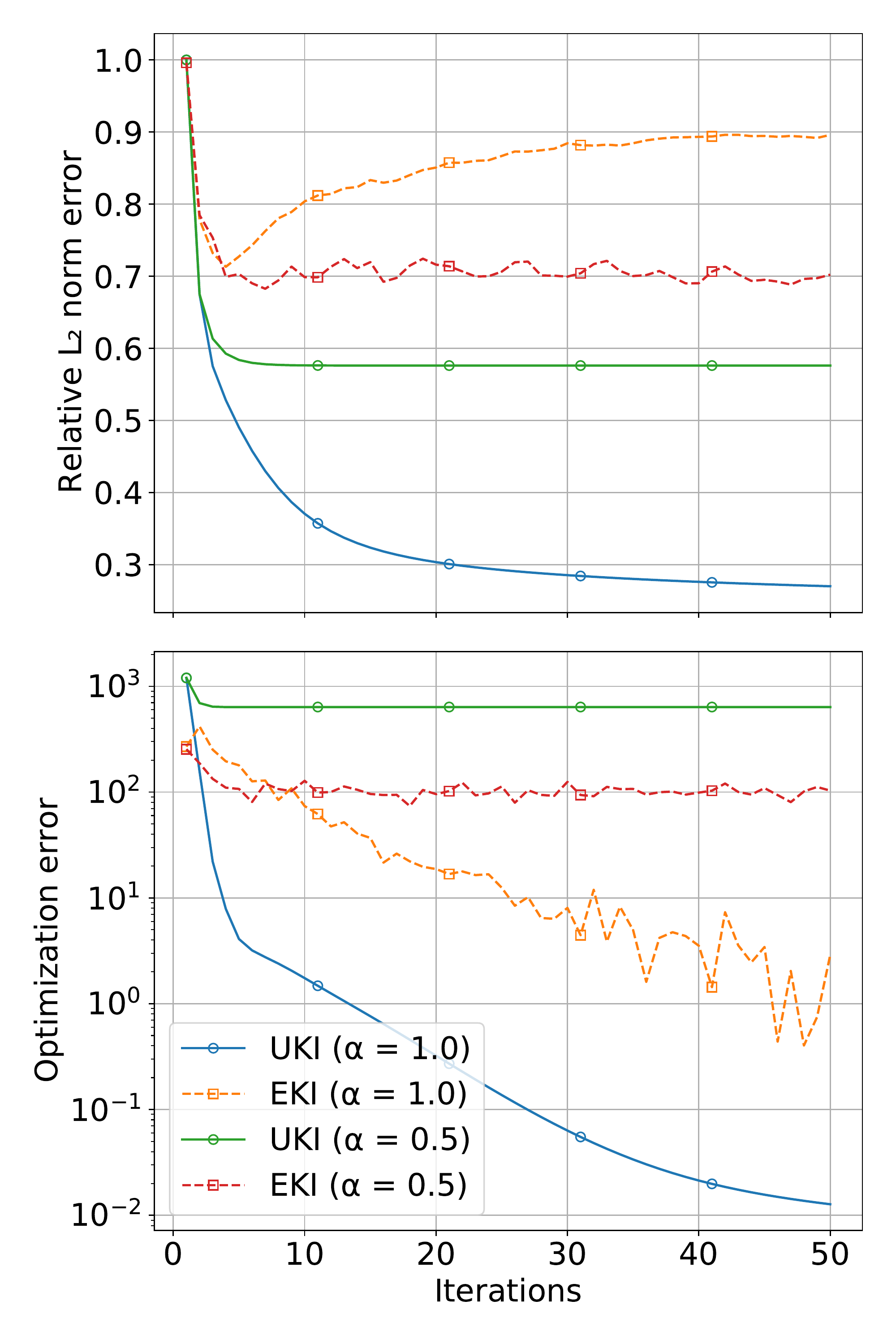}
    \includegraphics[width=0.32\textwidth]{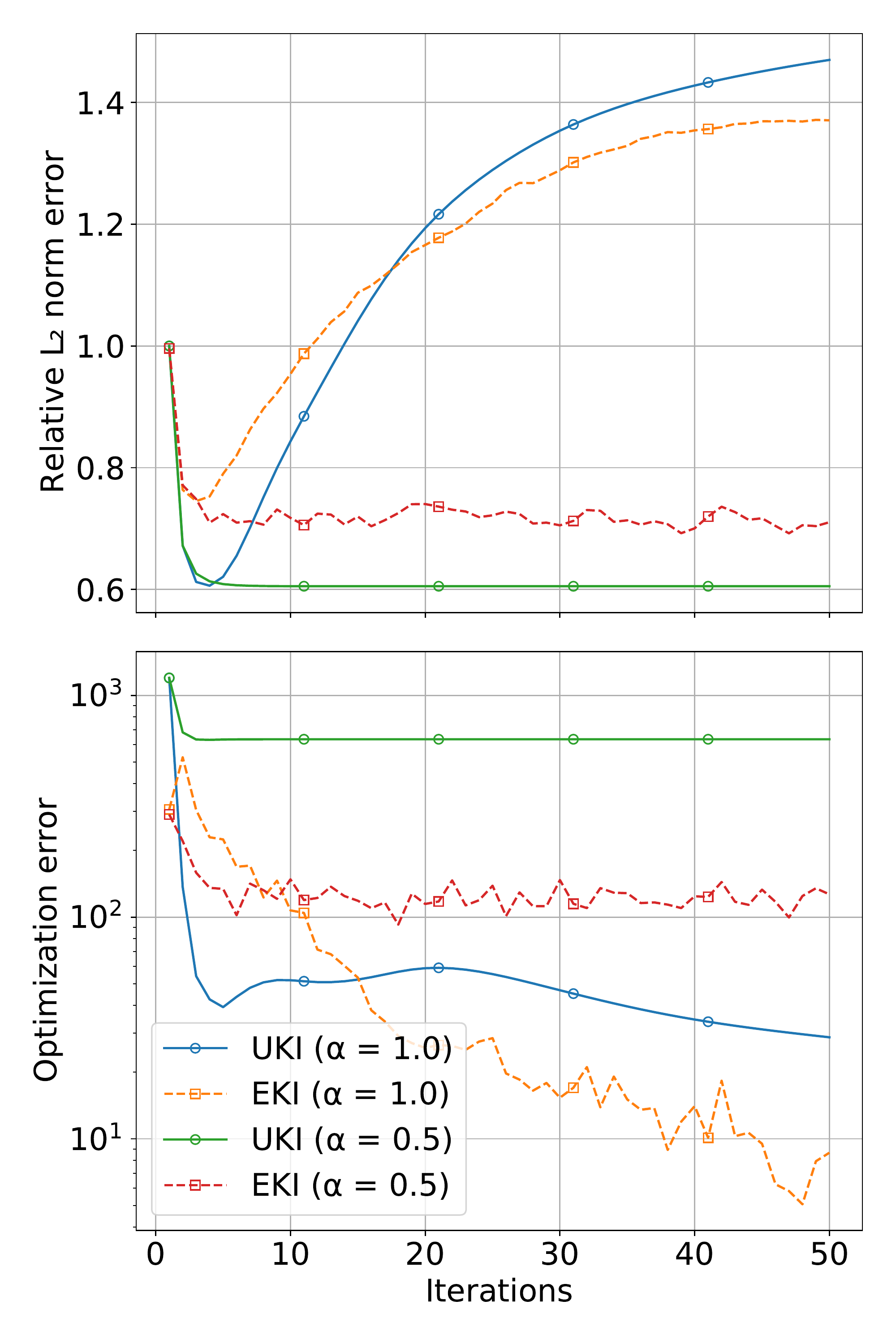}
    \includegraphics[width=0.32\textwidth]{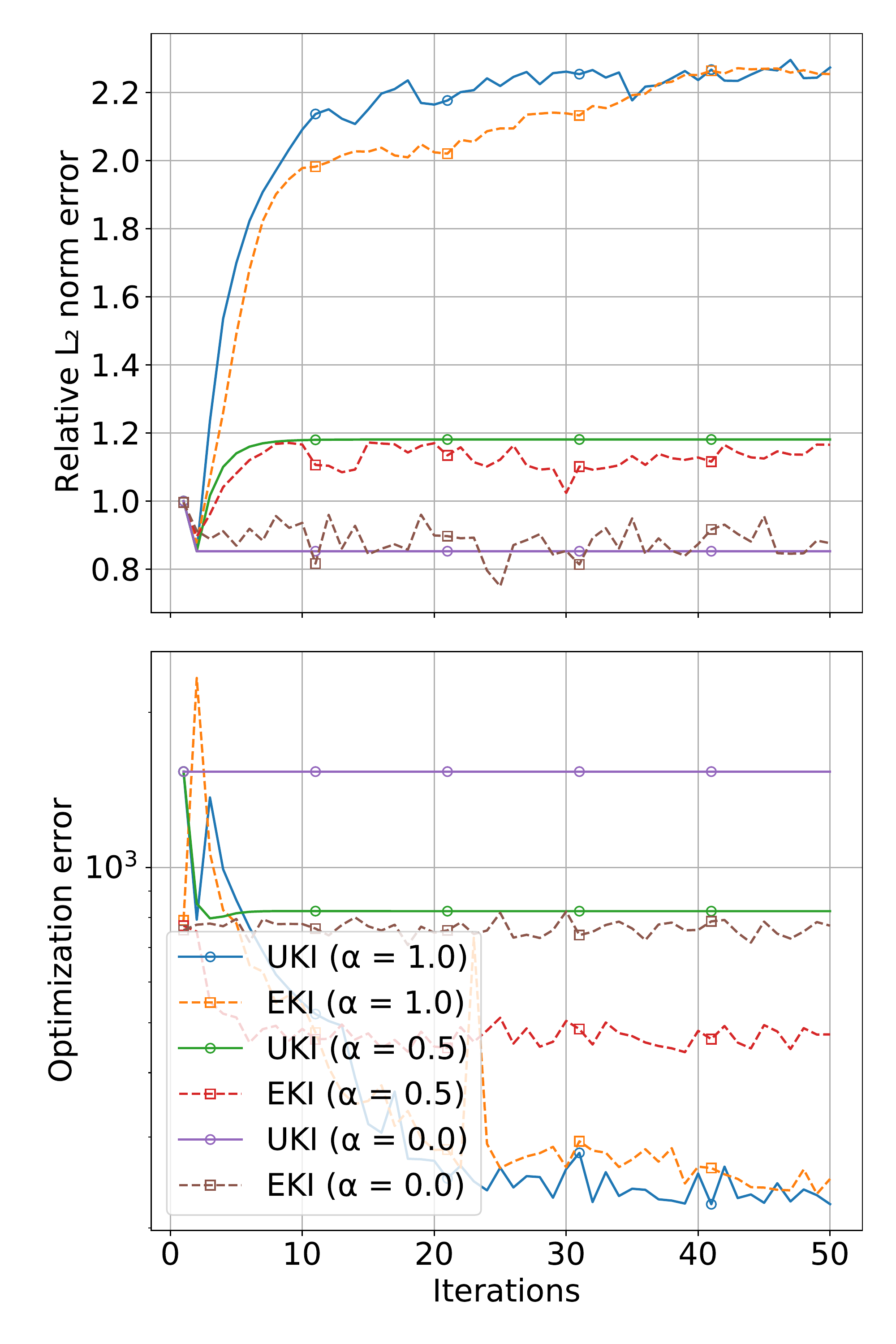}
    \caption{
    Relative error $\displaystyle \frac{\lVert\omega(\theta(x, m_n)) - \omega_{ref}\rVert_2}{\lVert\omega_{ref}\rVert_2}$ (top) and the optimization error $\displaystyle \frac{1}{2}\lVert\Sigma_{\eta}^{-\frac{1}{2}} (y_{obs} - \py_n)\rVert^2$~(bottom) of the damage detection problem with different noise levels: noiseless~(left), $1\%$ error~(middle), and~$5\%$ error (right).}
    \label{fig:Damage-converge}
\end{figure}

\begin{figure}[ht]
\centering
    \includegraphics[width=0.96\textwidth]{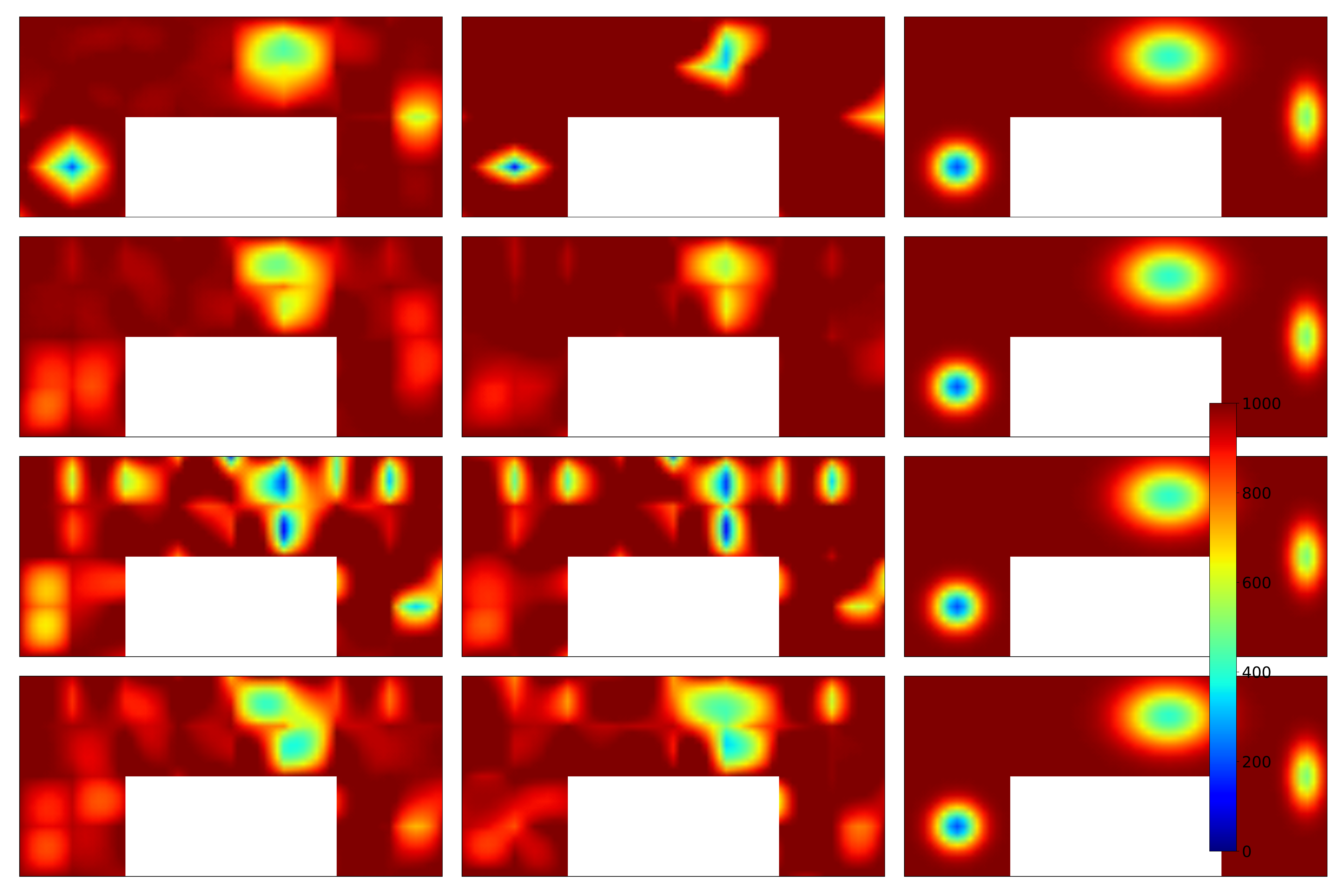}
    \caption{
    Damaged Young's modulus fields $(1 - \omega(x, \mean_n))E_0$ obtained by UKI, EKI, and the truth~(left to right) at different noise levels: noiseless~$\alpha=1$, $1\%$ noise~$\alpha=0.5$, $5\%$ noise~$\alpha=0.5$, and $5\%$ noise~$\alpha=0$~(top to bottom).}
    \label{fig:Damage}
\end{figure}

\subsection{Navier-Stokes Problem}
\label{sec:app:NS}

We consider the 2D Navier-Stokes equation on a periodic domain $D = [0,2\pi]\times[0,2\pi]$:
\begin{equation*}
\begin{split}
    &\frac{\partial v}{\partial t} + (v\cdot \nabla) v + \nabla p - \nu\Delta v = 0, \\
    &\nabla \cdot v = 0, \\
\end{split}
\end{equation*}
with initial condition chosen to imply the conservation law
$$\frac{1}{4\pi^2}\int v  = v_b.$$
Here $v$ and $p$ denote the velocity vector and the pressure, $\nu=0.01$ denotes the dynamic viscosity, and $v_b = (2\pi, 2\pi)$ denotes the non-zero mean background velocity.
The forward problem is rewritten in the vorticity-streamfunction~($\omega-\psi$) formulation:
\begin{equation*}
\begin{split}
    &\frac{\partial \omega}{\partial t} + (v\cdot\nabla)\omega - \nu\Delta\omega = 0, \\
    &\omega = -\Delta\psi \qquad \frac{1}{4\pi^2}\int\psi = 0,\\
    &v = \Big(\frac{\partial \psi}{\partial x_2}, -\frac{\partial \psi}{\partial x_1}\Big) + v_b,
\end{split}
\end{equation*}
and solved by the pseudo-spectral method~\cite{hesthaven2007spectral} on a $128\times128$ grid. To eliminate aliasing error, the Orszag 2/3-Rule~\cite{orszag1972numerical} is applied and, therefore there are $85^2$ Fourier modes (padding with zeros). Time-integration is 
performed using the Crank–Nicolson method with $\Delta T=2.5\times 10^{-4}$. 

We study the problem of recovering the initial vorticity field from measurements
at positive times. We parameterize this field as $\omega_0(x, \theta)$, defined by parameters $\theta \in \R^{N_{\theta}}$, and modeled {\em a priori} as a Gaussian field with covariance operator $\mathsf{C} = \Delta^{-2}$, subject to periodic boundary conditions, on the space of spatial-mean zero functions. The KL expansion of the initial vorticity field is given by 
\begin{equation}
\label{eq:NS-KL-2d}
\omega_0(x, \theta) = \sum_{l\in K} \theta^{c}_{(l)} \sqrt{\lambda_{l}} \psi^c_l  +  \theta^{s}_{(l)}\sqrt{\lambda_{l}} \psi^s_l,
\end{equation}
where $K = \{(k_x, k_y)| k_x + k_y > 0 \textrm{ or } (k_x + k_y = 0 \textrm{ and } k_x > 0)\}$, and the eigenpairs are of the form
\begin{equation*}
    \psi^c_l(x) =\frac{\cos(l\cdot x)}{\sqrt{2}\pi}\quad \psi^s_l(x) =\frac{\sin(l\cdot x)}{\sqrt{2}\pi} \quad \lambda_l = \frac{1}{|l|^{4}},
\end{equation*}
and $\theta^{c}_{(l)},\theta^{s}_{(l)}  \sim \N(0,2\pi^2)$ i.i.d. The KL expansion~\cref{eq:NS-KL-2d} can be rewritten as a sum over $\Z^{0+}$ rather than a lattice: 
\begin{equation}
\label{eq:NS-KL-1d}
    \omega_0(x,\theta) = \sum_{k\in \Z^{0+}} \theta_{(k)}\sqrt{\lambda_k} \psi_k(x),
\end{equation}
where the eigenvalues $\lambda_k$ are in descending order.

For the inverse problem, we recover the initial condition, specifically the initial vorticity field of the Navier-Stokes equation, given pointwise observations $y_{ref}$ of the vorticity field at 16 equidistant points~($N_y=32$) at $T=0.25$ and $T=0.5$~(See \cref{fig:NS-obs}).
The observations $y_{obs}$ are defined as in \eqref{eq:add-noise}.
The initial vorticity field $\omega_{0,ref}$ is generated with all $85^2$ Fourier modes, and the first $N_{\theta}=100$ KL modes of~\cref{eq:NS-KL-1d} are recovered. We
take $\alpha=1.0$ and $\alpha=0.9$, and fix $r_0 = 0$ and $\gamma=10$.
Both UKI and EKI are applied  with  $\theta_0 \sim \N(0, 10\I)$ and the observation error assumed for inversion purposes is $\eta \sim \N(0, \I)$. For the EKI, the ensemble size is set to be $J = 201$, which equals the number of $\sigma-$points in  UKI~($2N_\theta+1$).

\begin{figure}[ht]
\centering
\includegraphics[width=0.45\textwidth]{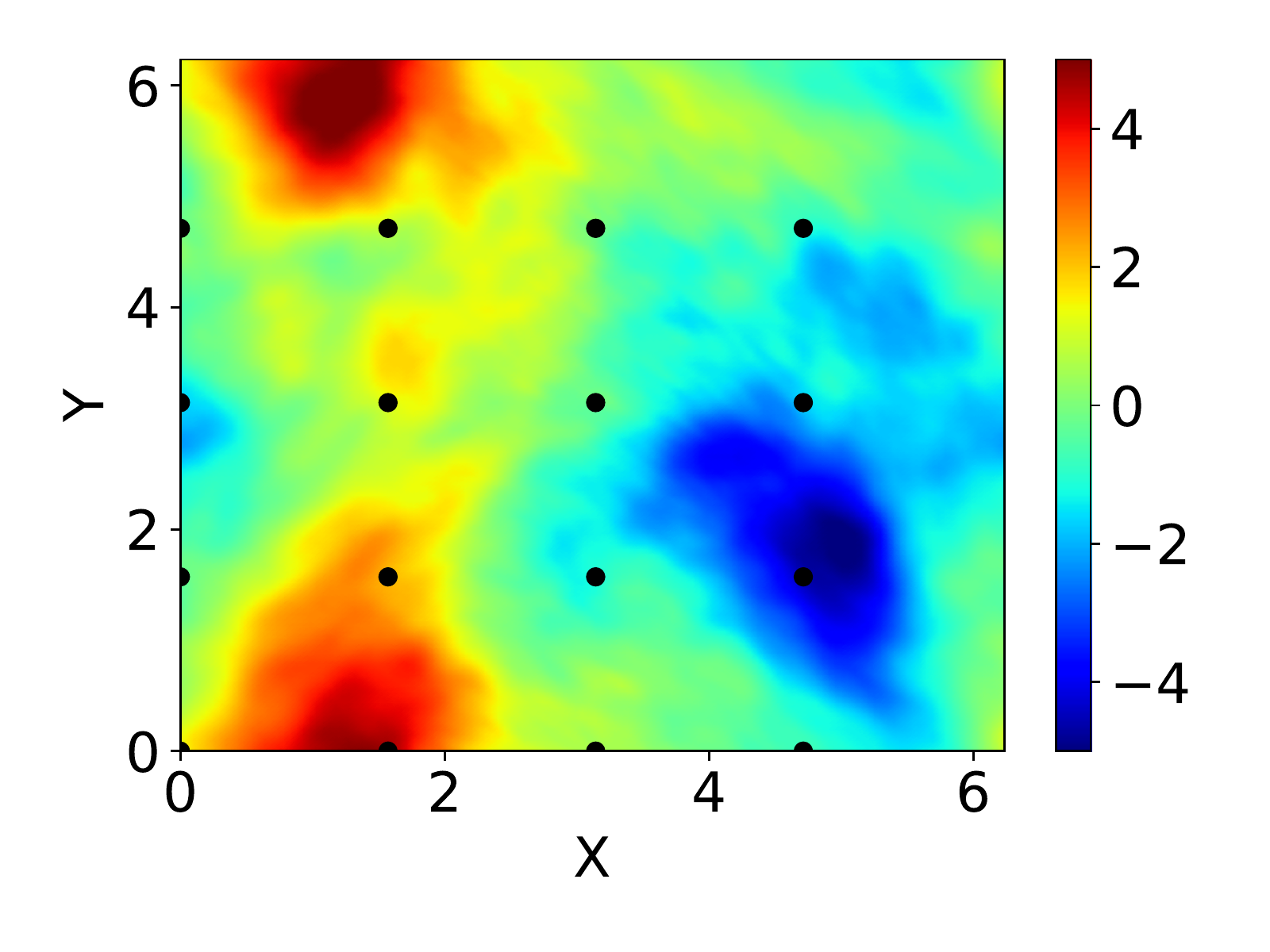}
\includegraphics[width=0.45\textwidth]{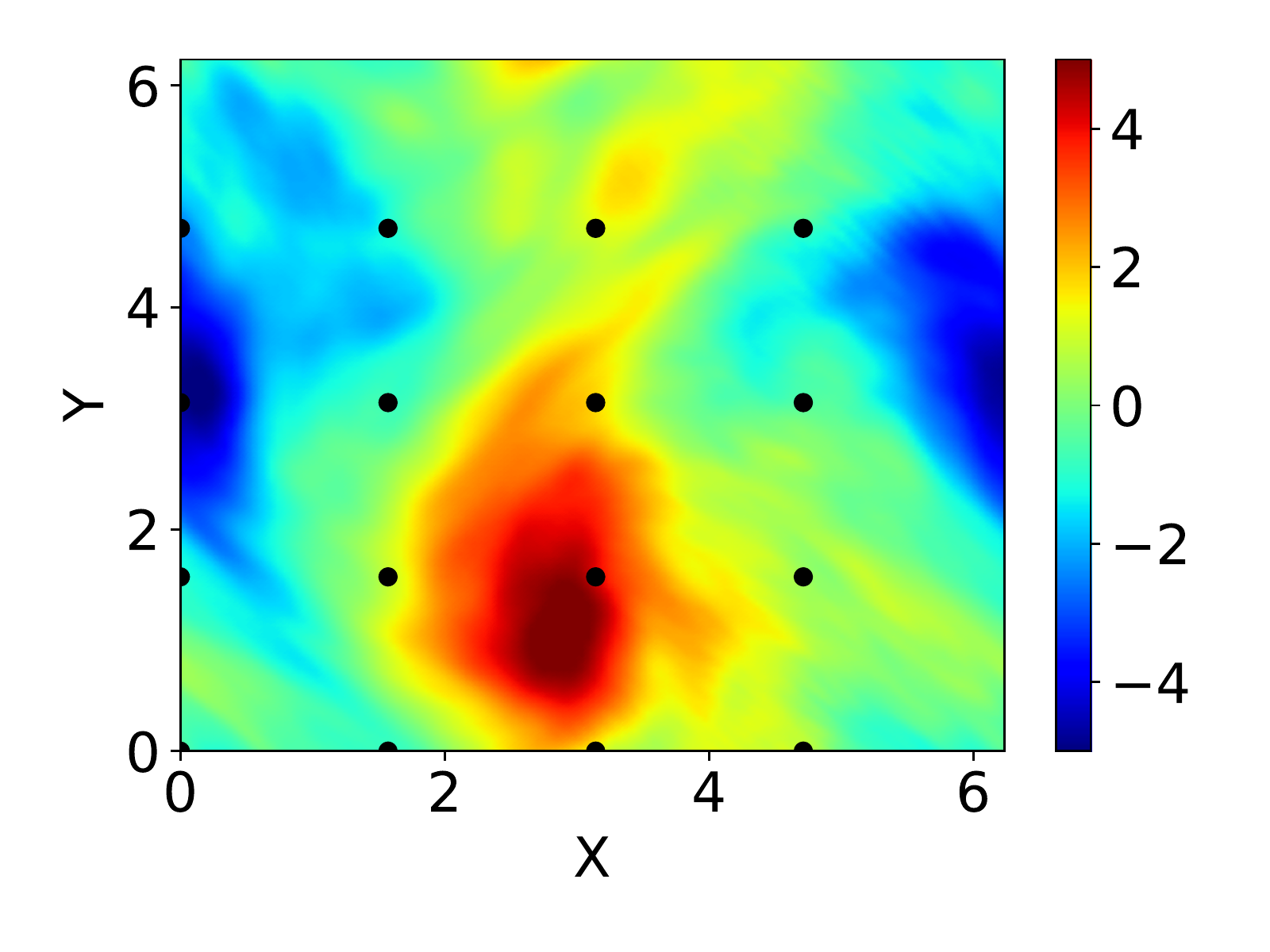}
\caption{The vorticity fields of the Navier-Stokes problem and the $16$ equidistant pointwise measurements~(black dots) at two observation times~($T=0.25$ and $T=0.5$).}
\label{fig:NS-obs}
\end{figure}

The convergence of the initial vorticity field $\omega_0(x, \mean_n)$ and the optimization errors
for different noise levels at each iteration are depicted in \cref{fig:NS-converge}; the
organization of the figure is the same as in the Darcy case.
In all scenarios, the UKI outperforms EKI. Moreover, without regularization~($\alpha=1.0$), EKI exhibits slight divergence. 
This inverse problem is not sensitive to added Gaussian random noise, and the behavior of any given Kalman inversion, with respect to different noise levels, are almost indistinguishable.
The estimated initial vorticity fields $\omega_0(x,m_n)$ at the 50th iteration for different noise levels obtained by the Kalman inversions and the truth random field are depicted in \cref{fig:NS}. Both Kalman inversions capture main features of the truth random initial field, but not the detailed small features, due to the irreversibility of the diffusion process~($\nu=0.01$).

\begin{figure}[ht]
\centering
    \includegraphics[width=0.32\textwidth]{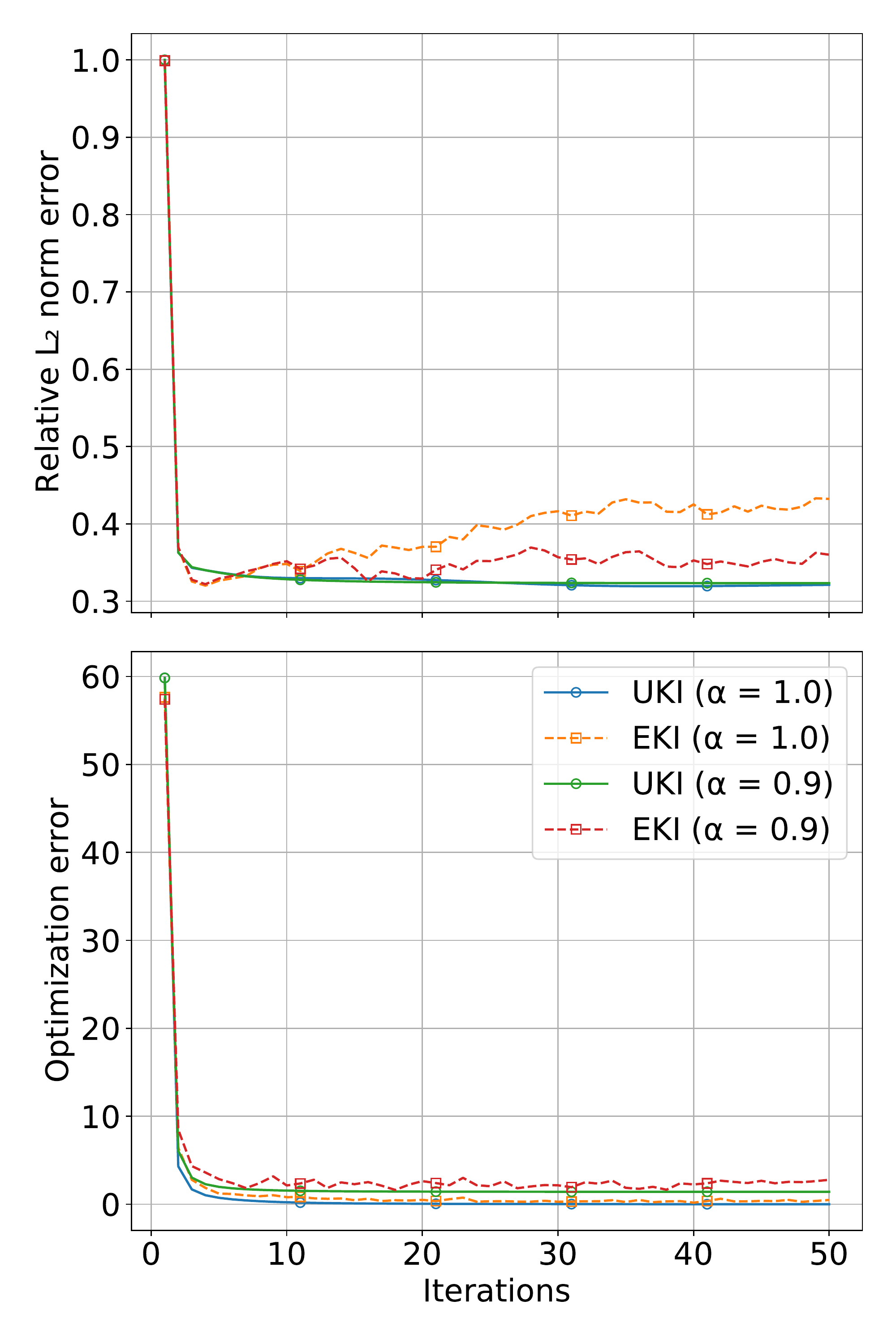}
    \includegraphics[width=0.32\textwidth]{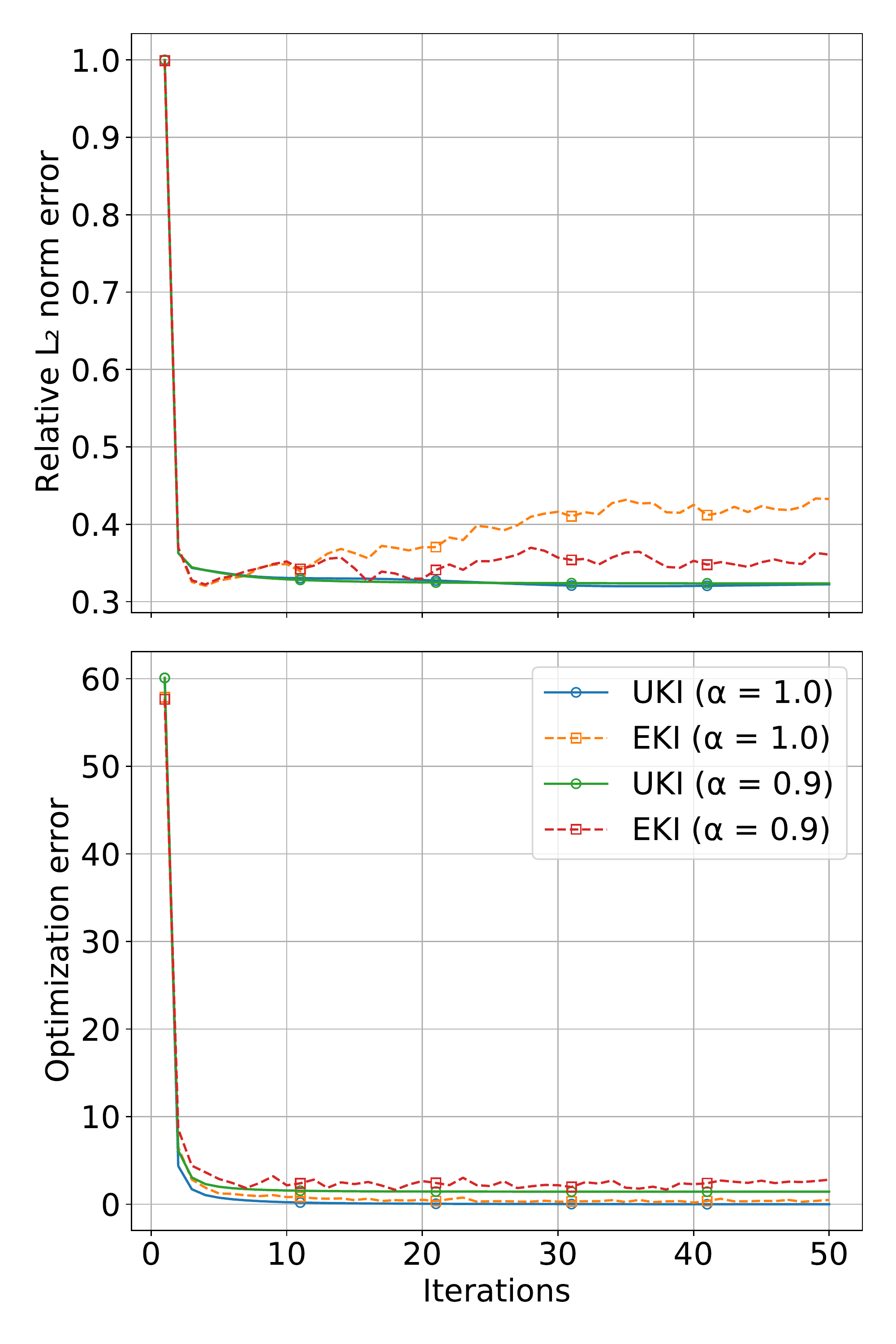}
    \includegraphics[width=0.32\textwidth]{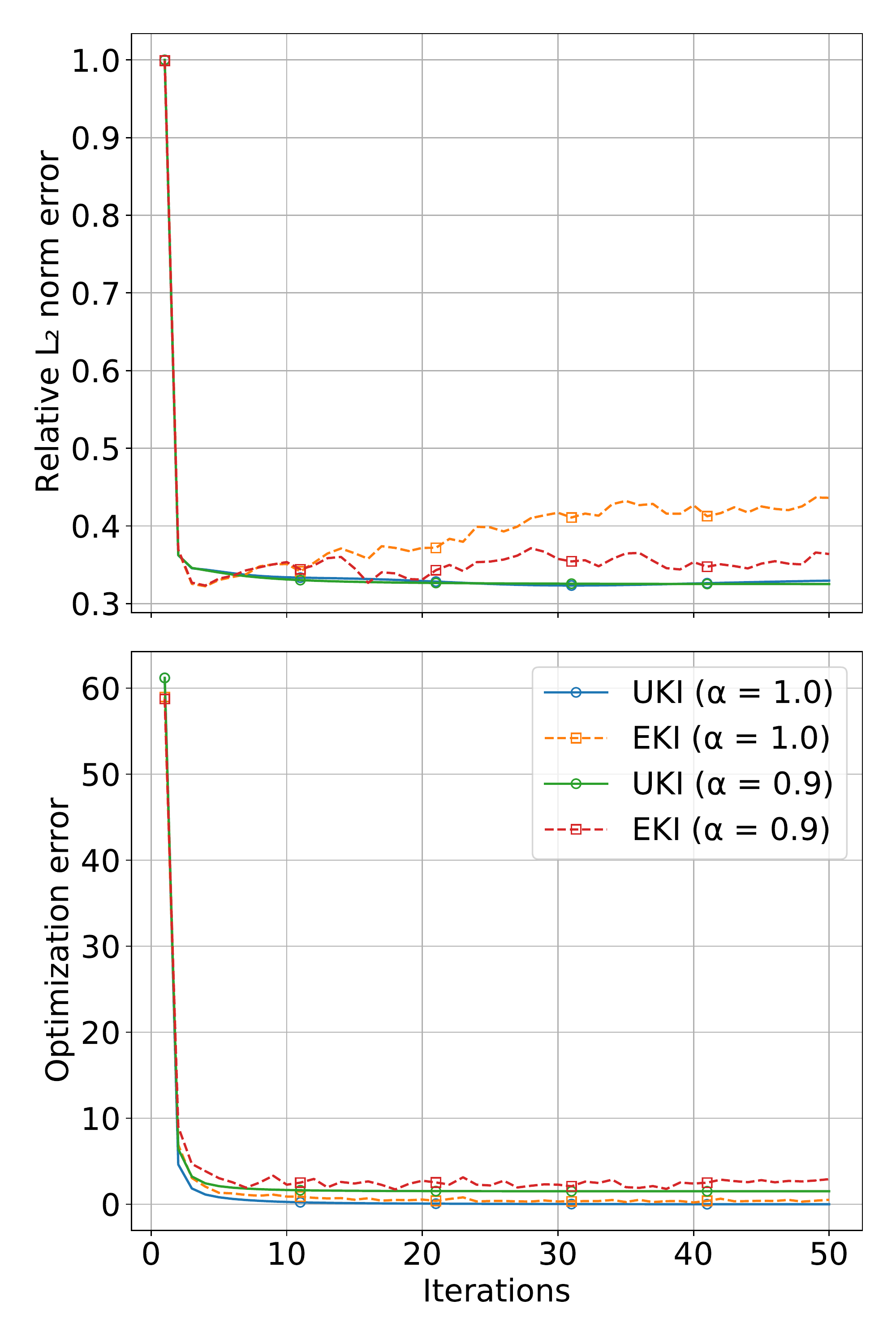}
    \caption{
    Relative error $\frac{\lVert\omega_0(x,\mean_n) - \omega_{0,ref}\rVert_2}{\lVert \omega_{0,ref}\rVert_2}$ (top) and the optimization error $\displaystyle \frac{1}{2}\lVert\Sigma_{\eta}^{-\frac{1}{2}} (y_{obs} - \py_n)\rVert^2$~(bottom) of the Navier-Stokes problem with different noise levels: noiseless~(left), $1\%$ error~(middle), and $5\%$ error~(right).}
    \label{fig:NS-converge}
\end{figure}

\begin{figure}[ht]
\centering
    \includegraphics[width=0.72\textwidth]{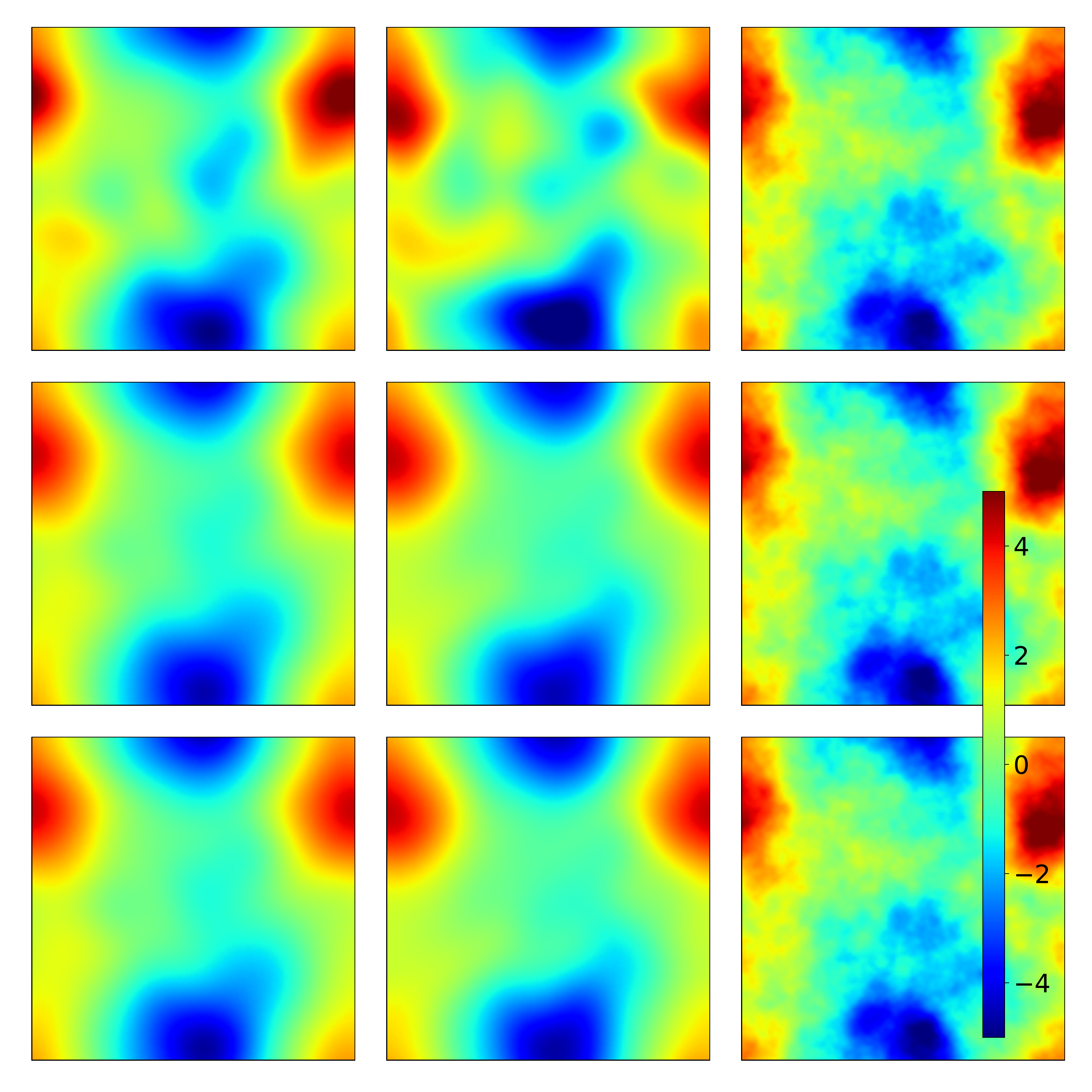}
    \caption{
    Initial vorticity fields $\omega_0(x, \mean_n)$ recovered by UKI, EKI, and the truth~(left to right) for different noise levels: noiseless~$\alpha=1$, $1\%$ noise~$\alpha=0.9$, $5\%$ noise~$\alpha=0.9$~(top to bottom).}
    \label{fig:NS}
\end{figure}

\subsection{Lorenz63 Model Problem}
\label{sec:app:Lorenz63}
Consider the Lorenz63 system, a simplified mathematical model for atmospheric convection~\cite{lorenz2004deterministic}:
\begin{equation*}
\begin{split}
    &\frac{dx_1}{dt} = \sigma (x_2 - x_1), \\
    &\frac{dx_2}{dt} = x_1(r-x_3)-x_2,        \\
    &\frac{dx_3}{dt} = x_1x_2-\beta x_3;   \\
\end{split}
\end{equation*}
the system is parameterized by $\sigma, r, \beta \in \R_{+}$.
We consider learning various subsets of these parameters from time-averaged data.
To be concrete, the observation consists of the time-average of the various moments over time windows of size $T = 20$, with an initial spin-up period $T=30$ to eliminate the influence of the initial condition; if $f:\R^3 \mapsto \R$ computes a moment, then we define
\begin{equation}
\label{eq:timea}
\displaystyle \overline{f(x)} = \frac{1}{20}\int_{30}^{50} f\bigl(x(t)\bigr) dt.
\end{equation}
We view this as an approximation of the ergodic average
$$\E f(x)= \lim_{\tau \to \infty} \frac{1}{\tau}\int_{0}^{\tau} f\bigl(x(t)\bigr) dt.$$
If the observation operator comprises finite time averages of the form \eqref{eq:timea} for a collection of moments $f(x)$ then we may reformulate the inverse problem as
\begin{equation}
\label{eq:Lorenz-gen}
    y = \G(r) + \eta 
\end{equation}
with $\eta$ a Gaussian which may be estimated from a long time trajectory
(we use $T=200$) by  appealing to the central limit theorem
\cite{bahsoun2020variance}. In this interpretation $\G$ is the ergodic average. Note, however,
that when we run any algorithm we will only use finite-time average approximations of
$\G$.

The truth observation is computed with parameters $(\sigma,\ r,\ \beta) = (10,\ 28,\ 8/3)$ over a time window of size $T= 200$, also with an initial spin-up period $T=30$. 
To estimate the statistics of $\eta$ we split the observation time-series into 
$10$ windows of size $T=20$ and compute covariance of the observation error $\eta$ following~\cite{cleary2020calibrate}. We set $r_0 = 5.0\mathds{1}$ and $\gamma=1$.
The UKI is initialized with $\theta_0 \sim \N(5.0\mathds{1},\I)$, and $\alpha$ is set to $1.$

We start with the following one-parameter inverse problem with fixed $\sigma=10$ and $\beta=8/3$:
\begin{equation}
\label{eq:Lorenz-1}
    y = \G(r) + \eta \quad \mathrm{with} \quad y= \overline{x_3}.
\end{equation}
The UKI is applied, and the estimated $r$ and the associated  3-$\sigma$ confidence intervals at each iteration are depicted in \cref{fig:Loroze63-1para}. The confidence intervals
give an indication of the evolving covariance $C_n$. The estimation of $r$ at the 20th iteration is $r\sim \N(28.03, 0.22)$.

\begin{figure}[ht]
\centering
\includegraphics[width=0.6\textwidth]{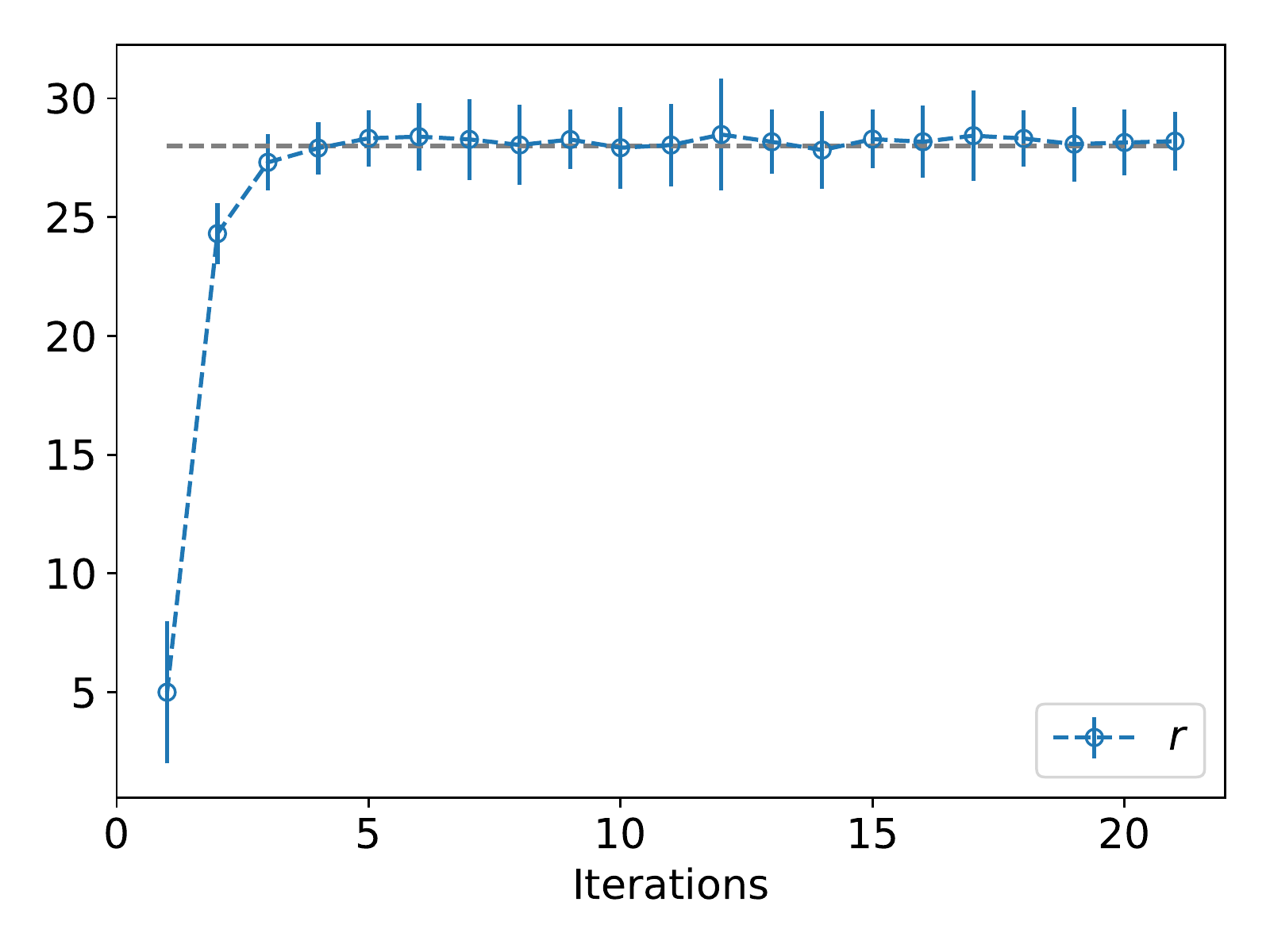}
\caption{Convergence of the 1-parameter Lorenz63 inverse problem with UKI~($\alpha=1.0$); the true parameter value is represented by the dashed grey line.}
\label{fig:Loroze63-1para}
\end{figure}

The landscape of $\G$ and sensitivity of $\G(\cdot)$ with respect to the input for
observations, derived from chaotic problems such as \cref{eq:Lorenz-1}, are widely studied~\cite{lea2000sensitivity, wang2014least}. 
We study them further, here, and the results are depicted 
in \cref{fig:Loroze63-func1}. The function $\G$ is characterized by a sudden change at $r \approx 22$ and the landscape is highly oscillatory for $r>22$; furthermore,
the sensitivity $d\G(r)$ computed with the discrete adjoint method blows up: 
\begin{equation*}
    |d\G(r)| \propto \bigO(e^{\lambda T}),
\end{equation*}
with the value of the exponent $\lambda$ consistent with the first global Lyapunov exponent~\cite{lea2000sensitivity,froyland1984lyapunov}.
This illustrates the challenges inherent in parameter estimation and sensitivity analyses
for chaotic systems. In particular, the ExKI method suffers from the large derivatives
of $\G$.
Based on~\cref{th:filter}, it is natural to study
the landscape of the averaged function $\F\G$ and its associated gradient $\F d\G$, with the standard deviation $\sigma_r = \sqrt{0.22}$ fixed; this
gives an indication of the landscape as perceived by the UKI. In particular, we have:
\begin{equation*}
\begin{split}
    \F\G(r) = \int \G(x) \frac{1}{\sqrt{2\pi}\sigma_r}e^{-\frac{(x -r)^2}{2\sigma_r^2}} dx, \qquad
    \F d\G(r) = \frac{\int (x-r)(\G(x) - \G(r)) \frac{1}{\sqrt{2\pi}\sigma_r}e^{-\frac{(x -r)^2}{2\sigma_r^2}} dx}{\int(x-r)^2 \frac{1}{\sqrt{2\pi}\sigma_r}e^{-\frac{(x -r)^2}{2\sigma_r^2}} dx}.
\end{split}
\end{equation*}
These functions are depicted in \cref{fig:Loroze63-smoothfunc1}, which should be
compared with \cref{fig:Loroze63-func1}. We see that $\F\G$ is smooth (except the transition point), and $\F d\G$ does not suffer from blow-up in the way $d\G$ does;
furthermore, $\F d\G$ represents the averaged gradient $ \overline{d\G(r)} \approx 0.96$ well,
away from the blow-up regions. This explains why the adjoint/gradient-based methods, including ExKI, fail, but the UKI succeeds for this chaotic inverse problem.

\begin{figure}[ht]
\centering
\includegraphics[width=0.49\textwidth]{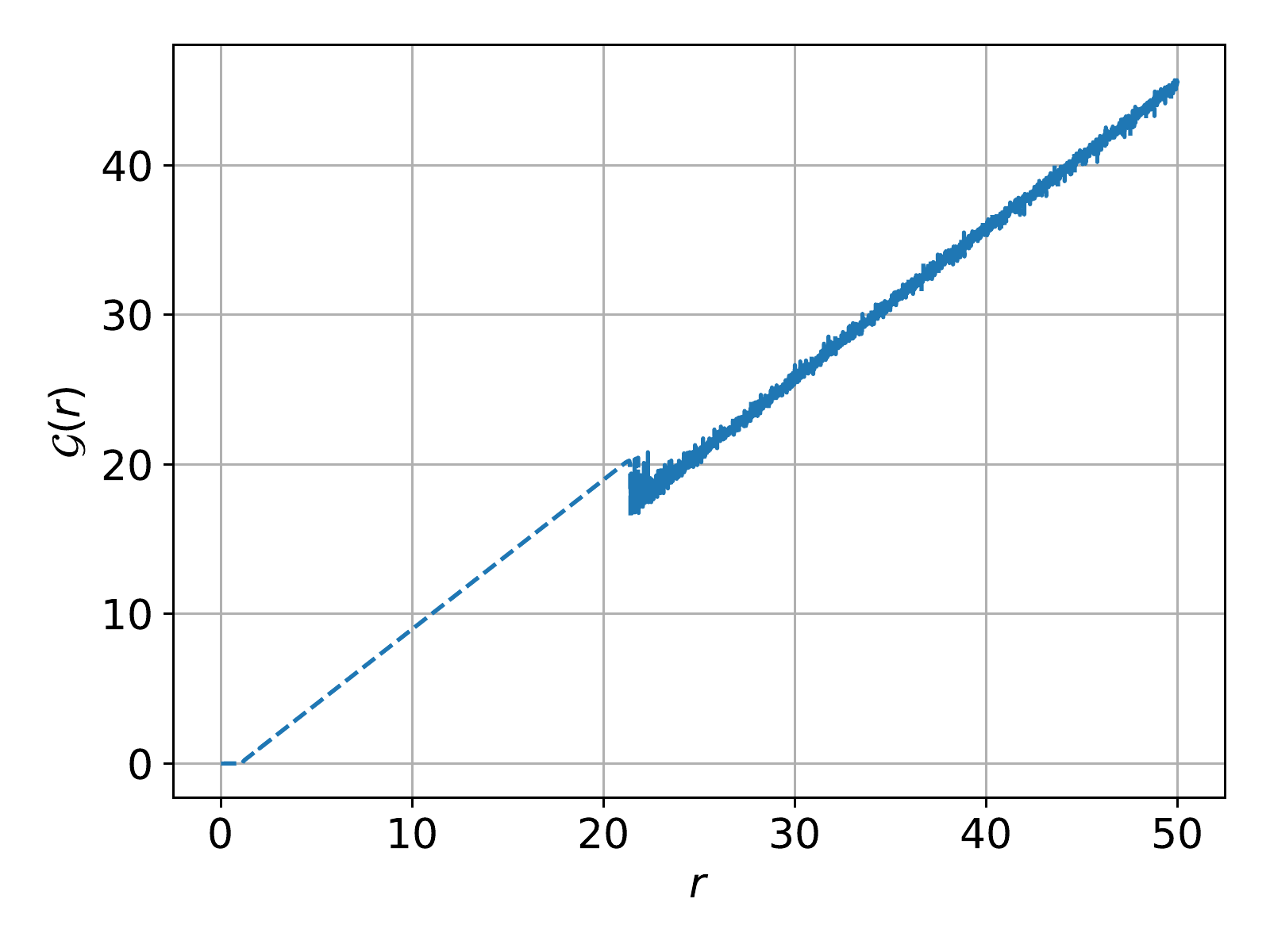}
\includegraphics[width=0.49\textwidth]{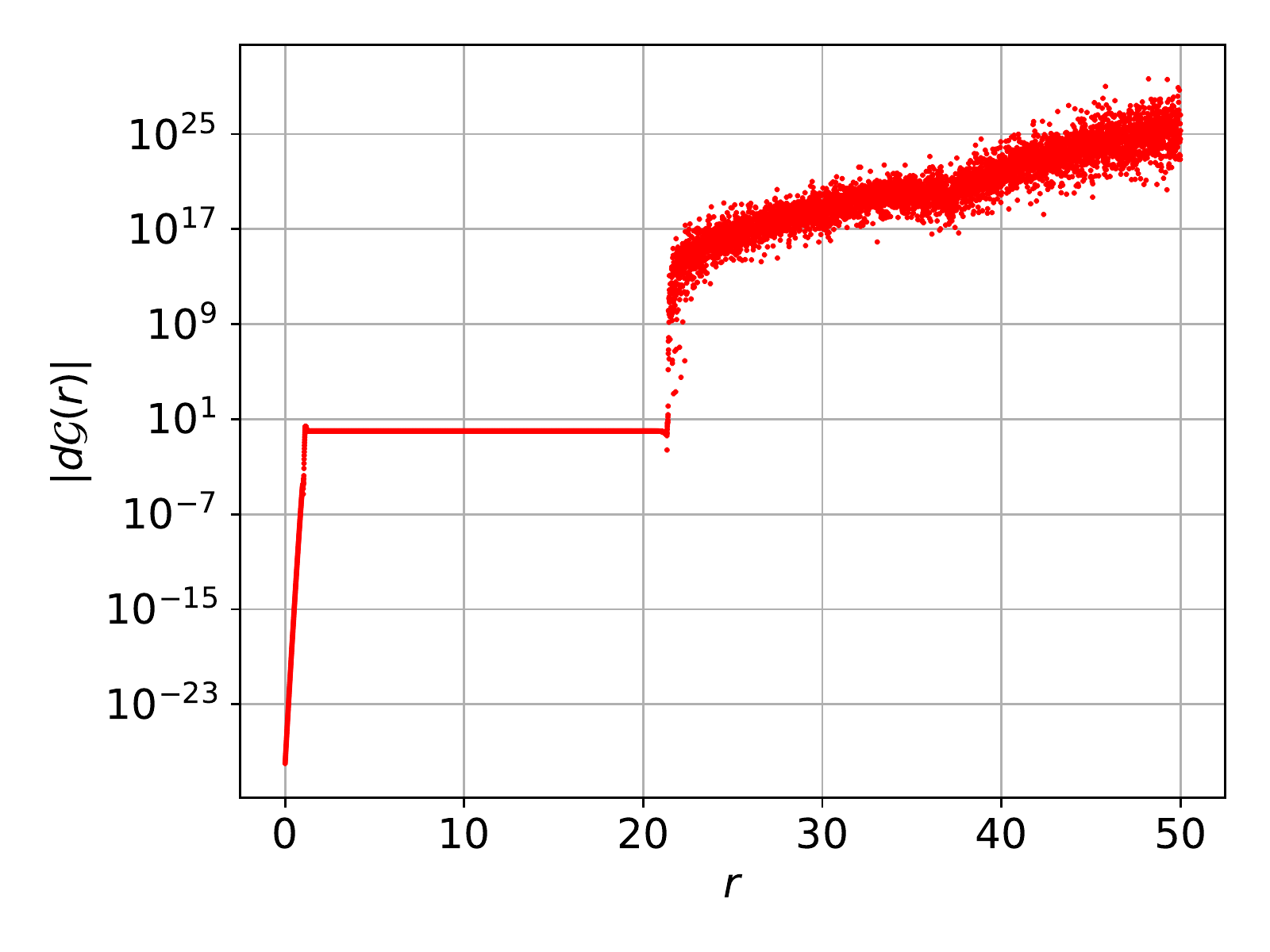}
\caption{Landscape~(left) and sensitivity~(right) of $\G$ in the 1-parameter Lorenz63 inverse problem~\cref{eq:Lorenz-1}}
\label{fig:Loroze63-func1}
\end{figure}

\begin{figure}[ht]
\centering
\includegraphics[width=0.49\textwidth]{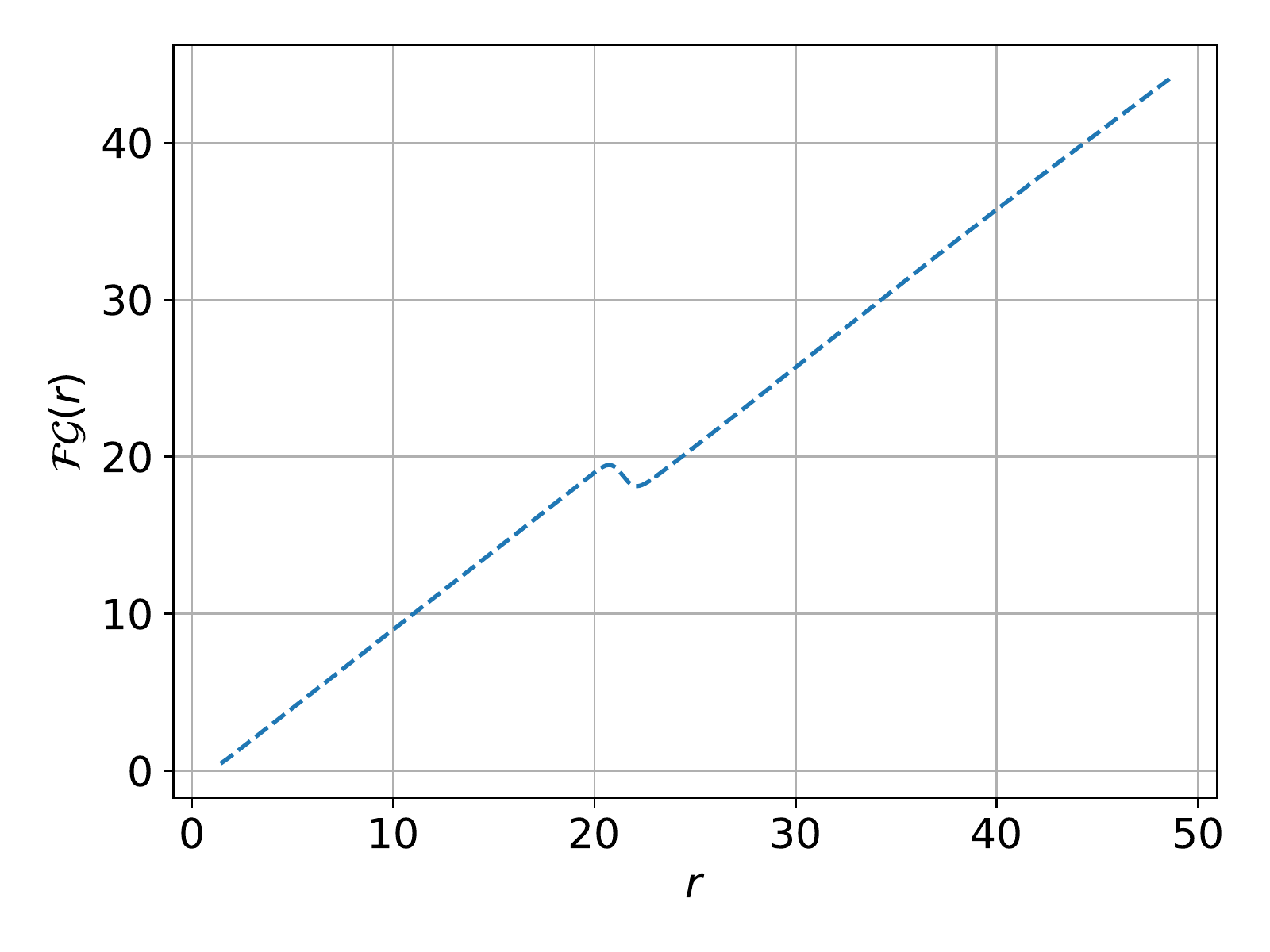}
\includegraphics[width=0.49\textwidth]{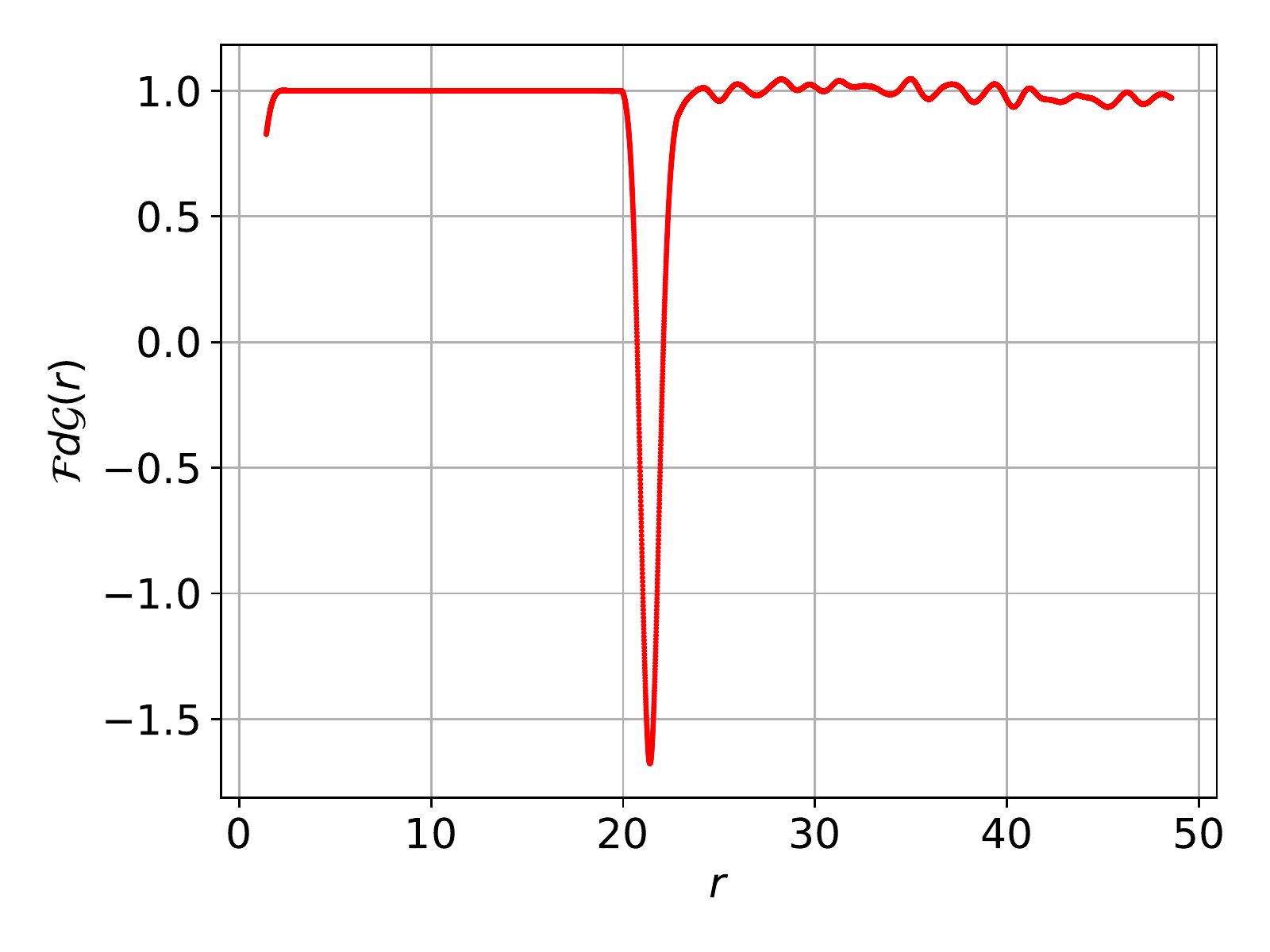}
\caption{Landscape~(left) and sensitivity~(right) of $\F\G$ in the 1-parameter Lorenz63 inverse problem~\cref{eq:Lorenz-1} smoothed and viewed by UKI.}
\label{fig:Loroze63-smoothfunc1}
\end{figure}

Next, we consider a three-parameter inverse problem, using the ideas in  
Subsection \ref{ssec:constraints}. Let $\theta=(\theta_{(1)},\theta_{(2)},\theta_{(3)})$
and let $(\sigma,r,\beta) = (|\theta_{(1)}|,|\theta_{(2)}|,|\theta_{(3)}|).$ The map
$\G(\theta)$ is found by computing time-averages of all three components of $x$,
as described above, for given input parameter $\theta.$ The use of the modulus
helps ensure solution trajectories which do not blow-up. We have
\begin{equation}
\label{eq:Lorenz-3}
    y = \G(\theta)+\eta \quad \mathrm{with} \quad y= (\overline{x_1}, \overline{x_2},\overline{x_3}, \overline{x_1^2}, \overline{x_2^2}, \overline{x_3^2}).
\end{equation}
All other aspects of the setup are the same as the aforementioned one-parameter inverse problem. The estimated parameters and associated 3-$\sigma$ confidence intervals for each component at each iteration are depicted in \cref{fig:Loroze63-3para}. The estimation of the parameters at the 20th iteration is 
\begin{equation*}
   \left(\sigma\quad r\quad \beta \right) = \left(10.28\quad27.90\quad2.63\right)
\end{equation*}
For both scenarios, the UKI converges efficiently, thanks to the linear~(or superlinear) convergence rate of the LMA and the averaging property.
\begin{figure}[ht]
\centering
\includegraphics[width=0.6\textwidth]{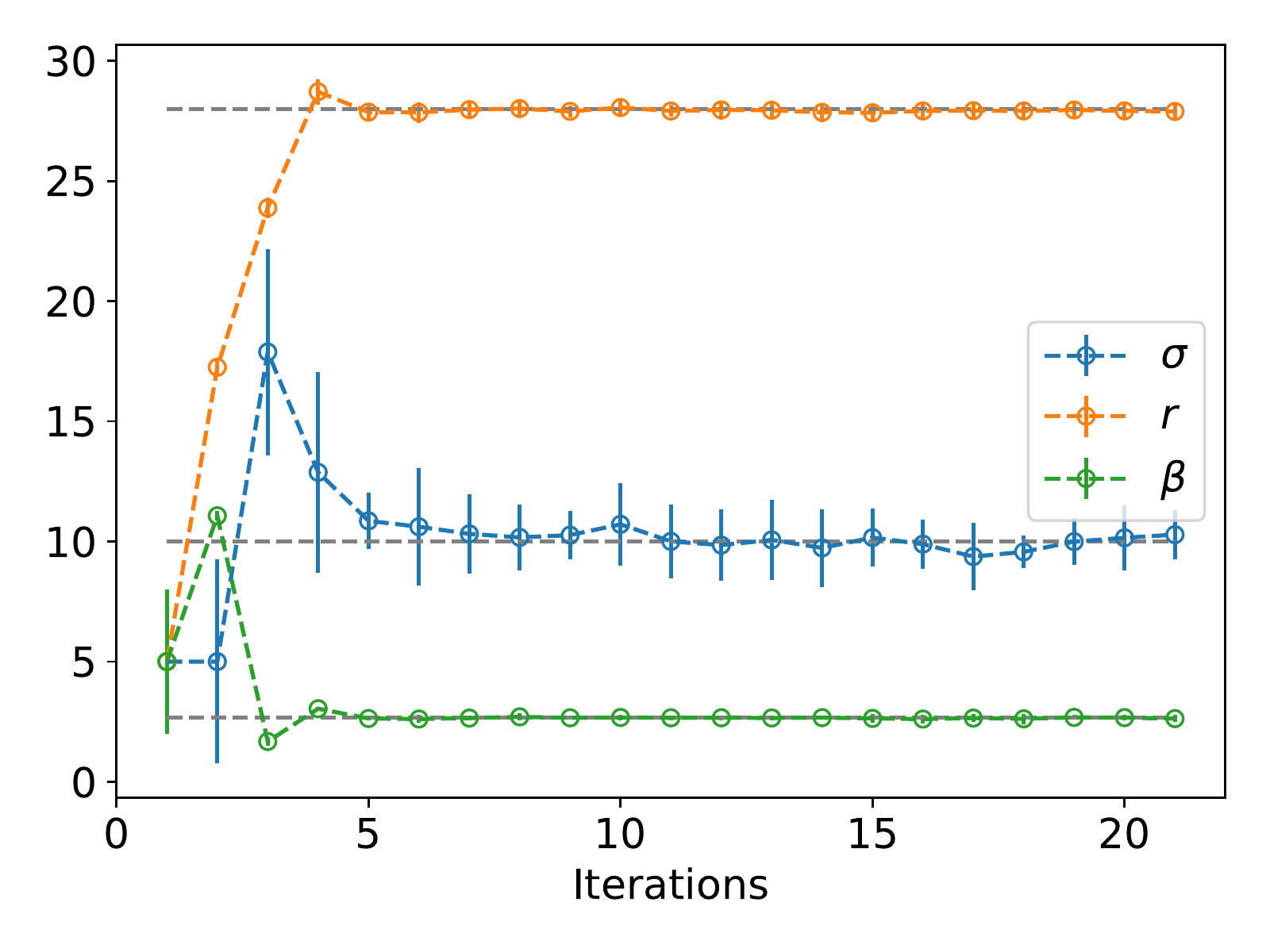}
\caption{Convergence of the 3-parameter Lorenz63 inverse problem with UKI~($\alpha=1.0$); true parameter values are represented by dashed grey lines.}
\label{fig:Loroze63-3para}
\end{figure}

\subsection{Multiscale Lorenz96 Problem}
\label{sec:app:Lorenz96}
Consider the multi-scale Lorenz96 system, a simplified mathematical model for the midlatitude atmosphere~\cite{lorenz1996predictability}, with $K$ slow variables $X^{(k)}$ which are each coupled with $J$ fast variables $Y^{(j,k)}$, given by:
\begin{equation}
\label{eq:Lorenz96}
\begin{split}
    &\frac{dX^{(k)}}{dt} = -X^{(k-1)}(X^{(k-2)} - X^{(k+1)}) - X^{(k)} + F - \frac{hc}{b}\sum_{j=1}^{J} Y^{(j,k)}, \\
    &\frac{dY^{(j,k)}}{dt} = -cbY^{(j+1,k)}(Y^{(j+2,k)} - Y^{(j-1,k)}) - cY^{(j,k)} +\frac{hc}{b}X^{k}.
\end{split}
\end{equation}
To close the system, it is appended with the cyclic boundary conditions $X^{(k+K)} = X^{(k)}$, $Y^{(j,k+K)} = Y^{(j,k)}$ and $Y^{(j+J,k)} = Y^{(j,k+1)}$.
The time scale separation is parameterized by the coefficient $c$ and the large-scales are subjected to external forcing $F$.
We choose here as parameters $K = 8$, $J = 32$, $F =20$, $c=b=10$ and $h=1$ as in~\cite{fatkullin2004computational,wilks2005effects,arnold2013stochastic,gottwald2020supervised}. As time-integrator, we use the 4th-order Runge Kutta method with $\Delta T = 5\times10^{-3}$.

Our goal is to learn the closure model $\psi(X)$ of the fast dynamics for a reduced model of the form
\begin{equation*}
    \frac{dX^{(k)}}{dt} =  -X^{(k-1)}(X^{(k-2)} - X^{(k+1)}) - X^{(k)} + F + \psi(X^{(k)}).
\end{equation*}
The closure model $\psi: D \subset \R \mapsto \R$ is parameterized by the finite element method with cubic Hermite polynomials. The domain is set to be $D=[-20, 20]$ and decomposed into $5$ elements and, therefore, $N_{\theta} = 12$.

For the inverse problem, the observations consist of the time-average of the first and second moments of $X^{(1)}, X^{(2)}, X^{(3)}$, and  $X^{(4)}$ over a time window of size $T = 1000$ and, therefore $N_{y} = 14$. The same central limit theorem arguments are used to formulate the
problem as in the Lorenz63 model.
The truth observation $y_{ref}$ is computed with the multiscale chaotic system~\cref{eq:Lorenz96} with a random initial condition $X^{(k)} \sim \N(0,1) \textrm{ and } Y^{(j,k)} \sim \N(0,0.01^2)$. And $1\%$, $2\%$, and $5\%$ Gaussian random noises are added to the observation following~\cref{eq:add-noise}.

We set $r_0 = 0$ and $\gamma=1$; the UKI is thus initialized with $\theta_0 \sim \N(0, \I)$. The observation error is set to be $\eta = \N(0,\textrm{diag}\{0.05^2 y_{obs}\odot y_{obs}\})$, and we take $\alpha=1$, since the system is over-determined. Moreover, these simulations start with another random initialization of $X^{(k)} \sim \N(0,1)$. The learned closure models at the 20th iteration are reported in \cref{fig:Loroze96-Closure}. The estimated empirical probability density functions of the slow variables are reported in \cref{fig:Loroze96-density}. For all scenarios, although the learned closure models show non-trivial variability with respect to
those published in \cite{wilks2005effects,arnold2013stochastic} at the left most
extreme of $D$, the predicted probability density functions match well with the reference, obtained from a full multiscale simulation. It is worth mentioning this problem is not sensitive with respect to the added Gaussian random noise.

\begin{figure}[ht]
\centering
\includegraphics[width=0.4\textwidth]{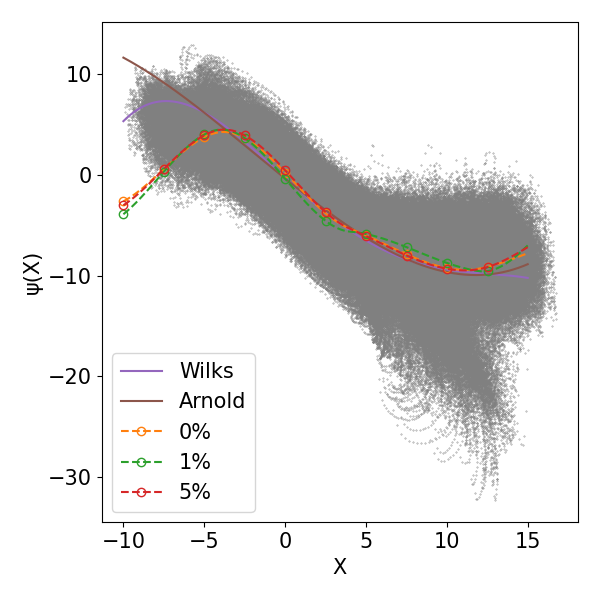}
\caption{Closure terms $\psi(X)$ for the multi-scale Lorenz96 system obtained from the truth~(grey dots) and  polynomial data-fitting by Wilks~\cite{wilks2005effects} and Arnold~\cite{arnold2013stochastic}, compared with what is learned using the UKI approach ($\alpha=1$) with different noise levels.}
\label{fig:Loroze96-Closure}
\end{figure}

\begin{figure}[ht]
\centering
\includegraphics[width=0.4\textwidth]{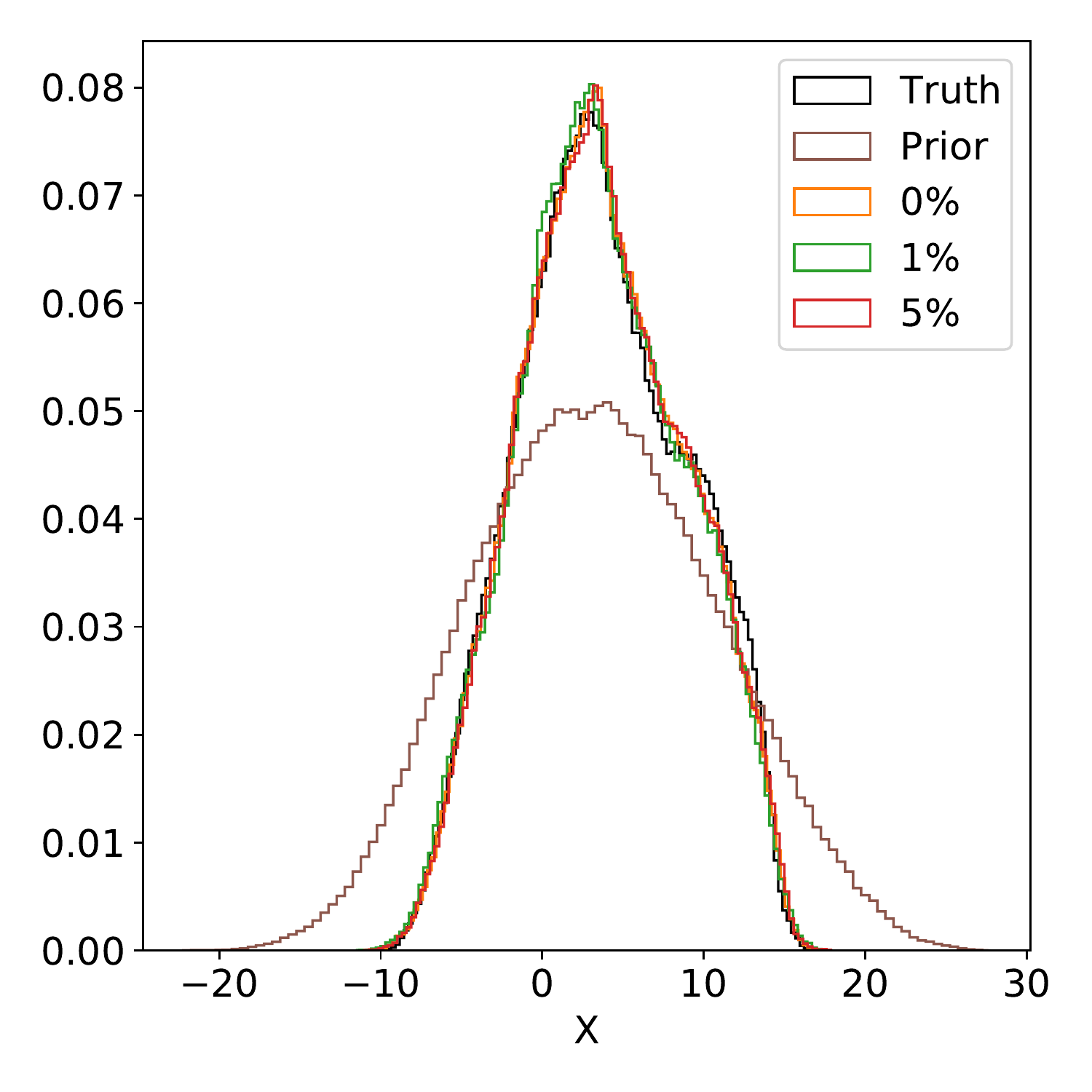}
\caption{Empirical probability density functions of the slow variables $X^{(k)}$ obtained from the full multi-scale Lorenz96 system~(Truth), the initial closure model~(Prior), and the  closure models learned by the UKI~($\alpha=1$) at different noise levels.}
\label{fig:Loroze96-density}
\end{figure}

\subsection{Idealized General Circulation Model}
\label{sec:app:GCM}

Finally, we consider an idealized general circulation model. The model is based on
the 3D Navier-Stokes equations, making the hydrostatic and shallow-atmosphere approximations common in atmospheric modeling. Specifically, we test UKI on the well-known Held-Suarez test case~\cite{held1994proposal}, in which a detailed radiative transfer model is replaced by Newtonian relaxation of temperatures toward a prescribed ``radiative equilibrium'' $T_{\mathrm{eq}}(\phi, p)$ that varies with latitude $\phi$ and pressure $p$. Specifically, the thermodynamic equation for temperature $T$ 
\[
\frac{\partial T}{\partial t} + \dots = Q
\]
(dots denoting advective and pressure work terms) contains a diabatic heat source 
\[
Q = -k_T(\phi, p, p_s) \bigl(T - T_{\mathrm{eq}}(\phi, p)\bigr),
\]
with relaxation coefficient (inverse relaxation time)
\[
k_T = k_a + (k_{s} - k_a)\max\Bigl(0, \frac{\sigma - \sigma_b}{1 - \sigma_b}\Bigr)\cos^4\phi.
\]
Here, $\sigma = p/p_s$, which is pressure $p$ normalized by surface pressure $p_s$, is the vertical coordinate of the model, and 
\[
T_{\mathrm{eq}} = \max\Bigl\{200K, \Bigl[315K - \Delta T_y \sin^2\phi - \Delta\theta_z\log\Bigl(\frac{p}{p_0}\Bigr)\cos^2\phi\Bigr]\Bigl(\frac{p}{p_0}\Bigr)^{\kappa}\Bigr\}
\]
is the equilibrium temperature profile ($p_0 = 10^5~\mathrm{Pa}$ is a reference surface pressure and $\kappa = 2/7$ is the adiabatic exponent). Default parameters are 
\[
k_a = (40\ \mathrm{day})^{-1},\qquad  
k_{s} = (4\ \mathrm{day})^{-1}, \qquad
\Delta T_y = 60\ \mathrm{K}, \qquad
\Delta\theta_z = 10\ \mathrm{K}.
\]

For the numerical simulations, we use the spectral transform method in the horizontal, with T42 spectral resolution~(triangular truncation at wavenumber 42, with $64 \times 128$ points on the latitude-longitude transform grid); we use 20 vertical levels equally spaced in $\sigma$. With the default parameters, the model produces an Earth-like zonal-mean circulation, albeit without moisture or precipitation. A single jet is generated with maximum strength of roughly $30\ \mathrm{m~s^{-1}}$ near $45^{\circ}$ latitude (\cref{fig:GCM-sol}).

Our inverse problem is constructed to learn parameters in the Newtonian relaxation term $Q$:
\begin{equation*}
    (k_a,\ k_{s},\ \Delta T_y,\ \Delta\theta_z).
\end{equation*}
We do so in the presence of the following constraints: 
$$0\  \mathrm{day}^{-1}<k_a < 1\  \mathrm{day}^{-1}, \qquad   k_a<k_s <1\  \mathrm{day}^{-1} + k_a, \qquad   0\  \mathrm{K} < \Delta T_y, \qquad   0\  \mathrm{K}<\Delta\theta_z.$$
Conceptually, the setting is identical to that for the Lorenz63
example. We use the same overline notation to denote averaging, which here in addition to the time average in the Lorenz models also includes a zonal average over longitude (because the model is statistically symmetric under rotations around the planet's spin axis), and we
apply the same central limit theorem arguments to formulate the inverse problem.
To incorporate the imposition of the constraints, the inverse problem is formulated as follows~(see Subsection~\ref{ssec:constraints} for details):
\begin{equation}
    y = \G(\theta) + \eta \quad \mathrm{with} \quad \G(\theta)=\overline{T}(\phi, \sigma)
\end{equation}
with the parameter transformation 
\begin{equation}
    \theta: (k_a, k_s, \Delta T_y, \Delta\theta_z) = \Big(\frac{1}{1+|\theta_{(1)}|},\ \frac{1}{1+|\theta_{(1)}|}+\frac{1}{1+|\theta_{(2)}|},\
    |\theta_{(3)}|,
    |\theta_{(4)}|\Big).
\end{equation}
The observation mapping is defined by  mapping from the unknown $\theta$ to 
the 200-day zonal mean of the temperature as a function of latitude ($\phi$) and height ($\sigma$), after an initial spin-up of 200 days. The truth observation is the 1000-day zonal mean of the temperature~(see \cref{fig:GCM-sol}-a), after an initial spin-up 200 days to eliminate the influence of the initial condition. Because the truth observations come from an average 5 times as long as the observation window used for parameter learning, the chaotic internal variability of the model introduces noise in the observations. As for the Lorenz63 setting, the central limit theorem may be invoked to model the observation error from internal variability. 

To perform the inversion, we set $r_0 = [2\ \textrm{day},\ 2\ \textrm{day},\ 20\ \textrm{K},\ 20\ \textrm{K}]^{T}$  and $\gamma=1$. Thus UKI is initialized with $\displaystyle \theta_0 \sim \N\Big(r_0,\ \I\Big).$ Within the algorithm, we assume that the observation error satisfies $\eta \sim \N(0\ \textrm{K}, 3^2\I\ \textrm{K}^2)$. Because the problem is over-determined,
we set $\alpha=1.$ The estimated parameters and associated 3-$\sigma$ confidence intervals for each component at each iteration are depicted in \cref{fig:GCM-obj}. The estimation of model parameters at the 20th iteration is
\begin{equation*}
    \left(
    k_a\quad
    k_s\quad
    \Delta T_y\quad
    \Delta \theta_z
    \right)
    =
    \left(
    0.0243\ \textrm{day}^{-1}\quad
    0.243\ \textrm{day}^{-1}\quad
    60.2\ \textrm{K} \quad
    9.91\ \textrm{K}
    \right).
\end{equation*}
UKI converges to the true parameters in fewer than 10 iterations with 9 $\sigma$-points, demonstrating the potential of applying UKI for large-scale inverse problems.

\begin{figure}[ht]
    \centering
    \begin{subfigure}{0.49\textwidth}
    \centering
        \includegraphics[width=0.9\linewidth]{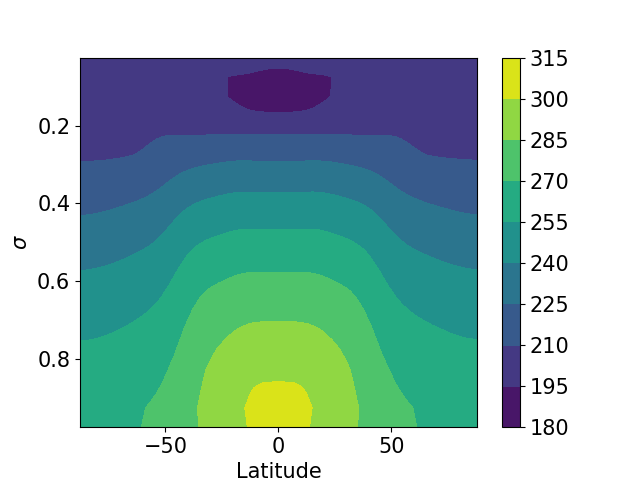}
        \caption{T}
    \end{subfigure}%
    ~ 
    \begin{subfigure}{0.49\textwidth}
    \centering
        \includegraphics[width=0.9\linewidth]{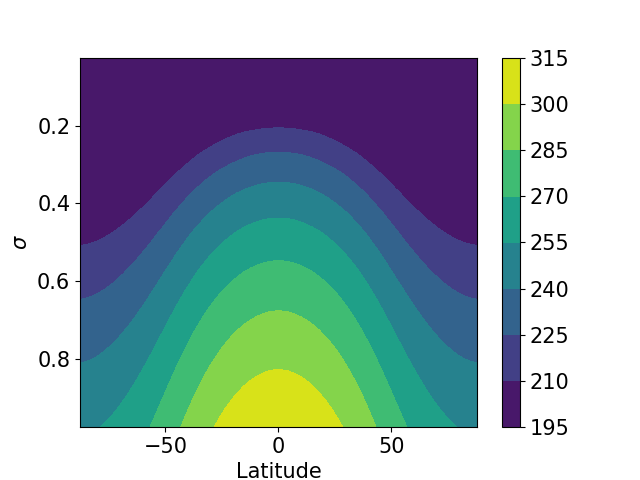}
        \caption{T$_\mathrm{eq}$}
    \end{subfigure}%
    
    \begin{subfigure}{0.49\textwidth}
    \centering
        \includegraphics[width=0.9\linewidth]{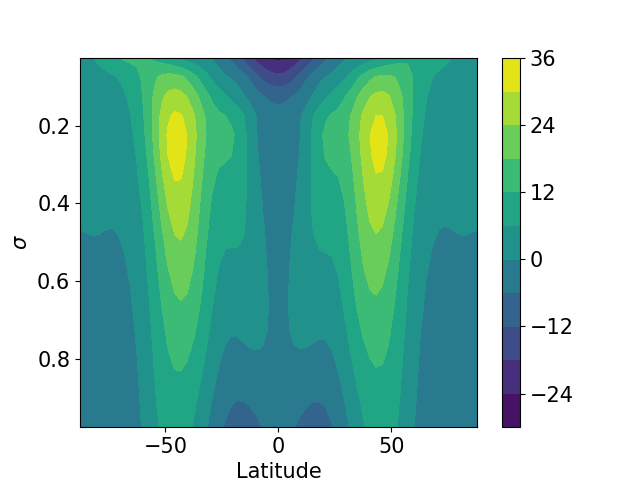}
        \caption{U}
    \end{subfigure}%
    ~ 
    \begin{subfigure}{0.49\textwidth}
    \centering
        \includegraphics[width=0.9\linewidth]{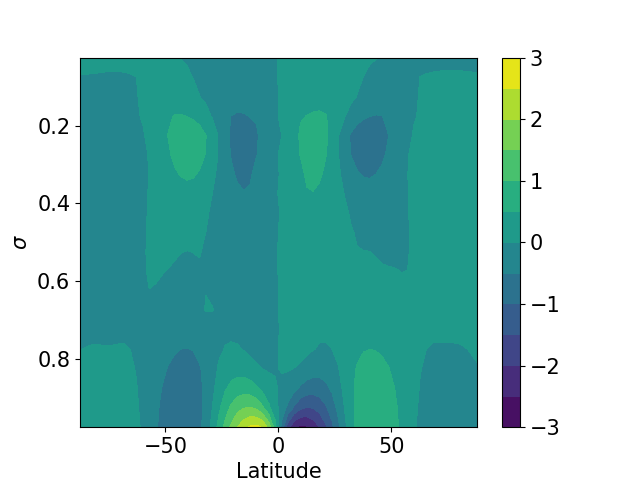}
        \caption{V}
    \end{subfigure}%
    \caption{Zonal mean profile of temperature~(a), radiative equilibrium temperature~(b), zonal wind velocity~(c), and meridional wind velocity~(d), all from a 1000-day average. The horizontal coordinate is latitude and the vertical coordinate is the nondimensional $\sigma$ coordinate of the model.}
    \label{fig:GCM-sol}
\end{figure}

\begin{figure}[ht]
\centering
\includegraphics[width=0.7\textwidth]{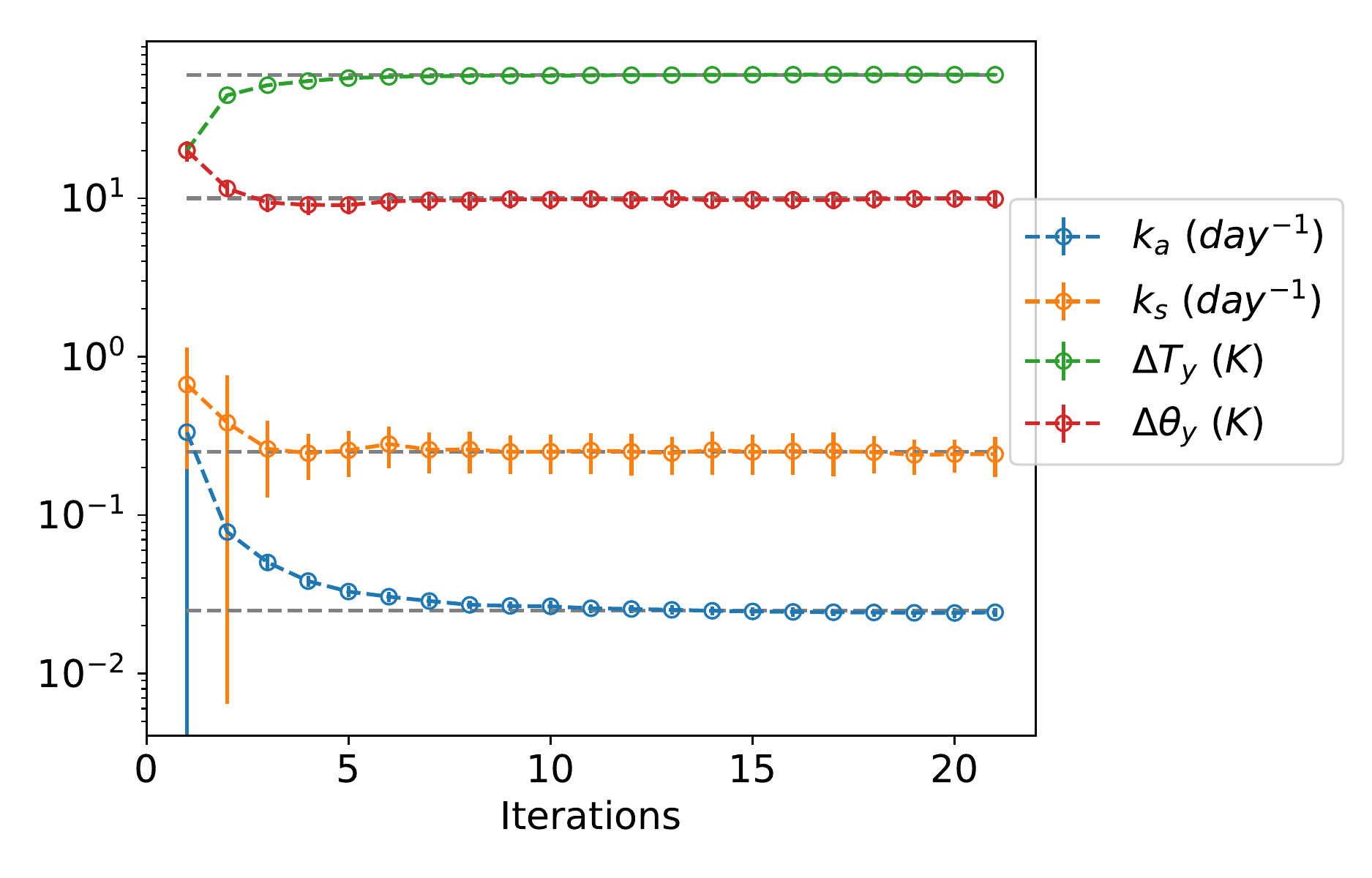}
\caption{Convergence of the idealized general circulation model inverse problem with UKI~($\alpha=1.0$). The true parameter values are represented by dashed grey lines.}
\label{fig:GCM-obj}
\end{figure}

\section{Conclusion}
We introduced a novel stochastic dynamical system, into which an arbitrary inverse problem may be embedded as an observation operator; by applying filtering methods to this
stochastic dynamical system we obtain methods to solve inverse problems. In the linear case,
we have demonstrated that this approach leads to an unusual Tikhonov regularized least
squares solution, with prior covariance depending on the forward model, and a tunable
parameter in the stochastic dynamical system determining the level of
regularization. We have also introduced unscented Kalman inversion (UKI) and
shown that it outperforms the EKI, when applied to the same novel stochastic dynamical system.
As well as outperforming EKI, UKI shares its advantages: it is derivative-free, black-box,
embarrassingly parallel, and robust.  Our numerical results demonstrate its theoretical properties and its applicability;
in particular, it is demonstrated to outperform the EKI on large scale problems in which the number
of unknown parameters is small. Because the methodology constitutes a novel approach to parameter
estimation, there are many avenues for future research,
including applications of the method, methodological improvements and extensions, and theoretical analysis.

\paragraph{Acknowledgments} This work was supported by the generosity of Eric and Wendy Schmidt by recommendation of the Schmidt Futures program and by the National Science Foundation (NSF, award AGS‐1835860). A.M.S. was also supported by the Office of Naval Research (award N00014-17-1-2079). The authors thank Sebastian Reich and anonymous reviewers for helpful comments on an earlier draft.

\appendix
\section{Proof of Theorems}
\label{sec:appendix}

\begin{proof}[Proof of Proposition~\ref{prop:affine-invariance}]
An affine transformation is an invertible mapping from $R^{N_{\theta}}$ to $R^{N_\theta}$ of the form ${}^{*}x = Ax + b$.
When we apply the following affine transformation
\begin{equation*}
\begin{split}
& {}^{*}\mean_n =  A\mean_n + b \qquad {}^{*}\Cov_{n} = A\Cov_{n} A^T  \quad\textrm{ with }\quad {}^{*}r_0 = Ar_0 + b \qquad {}^{*}\Sigma_{\omega} = A\Sigma_{\omega} A^T, \\
 \end{split}
\end{equation*}
keep $y_n$ and $\Sigma_{\nu}$ unchanged, and define ${}^{*}\G(\theta) =  \G\big(A^{-1}(\theta - b)\big)$. We prove 
\begin{equation}
 {}^{*}\mean_{n+1} = A\mean_{n+1} + b \qquad {}^{*}\Cov_{n+1} = A\Cov_{n+1}A^T.
\end{equation}
\Cref{eq:KF_pred_mean} leads to  
\begin{equation}
     {}^{*}\pmean_{n+1} = \alpha{}^{*}\mean_n + (1-\alpha){}^{*}r_0 = A\pmean_{n+1} + b \qquad     {}^{*}\pCov_{n+1} =  \alpha^2 {}^{*}\Cov_{n} + {}^{*}\Sigma_{\omega} = A \pCov_{n+1}A^T.
\end{equation}
Therefore, the distribution of ${}^{*}\theta_{n+1}|Y_n \sim \N( {}^{*} \pmean_{n+1}  {}^{*}\pCov_{n+1})$ is the same as $A\theta_{n+1} + b|Y_n$ and 
\cref{eq:KF_joint2} becomes 
\begin{equation}
   {}^{*}\py_{n+1} =     \py_{n+1}  \qquad 
    {}^{*}\pCov_{n+1}^{\theta y} =    A   \pCov_{n+1}^{\theta y} \qquad 
    {}^{*}\pCov_{n+1}^{yy} =  \pCov_{n+1}^{yy}.
\end{equation}
Finally, \cref{eq:KF_analysis} leads to 
\begin{equation}
    \begin{split}
        {}^{*}\mean_{n+1} &= {}^{*}\pmean_{n+1} + {}^{*}\pCov_{n+1}^{\theta y} ({}^{*}\pCov_{n+1}^{yy})^{-1} (y_{n+1} - {}^{*}\py_{n+1})= A\mean_{n+1} + b,\\
         {}^{*}\Cov_{n+1} &= {}^{*}\pCov_{n+1} - {}^{*}\pCov_{n+1}^{\theta y}({}^{*}\pCov_{n+1}^{yy})^{-1} {{}^{*}\pCov_{n+1}^{\theta y}}{}^{T} = A\Cov_{n+1}A^T.
    \end{split}
\end{equation}
\end{proof}

\begin{proof}[Proof of Lemma~\ref{lem:uki}]
In this proof recall that $\theta \sim \N(\mean, \Cov).$ 
When both $\G_1 = G_1$ and $\G_2 = G_2$ are linear transformations, we have 
\begin{align*}
  \E[ \G_i(\theta)] &=  G_i \E[\theta] = G_i m = \G_i(m),\\
  \mathrm{Cov}[\G_1(\theta),\G_2(\theta)] &=
  G_1 \mathrm{Cov}[\theta,\theta]G_2^T
  = G_1 C G_2^T,\\
  \sum_{j=1}^{2N_{\theta}} W_j^{c} (\G_1(\theta^j) - \E\G_1(\theta))(\G_2(\theta^j) - \E\G_2(\theta))^T
  &= 
  \frac{1}{2a^2N_{\theta}}\sum_{j=1}^{2N_{\theta}} (G_1\cdot\theta^j - G_1 m)(G_2\cdot\theta^j - G_2 m )^T\\
  &= \frac{1}{2a^2N_{\theta}}\sum_{j=1}^{2N_{\theta}} 2G_1 c_j[\sqrt{C}]_j  c_j[\sqrt{C}]_j G_2^T\\
  &= G_1CG_2^T. 
\end{align*}
In the following we use $\nabla^k \G_i$ to denote the $k^{th}$
derivative of $\G_i$ evaluated at $m$. For the nonlinear case,  Taylor's expansion of $\G_i(\cdot)$ at $m$ is then
\begin{align*}
    \G_i(\theta) = \G_i(m) + \nabla \G_i \delta \theta + \frac{1}{2} \nabla^2 \G_i \delta \theta \otimes \delta \theta +\frac{1}{6} \nabla^3 \G_i \delta \theta \otimes \delta \theta \otimes \delta \theta + O(\|\delta \theta\|^4) \quad\textrm{with}\quad \delta\theta = \theta - m.\\
\end{align*}
The mean approximation is thus first-order accurate:
\begin{align*}
    \E\G_i(\theta)  
    = \G_i(m) +  O(\|C\|).\\
\end{align*}
The covariance approximation is second-order accurate:
\begin{align*}
 \mathrm{Cov}[\G_1(\theta),\G_2(\theta)] 
 & =\E\left(\G_1(\theta) - \E\G_1(\theta)\right)\left(\G_2(\theta) - \E\G_2(\theta)\right)^T\\
 & =\E\left(\nabla\G_1 \delta \theta
 + \frac{1}{2}\nabla^2\G_1(\delta\theta\otimes\delta\theta - C)\right) \left(\nabla\G_2 \delta \theta + \frac{1}{2}\nabla^2\G_2(\delta\theta\otimes\delta\theta - C)\right)^T + \bigO(\|C\|^2)\\
&= \nabla\G_1 C \nabla\G_2^T + \bigO(\|C\|^2), 
\end{align*}
whilst we also have
\begin{align*}
  \sum_{j=1}^{2N_{\theta}}& W_j^{c} (\G_1(\theta^j) - \E\G_1(\theta))(\G_2(\theta^j) - \E\G_2(\theta))^T\\
  &= \sum_{j=1}^{2N_{\theta}} W_j^{c} (\G_1(\theta^j) - \G_1(m))(\G_2(\theta^j) - \E\G_2(m))^T \\
  &= \sum_{j=1}^{N_{\theta}} W_j^{c} (\nabla \G_1 c_j[\sqrt{C}]_j + \frac{1}{2} \nabla^2 \G_1 c_j[\sqrt{C}]_j \otimes c_j[\sqrt{C}]_j)(\nabla \G_2 c_j[\sqrt{C}]_j+ \frac{1}{2} \nabla^2 \G_2 c_j[\sqrt{C}]_j\otimes c_j[\sqrt{C}]_j)^T\\
  &+ \sum_{j=1}^{N_{\theta}} W_j^{c} (-\nabla \G_1 c_j[\sqrt{C}]_j + \frac{1}{2} \nabla^2 \G_1 c_j[\sqrt{C}]_j \otimes c_j[\sqrt{C}]_j)(-\nabla \G_2 c_j[\sqrt{C}]_j + \frac{1}{2} \nabla^2 \G_2 c_j[\sqrt{C}]_j\otimes c_j[\sqrt{C}]_j)^T\\
  &\quad\quad+ \bigO(\|C\|^2)\\
  &= \frac{1}{2a^2N_{\theta}}\sum_{j=1}^{N_{\theta}} 2\nabla \G_1 c_j[\sqrt{C}]_j c_j[\sqrt{C}]_j \nabla \G_2^T  + \bigO(\|C\|^2)\\
   &= \nabla \G_1 C \nabla \G_2^T   + \bigO(\|C\|^2).
\end{align*}
\end{proof}

\begin{proof}[Proof of Theorem~\ref{th:lin_converge}]
In this proof we let $\cB$ denote the Banach space of matrices in $\R^{N_\theta \times N_\theta}$ equipped
with the operator norm induced by the Euclidean norm on $\R^{N_\theta}.$ Furthermore, we let $\cL$
denote the Banach space of bounded linear operators from $\cB$ into itself, equipped with the standard
induced operator norm. 
For simplicity we consider the case $r_0=0$; a change of origin may be used to extend to the case $r_0 \ne 0.$
We first prove that the precision operators converge: $\displaystyle \lim_{n \rightarrow \infty} \Cov_{n}^{-1} = \Cov_{\infty}^{-1}$; we then study behaviour of the mean sequence $\{m_n\}_{n \in \bbZ^+}$. For both the precision and the
mean we first study $\alpha \in (0,1)$ and then $\alpha=1.$
In what follows it is useful to note 
\cite{law2015data}[Theorem 4.1] that the mean and covariance update equations~\eqref{eq:Lin_KF_analysis} can be rewritten as  
\begin{equation}
\begin{split}
~\label{eq:Sigma_update}
    \Cov_{n+1}^{-1} &= G^T\Sigma_{\nu}^{-1}G + (\alpha^2 \Cov_n + \Sigma_{\omega})^{-1},\\
  \Cov_{n+1}^{-1}m_{n+1} &= G^T\Sigma_{\nu}^{-1}y + (\alpha^2 \Cov_n + \Sigma_{\omega})^{-1}\alpha m_n;  
\end{split}
\end{equation}
furthermore the iteration for the covariance remains in the cone of positive semi-definite
matrices \cite{law2015data}[Theorem 4.1].
Since $\Sigma_{\omega} \succ 0$, the sequence $\{\Cov_n^{-1}\}$ is bounded:
\begin{equation}
    \label{eq:Cbnd}
G^T\Sigma_{\nu}^{-1}G\preceq\Cov_n^{-1}\preceq G^T\Sigma_{\nu}^{-1}G + \Sigma_{\omega}^{-1}, \quad \forall n \in \mathbb{Z}_+.
\end{equation}
Introducing $(\Cov'_{n})^{-1} := \Sigma_{\omega}^{\frac{1}{2}}\Cov_{n}^{-1}\Sigma_{\omega}^{\frac{1}{2}}$, 
we may rewrite the covariance update equation~\eqref{eq:Sigma_update} in the form
\begin{equation}
\label{eq:Sigma_update2525}
    \begin{split}
    &(\Cov'_{n+1})^{-1} =  \Sigma_{\omega}^{\frac{1}{2}}G^T\Sigma_{\nu}^{-1}G\Sigma_{\omega}^{\frac{1}{2}} + \left(\alpha^2 \Cov'_{n} + \I\right)^{-1}.\\
    \end{split}
\end{equation}
We define the map
\begin{equation}
f(X;\alpha)=  \Sigma_{\omega}^{\frac{1}{2}}G^T\Sigma_{\nu}^{-1}G\Sigma_{\omega}^{\frac{1}{2}} + \left(\alpha^2 X^{-1} + \I\right)^{-1}
\end{equation}
noting that then $(\Cov'_{n+1})^{-1}=
f\bigl((\Cov'_{n})^{-1};\alpha\bigr).$
This iteration is well-defined for $\Cov'_{n}$ in
$\cB$ satisfying \eqref{eq:Cbnd} and hence for the iteration
\eqref{eq:Sigma_update}.

We first consider $\alpha \in (0, 1)$. Then \cref{eq:Sigma_update2525} leads to 
\begin{equation}
    \begin{split}
    &\Cov'_{n+1} \preceq \alpha^2 \Cov'_{n} + \I \preceq \frac{1-\alpha^{2n+2}}{1-\alpha^2}\I + \alpha^{2n+2}
    \Cov'_0 \preceq \frac{1}{1-\alpha^2}\I + \alpha^{2n+2}\Cov'_0,\\
    \end{split}
\end{equation}
and hence there exists $\epsilon_0 \in (0,1-\alpha)$ such that,
for $n$ is sufficiently large, we have 
\begin{equation}
\label{eq:alpha_cov_lb}
    \begin{split}
    &(\Cov'_{n+1})^{-1} \succeq (1 -  \alpha^2 - \epsilon_0)\I. \\
    \end{split}
\end{equation}
Let $\cM \subset \cB$ denote the set of matrices $B \in \cB$
satisfying $B \succeq (1 -  \alpha^2 - \epsilon_0)\I.$ Then $\cM$ is
absorbing and forward invariant under $f$. Thus to show the existence
of a globally exponentially attracting steady state it suffices to show that $f(\cdot;\alpha)$ is a contraction on $\cM.$
\footnote{The use of contraction mapping arguments to study
convergence of the Kalman filter is widespread, sometimes
applied to the covariance and not the precision \cite{lancaster1995algebraic}, and sometimes
using Riemannian metric space structure on positive-definite matrices, rather than the vector space structure used here \cite{bougerol1993kalman}.} The derivative of $f(\cdot;\alpha):\cM \mapsto \cM$ is the element $Df(X;\alpha) \in \cL$ defined by its action on $\Delta X \in \cB$ as follows:
\begin{equation}
Df(X;\alpha)\Delta X = \alpha^2 (X+\alpha^2\I )^{-1} \Delta X 
   (X +\alpha^2\I)^{-1}.
\end{equation}
Thus
\begin{align*}
   \left\lVert Df(X;\alpha) \Delta X \right\rVert 
   =&  \alpha^2\left\lVert (X+\alpha^2\I )^{-1} \Delta X 
   (X +\alpha^2\I)^{-1} \right\rVert. \\
   \leq & \frac{\alpha^2}{(1-\epsilon_0)^2}\left\lVert  \Delta X 
   \right\rVert. 
\end{align*}
Therefore, since $\alpha \in (0,1-\epsilon_0)$,
$$\sup_{X \in \cM} \|Df(X;\alpha)\|_{\cL}<1$$
and $f$
is a contraction map on $\cM.$ This establishes the exponential convergence of  $\{(\Cov'_{n})^{-1}\}$. Finally, 
the sequence $\{\Cov_n^{-1}\}$ converges exponentially fast to $\Cov_{\infty}^{-1}$, the non-singular fixed point of \cref{eq:Sigma_update}; \Cref{eq:alpha_cov_lb} indicates that $\Cov_{\infty}^{-1}$ is indeed non-singular.

When $\alpha = 1$ define mapping $f(X)=f(X;1)$ so that
$$(\Cov'_{n+1})^{-1}=
f\bigl((\Cov'_{n})^{-1}\bigr).$$
The derivative $Df(X) \in \cL$ is
defined by its action on $\Delta X \in \cB$ as follows:
\begin{equation}
\begin{split}
   Df(X) \Delta X =&  (\I+X)^{-1}\Delta X (I+X)^{-1}.
\end{split}
\end{equation}
Thus, using the lower bound from \eqref{eq:Cbnd} and $\text{Range}(G^T)=\R^{N_{\theta}}$,
\begin{equation}
\begin{split}
   \left\lVert Df(X) \Delta X \right\rVert 
   \leq& \left\lVert \left(\I+X\right)^{-1}\right\rVert^2 \left\lVert \Delta X\right\rVert \\
   \leq& \left\lVert \left(\I+\Sigma_{\omega}^{\frac{1}{2}}G^T\Sigma_{\nu}^{-1}G\Sigma_{\omega}^{\frac{1}{2}}\right)^{-1}\right\rVert^2 \left\lVert \Delta X\right\rVert \\
   \leq& (1 + \epsilon_1)^{-2} \left\lVert \Delta X\right\rVert, \\
\end{split}
\end{equation}
where $\epsilon_1 > 0$. Therefore, $f$ is a contraction map on
the whole of $\cB$ and
the sequence $\{\Cov_n^{-1}\}$ converges.  
This completes the proof of exponential
convergence of $\{\Cov^{-1}_{n}\}$
to a limit; the sequence $\{\Cov_n^{-1}\}$ converges to $\Cov_{\infty}^{-1}$, the fixed point of \cref{eq:Sigma_update},
viewed as a mapping on precision matrices. 
That $\Cov_{\infty} \succ 0$ follows from \eqref{eq:Cbnd}.
Because the convergence is global,
the result also establishes the uniqueness of the steady state of \cref{eq:cov_steady}.

We now prove that the mean $\{m_n\}$ converges exponentially fast to $\mean_{\infty}.$
Using \eqref{eq:Sigma_update} the update equation~(\ref{eq:Lin_KF_analysis}) of $\mean_n$ can be rewritten as
\begin{equation}
 \mean_{n+1}=\alpha(\I-C_{n+1}G^T\Sigma_{\nu}^{-1}G)m_n
 +C_{n+1}G^T\Sigma_{\nu}^{-1}y.   
\end{equation}
Thus convergence to $\mean_{\infty}$ satisfying
\begin{equation}
\label{eq:mean_infty}
 \mean_{\infty}=\alpha(\I-C_{\infty}G^T\Sigma_{\nu}^{-1}G)
 \mean_{\infty}
 +C_{\infty}G^T\Sigma_{\nu}^{-1}y   
\end{equation}
is determined by the spectral radius of 
$\alpha(\I-C_{n+1}G^T\Sigma_{\nu}^{-1}G).$
The matrix $\I-C_{n+1}G^T\Sigma_{\nu}^{-1}G$ has real spectrum; this may be established
by showing the same for $\I-C_{n+1}(G^T\Sigma_{\nu}^{-1}G+\delta\I)$, for $\delta>0$, and
letting $\delta \to 0.$
If $\alpha \in (0,1)$, using~\cref{eq:Cbnd}, it follows that
\begin{equation*}
\begin{split}
\rho(\alpha\I - \alpha \Cov_{n+1}G^T\Sigma_{\nu}^{-1}G) \leq \alpha < 1
\end{split}
\end{equation*}
and, by using a vector norm on $\R^{N_{\theta}}$ in which the induced operator norm on $\I-C_{n+1}G^T\Sigma_{\nu}^{-1}G$
is less than one, it follows that $\{\mean_{n}\}$ converges exponentially fast to 
$\mean_{\infty}.$ If $\alpha=1$ then we use the fact that $B:=G^T\Sigma_{\nu}^{-1}G$ is symmetric and that $B \succ 0.$ From this it follows that $I-\Cov_{n+1}B$ has the same
spectrum as $I-B^{\frac12}\Cov_{n+1}B^{\frac12}.$
Using the upper bound on $\Cov_{n+1}^{-1}$ appearing in \eqref{eq:Cbnd} we deduce that
\begin{equation*}
\begin{split}
\rho(\I-\Cov_{n+1}B) &=
\rho\Big(\I - B^{\frac{1}{2}}\Cov_{n+1}B^{\frac{1}{2}}\Big)\\ 
&\leq
1 - \rho\Big(B^{\frac{1}{2}}\big(B + \Sigma_{\omega}^{-1}\big)^{-1}B^{\frac{1}{2}}\Big) \\
&= 1 -\epsilon_2,
\end{split}
\end{equation*}
for some $\epsilon_2 \in (0,1)$. Since the spectral radius of  $I-\Cov_{n+1}B$  is less than one, there is again a norm
on $\R^{N_{\theta}}$ in which the operator norm on $I-\Cov_{n+1}B$ is less than one and exponential convergence
follows.
Equation \ref{eq:mean_infty} can be rewritten as 
\begin{equation*}
\begin{aligned}
    0 &= C_{\infty}\Big(G^T\Sigma_{\nu}^{-1}(y - G \mean_{\infty}) + 
    (1 - \alpha)(G^T\Sigma_{\nu}^{-1}G - C^{-1}_{\infty})\mean_{\infty} \Big)\\
    &= C_{\infty}\Big(G^T\Sigma_{\nu}^{-1}(y - G \mean_{\infty}) - 
    (1 - \alpha) \pCov^{-1}_{\infty}\mean_{\infty} \Big).
\end{aligned}
\end{equation*}
Finally we note that $\mean_{\infty}$ is the minimizer of~\cref{eq:KI299}. 
\end{proof}

\begin{proof}[Proof of Theorem~\ref{th:lin_converge2}]
In this setting where $\alpha=1$ and $\Sigma_\omega=0$ it follows 
from \eqref{eq:Sigma_update} that
\begin{equation}
\begin{split}
    \Cov_{n+1}^{-1} &= G^T\Sigma_{\nu}^{-1}G + \Cov_n^{-1},\\
  \Cov_{n+1}^{-1}m_{n+1} &= G^T\Sigma_{\nu}^{-1}y + \Cov_n^{-1}m_n;  
\end{split}
\end{equation}
so that
\begin{equation}
\begin{split}
~\label{eq:Sigma_update3}
    \Cov_{n}^{-1} &= nG^T\Sigma_{\nu}^{-1}G + \Cov_0^{-1},\\
  \Cov_{n}^{-1}m_{n} &= nG^T\Sigma_{\nu}^{-1}y + \Cov_0^{-1}m_0.  
\end{split}
\end{equation}
This demonstrates that if $C_0$ is positive definite so is $C_n$
for all $n \in \bbN.$ In the variables \eqref{eq:newvar} we obtain
\begin{equation}
\begin{split}
~\label{eq:Sigma_update4} 
    (\Cov_{n}')^{-1} &= n(G')^T\Sigma_{\nu}^{-1}G' + I,\\
  (\Cov_{n}')^{-1}m_{n}' &= n(G')^T\Sigma_{\nu}^{-1}y + m_0'.  
\end{split}
\end{equation}
This gives \eqref{eq:newvar2} and the proof is completed by
applying the projections $P$ and $Q$, noting that $PG'=G'$ 
and $QS=0, QG'=0.$
\end{proof}

\begin{proof}[Proof of Proposition~\ref{th:lin_converge3}]
In this setting recall that we have $\alpha=1$ and $\Sigma_{\omega}\succ 0$. The covariance update equation \eqref{eq:Lin_KF_analysis_cov} can be rewritten as 
\begin{equation}
\begin{split}
~\label{eq:Sigma_update27}
    \Cov_{n+1}^{-1} &= G^T\Sigma_{\nu}^{-1}G + (\Cov_n + \Sigma_{\omega})^{-1},\\
  \Cov_{n+1}^{-1}m_{n+1} &= G^T\Sigma_{\nu}^{-1}y + (\Cov_n + \Sigma_{\omega})^{-1}m_n. 
\end{split}
\end{equation}
Since $\Sigma_{\omega} \succ 0$, the sequence $\{\Cov_n^{-1}\}$ is bounded: $G^T\Sigma_{\nu}^{-1}G\preceq\Cov_n^{-1}\preceq G^T\Sigma_{\nu}^{-1}G + \Sigma_{\omega}^{-1}$ and $\Cov_n \succ 0$. Let us denote $$
{\Cov}_{n}^{'} = \Sigma_{\omega}^{-\frac{1}{2}}\Cov_{n}\Sigma_{\omega}^{-\frac{1}{2}},\quad \mean_n' = \Sigma_{\omega}^{-\frac{1}{2}}\mean_n,\quad G' = G \Sigma_{\omega}^{\frac{1}{2}}, \quad S = (G')^T \Sigma_{\nu}^{-1} G'.$$

First we prove the convergence of $\{\Cov_n^{-1}\}$.
Note that the update equation~\eqref{eq:Sigma_update27} becomes
\begin{equation}
\label{eq:Sigma_update2}
    \begin{split}
    &(\Cov_{n+1}')^{-1} =  f\left((\Cov_{n}')^{-1}\right)
    \end{split}
\end{equation}
where
\begin{equation*}
f(X)=  S + \left(X^{-1} + \I\right)^{-1}.
\end{equation*}
We note that the nullspace of $S$ is equal to the nullspace of $G'$.
Now consider the Ker($G'$) $\otimes$ Range(${G'}^T$) decomposition of the vector space, and the corresponding orthogonal projections $P$ and $Q$. Constraining on Ker($G'$), we have
\begin{equation}
    \begin{split}
    &(\Cov'_{n+1})^{-1} =  \left(\Cov'_{n} + \I\right)^{-1} \prec (\Cov'_{n})^{-1}.\\
    \end{split}
\end{equation}
Since the sequence $\{(\Cov'_{n})^{-1}\}$ is strictly decreasing in the cone of
positive-semidefinite matrices it must have limit $0$. Therefore, we have $\displaystyle \lim_{n\to \infty} (\Cov'_{n})^{-1} = 0$ on Ker($G'$).
Constraining on Range(${G'}^T$), where $S\succ 0$, the update function~\eqref{eq:Sigma_update2} satisfies 
\begin{equation}
\begin{split}
   \left\lVert \frac{d f(X)}{dX} \Delta X \right\rVert =&  \left\lVert (\I+X)^{-1}\Delta X (I+X)^{-1}\right\rVert \\ 
   \leq& \left\lVert \left(\I+X\right)^{-1}\right\rVert^2 \left\lVert \Delta X\right\rVert \\
   \leq& \left\lVert \left(\I+S\right)^{-1}\right\rVert^2 \left\lVert \Delta X\right\rVert \\
   \leq& (1 + \epsilon_1)^{-2} \left\lVert \Delta X\right\rVert, \\
\end{split}
\end{equation}
where $\epsilon_1 > 0$. Therefore, \cref{eq:Sigma_update2} is a contraction map on Range(${G'}^T$), which leads to the convergence of  $(\Cov'_{n})^{-1}$ on that space.
Combining the convergence of $(\Cov'_{n})^{-1}$ on both subspaces, we deduce that $\Cov_n^{-1}$ converges to a singular matrix. We conclude the analysis of the covariance by
noting that Equation~\eqref{eq:Sigma_update27} 
leads to $\Cov_{n+1} \preceq \Cov_n + \Sigma_{\omega}$, 
which implies that $\Cov_{n+1} \preceq \Cov_0 + n\Sigma_{\omega}$
as required for \eqref{eq:ntb}. 

Now we establish the convergence of $\{m'_n\}$. 
The update equation of $\mean'_n$ can be rewritten as 
\begin{equation}
    \begin{split}
\mean'_{n+1} = \mean'_{n} + \Cov'_{n+1}{G'}^T\Sigma_{\nu}^{-1}y - \Cov'_{n+1}S \mean'_{n}.
    \end{split}
\end{equation}
Consider the Range(${G'}^T$) $\otimes$ Ker($G'$) decomposition $\mean'_n = P\mean'_n + Q\mean'_n$, noting that in these coordinates the update equation can be rewritten as 
\begin{subequations}
\begin{align}
    &P\mean'_{n+1} = P\mean'_{n} + P\Cov'_{n+1}{G'}^T\Sigma_{\nu}^{-1}y - P\Cov'_{n+1}S P\mean'_{n}, \label{eq:contracting-1} \\
     &Q\mean'_{n+1} = Q\mean'_{n} + Q\Cov'_{n+1}{G'}^T\Sigma_{\nu}^{-1}y - Q\Cov'_{n+1}S P\mean'_{n}. \label{eq:contracting-2}
\end{align}
\end{subequations}
Now consider the operator $P-P\Cov'_{n+1}SP$ constrained to apply on Range(${G'}^T$). On
this space $S \succ 0$ and, with $\I - S^{\frac{1}{2}}\Cov'_{n+1}S^{\frac{1}{2}}$ also
viewed as acting on Range(${G'}^T$),
\begin{equation}
\label{eq:cov-mean-par}
\begin{split}
\rho(P - P\Cov'_{n+1}SP) &= \rho\Big(\I - S^{\frac{1}{2}}\Cov'_{n+1}S^{\frac{1}{2}}\Big) \\
&
\leq
1 - \rho\Big(S^{\frac{1}{2}}(S + \I)^{-1}S^{\frac{1}{2}}\Big) \\
&= 1 -\epsilon_0,
\end{split}
\end{equation}
where $\epsilon_0 \in (0,1)$. Hence, we deduce that $\{P\mean'_{n}\}$ converges exponentially to $\theta_{ref} := S^{+}{G'}^{T}\Sigma_{\nu}^{-1}y$ where $S^+$ denotes the Moore-Penrose
inverse of $S$. Since $SS^+=\I$ on Range(${G'}^T$), the update equation~\eqref{eq:contracting-2} for  $Qm'_n$ may be written as
\begin{equation}
\begin{split}
    &Q\mean'_{n+1} = Q\mean'_n +  Q  \Cov'_{n+1}S(\theta_{ref} - P\mean'_{n}).
\end{split}
\end{equation}
It follows from \eqref{eq:ntb} that $\displaystyle \|Q \Cov'_{n+1}S\|$ is bounded above
by a function which grows linearly in $n$ in any norm. Furthermore $\displaystyle \lim_{n\rightarrow \infty} Pm'_n - \theta_{ref} = 0$ exponentially fast. Hence we deduce the exponential convergence of  $\{Q\mean'_n\}$ to a limit, depending on $Q\mean'_0.$ 
Therefore, $\{\mean_{n}\}$ converges exponentially fast to a stationary point of 
$\frac{1}{2}\lVert \Sigma_{\nu}^{-\frac{1}{2}} (y - G \theta)\rVert^2$.
\end{proof}

\begin{proof}[Proof of Lemma~\ref{th:filter}]
\begin{equation}
\begin{split}
    \frac{\partial\F\G(\mean, C)}{\partial \mean} &= \frac{\partial\E[\G(\theta)]}{\partial \mean} \\
    &= \int \G(\theta)\frac{1}{\sqrt{(2\pi)^{N_{\theta}}|\Cov|}} \exp\bigl(-\frac{1}{2}\|\Cov^{-\frac12}(\theta -\mean)\|^2\bigr) \big(\Cov^{-1} (\theta - \mean)\big)^{T} d\theta\\
    &= \int \G(\theta)(\theta - \mean)^{T} \frac{1}{\sqrt{(2\pi)^{N_{\theta}}|\Cov|}} \exp\bigl(-\frac{1}{2}\|\Cov^{-\frac12}(\theta -\mean)\|^2\bigr)  d\theta \cdot  \Cov^{-1}\\
    &= \int \bigl(\G(\theta)-\E\G(\theta)\bigr) (\theta - \mean)^{T} \frac{1}{\sqrt{(2\pi)^{N_{\theta}}|\Cov|}} \exp\bigl(-\frac{1}{2}\|\Cov^{-\frac12}(\theta -\mean)\|^2\bigr)  d\theta \cdot  \Cov^{-1}\\
    &= \F d\G(\mean, C).
\end{split}
\end{equation}
\end{proof}

\begin{proof}[Proof of Proposition~\ref{th:uki}]
From~\cref{eq:UKI_smooth_0} we have
\begin{equation}
\begin{split}
    &\py_{n+1} = \F_u\G_{n+1},\\
    &{\pCov_{n+1}^{\theta y}} = {\pCov_{n+1}}\F_u d \G_{n+1}^T.
\end{split}
\end{equation}
In what follow we use the modified unscented transform~\cref{def:unscented_tranform}, and specifically
its use to derive \eqref{eq:uki_sigma-points} and
\eqref{eq:UKI-analysis}. First note that 
\begin{equation*}
\pmean_{n+1} = \pp^0_{n+1}, \quad
    \py_{n+1} = \G(\pp^0_{n+1}) = \py^0_{n+1}, 
    \quad \textrm{and} \quad 
    w = W_1^c = W_2^c = \cdots = W_{2N_{\theta}}^c. 
\end{equation*} 
Now define the matrices
\begin{equation*}
\begin{split}
    &\mathcal{Y}_1 = [\py^1_{n+1} - \py_{n+1} \quad \py^2_{n+1} - \py_{n+1} \quad \cdots \quad \py^{N_{\theta}}_{n+1} - \py_{n+1}], \\
    &\mathcal{Y}_2 = [\py^{N_{\theta}+1}_{n+1} - \py_{n+1} \quad \py^{N_{\theta}+2}_{n+1} - \py_{n+1} \quad \cdots \quad \py^{2N_{\theta}}_{n+1} - \py_{n+1}], \\
    &\Theta = [\pp^1_{n+1} - \pmean_{n+1} \quad \pp^2_{n+1} - \pmean_{n+1}  \quad \cdots \quad \pp^{N_{\theta}}_{n+1} - \pmean_{n+1}].
\end{split}
\end{equation*}
Then we have 
\begin{subequations}
\begin{align}
&\pCov^{\theta y}_{n+1} = \sum_{j=1}^{2N_\theta}W_j^{c}
        (\pp^j_{n+1} - \pmean_{n+1} )(\ppy^j_{n+1} - \py_{n+1})^T  = w \Theta (\mathcal{Y}_1^T - \mathcal{Y}_2^T),\\
&\pCov^{yy}_{n+1} = \sum_{j=1}^{2N_\theta}W_j^{c}
        (\ppy^j_{n+1} - \py_{n+1} )(\ppy^j_{n+1} - \py_{n+1})^T + \Sigma_{\nu} = w (\mathcal{Y}_1\mathcal{Y}^T_1 + \mathcal{Y}_2\mathcal{Y}^T_2) + \Sigma_{\nu},\label{eq:app:cpp}\\
&\pCov_{n+1} = \sum_{j=1}^{2N_\theta}W_j^{c}
        (\pp^j_{n+1} - \pmean_{n+1} )(\pp^j_{n+1} - \pmean_{n+1})^T  = 2w \Theta \Theta^T. \label{eq:app-pCov}
        \end{align}
\end{subequations}
\Cref{eq:app-pCov} follows from the definition of the sigma points~(\ref{eq:uki_sigma-points}).
Since $\pCov_{n+1} \succeq \Sigma_{\omega}\succ 0$, the matrix $\Theta \in \R^{N_{\theta}\times N_{\theta}}$ is non-singular. Thus we have
\begin{equation}
\label{eq:app:fdg}
\begin{split}
\F_u d\G_{n+1}\pCov_{n+1} \F_u d\G_{n+1}^{T} &= {\pCov_{n+1}^{\theta y}}{}^T {\pCov_{n+1}}^{-1} \pCov_{n+1} {\pCov_{n+1}}^{-1} {\pCov_{n+1}^{\theta y}}\\
 &= {\pCov_{n+1}^{\theta y}}{}^T  {\pCov_{n+1}}^{-1} {\pCov_{n+1}^{\theta y}}\\
 &= w(\mathcal{Y}_1 - \mathcal{Y}_2)\Theta^T \Big(2w \Theta \Theta^T\Big)^{-1} \Theta (\mathcal{Y}^T_1 - \mathcal{Y}^T_2)w\\
 &=\frac{w}{2}(\mathcal{Y}_1\mathcal{Y}^T_1 + \mathcal{Y}_2\mathcal{Y}^T_2
 - \mathcal{Y}_1\mathcal{Y}^T_2 
 - \mathcal{Y}_2\mathcal{Y}^T_1).
\end{split}
\end{equation}
Using~\cref{eq:app:fdg} in \cref{eq:app:cpp} yields 
\begin{equation}
\begin{split}
\pCov^{yy}_{n+1} =\F_u d\G_{n+1}\pCov_{n+1} \F_u d\G_{n+1}^{T} + \Sigma_{\nu} + \widetilde{\Sigma}_{\nu, n+1}, 
\end{split}
\end{equation}
where
$$\widetilde{\Sigma}_{\nu, n+1} := \frac{w}{2}(\mathcal{Y}_1 + \mathcal{Y}_2)( \mathcal{Y}_1 + \mathcal{Y}_2)^T.$$ We note that $\widetilde{\Sigma}_{\nu, n+1}$ is positive semi-definite. Furthermore, the $i$-th column of $\mathcal{Y}_1 + \mathcal{Y}_2$ satisfies
\begin{equation}
\begin{split}
    \ppy^i_{n+1} + \ppy^{i+N_\theta}_{n+1} - 2 \py_{n+1} &= \G(\pmean_{n+1} + c_i[\sqrt{\pCov_{n+1}}]_j) + \G(\pmean_{n+1} - c_i[\sqrt{\pCov_{n+1}}]_j) - 2\G(\pmean_{n+1})\\
    &\approx \frac{d^2\G(\pmean_{n+1})}{d^2\theta} : [\sqrt{\pCov_{n+1}}]_j \otimes[\sqrt{\pCov_{n+1}}]_j.
\end{split}
\end{equation}
Hence  $\widetilde{\Sigma}_{\nu,n+1} = 0$ when $\G$ is linear; otherwise $\|\widetilde{\Sigma}_{\nu,n+1}\| = \bigO(\|\pCov_{n+1}^2\|)$, a second order term with small covariance $\pCov_{n+1}$.
\end{proof}

\begin{proof}[Proof of Lemma~\ref{lemma:UKI-Continuous}]
If the steady state $\Cov$ of \cref{eq:cts-b} is singular, then $\exists v\in R^{N_\theta}$ s.t. $v^T\Cov v = 0$.  We have 
\begin{equation*}
    \begin{split}
\Big(v^T\Cov^{\theta y}u\Big)^2 
&= \Big(\E[v^T (\theta - m)\otimes(\G(\theta) - \G(m))u]\Big)^2 \\
&\leq 
\E[v^T (\theta - m)\otimes(\theta - m)v]
\E[u^T (\G(\theta) - \G(m))\otimes(\G(\theta) - \G(m))u] \\
&= 0,
    \end{split}
\end{equation*}
for any $u \in R^{N_y}$.
This implies that $v^T\Cov^{\theta y} = 0$, and therefore,  $$-2\alpha_0 v^T C v - v^T\Cov^{\theta y} {\Sigma_{\nu}}^{-1}{\Cov^{\theta y}}^T v = 0,$$
which contradicts the assumption that $\Sigma_{\omega} \succ 0$.
\end{proof}





\section{Illustrative Examples for UKS}
\label{sec:app:UKS}

The primary focus of the paper is on using the UKI for optimization
purposes. However the basic ingredients of the method, and the dynamical
system \eqref{eq:uks} in particular, can also be used to
perform approximate posterior sampling from the measure $\mu$ given by \eqref{eq:post}.
In the case where $\mu$ is Gaussian, the posterior is exactly captured by the
steady state of these equations; when the posterior is not Gaussian, then only an approximation
is obtained. To illustrate the UKS,  we consider, in Subsection \ref{ssec:luks},
application to three linear inverse problems from Subsection
\ref{ssec:lin}, for which the posterior is Gaussian if the prior is Gaussian;
and then give a simple example of application to a non-Gaussian posterior
in Subsection \ref{ssec:nuks}.

The UKS equations~(\ref{eq:uks}) can be discretized by the following semi-implicit scheme
\label{sec:app-lin-UKS}
\begin{equation}
\begin{split}
    &\mean_{n+1} - \mean_{n} = h\Big(\Cov^{\theta y} \Sigma_{\eta}^{-1} \bigl(y - \E \G(\theta)\bigr)-C\Sigma_0^{-1}(\mean_{n+1} - r_0)\Big),\\
    &\Cov_{n+1} - \Cov_n     =h\Big(-2\Cov^{\theta y} \Sigma_{\eta}^{-1}{\Cov^{\theta y}}^T - 2\Cov_{n}\Sigma_0^{-1}\Cov_{n} + 2\Cov_{n+1}\Big),
\end{split}
\end{equation}
with a fixed time-step. The integrals defining $\Cov^{\theta y}$ and $\E \G(\theta)$ are explicitly approximated by the modified unscented transform (see ~\cref{def:unscented_tranform}) using the
Gaussian $\N(\mean_n, \Cov_n)$. Integration could also be performed using an adaptive time-step,
as in \cite{cleary2020calibrate}; however more work is needed to develop efficient methods 
stemming from the UKS as formulated here.

\subsection{Linear 2-parameter Model Problem}
\label{ssec:luks}
The linear 2-parameter model problems discussed in \cref{sec:app-lin} are used with prior 
$$r_0 = 0\quad \textrm{ and } \quad \Sigma_0 = \I.$$
Therefore, the posterior distribution is $\mu \sim \N(\mean_{ref}, \Cov_{ref})$, where
\begin{equation}
   \mean_{ref} = \Big(\Sigma_0^{-1} + G^T\Sigma_{\eta}^{-1}G\Big)^{-1}\Big(G^T\Sigma_{\eta}^{-1}y + \Sigma_0^{-1}r_0\Big) \quad\textrm{and} \quad
   \Cov_{ref} = \Big(\Sigma_0^{-1} + G^T\Sigma_{\eta}^{-1}G\Big)^{-1}.
\end{equation}

The UKS is initialized with $\theta_0 \sim \N(r_0, \Sigma_0)$. The convergence of the UKS, in terms of the posterior mean and covariance errors for $t \in [0, 10]$ are reported in \cref{fig:UKS}. Both mean and covariance converge to the posterior mean and covariance. However, even with the semi-implicit scheme the maximum time step that allows for stable simulation is $h = 5\times10^{-5}$.

\begin{figure}[ht]
\centering
    \includegraphics[width=0.49\textwidth]{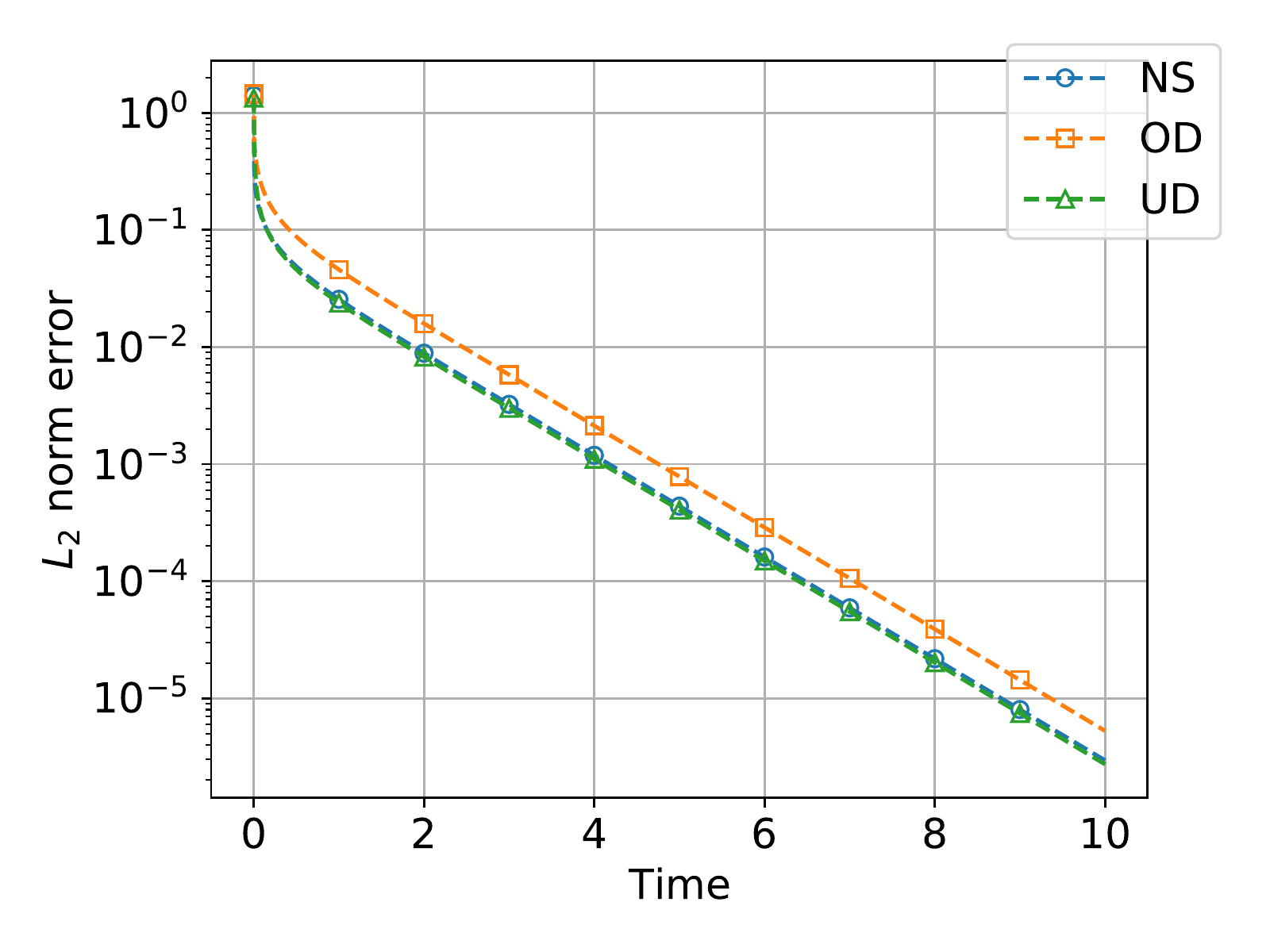}
    \includegraphics[width=0.49\textwidth]{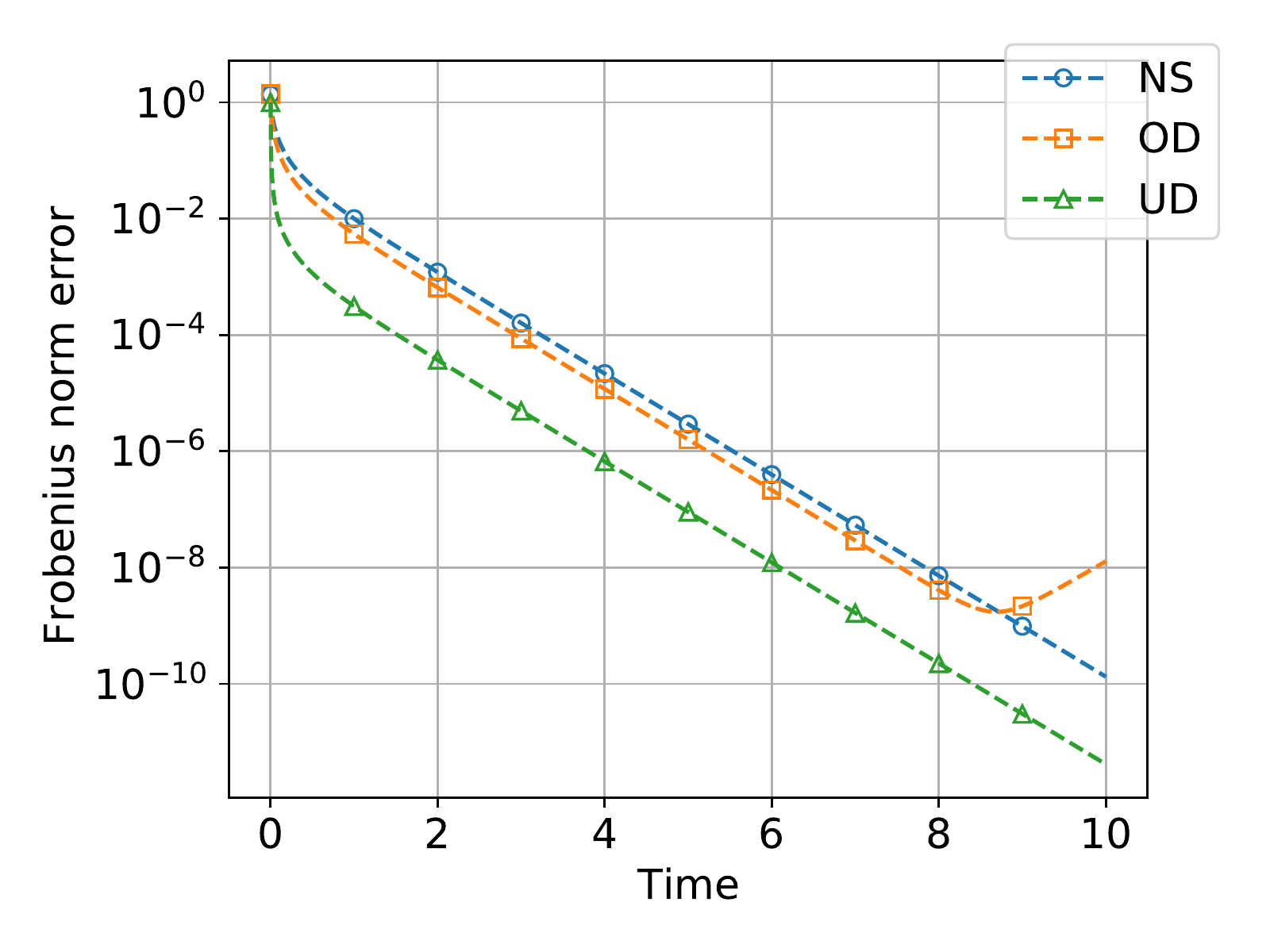}
    \caption{$L_2$ error $\lVert\mean_n - \mean_{ref}\rVert_2$~(left) and Frobenius norm $\lVert\Cov_n - \Cov_{ref}\rVert_F$~(right) obtained by UKS for non-singular~(NS), over-determined~(OD), and under-determined~(UD) systems of the linear 2-parameter model problem.}
    \label{fig:UKS}
\end{figure}

\subsection{Nonlinear 2-Parameter Model Problem}
\label{ssec:nuks}
The following Bayesian logistic regression problem is considered, 
\begin{equation*}
    y = \frac{1}{1 + \exp(\theta_{(1)} + \theta_{(2)}x)} + \eta.
\end{equation*}
Here $N_{\theta} = 2$ and $N_y=1$, and hence this is an under-determined problem.
The prior distribution $\N(r_0, \Sigma_0)$ satisfies 
$$r_0 = [1 \quad 1]^T \quad \textrm{ and } \quad \Sigma_0 = \I.$$
The observation data $y_{ref} = 0.08$ is generated at $x = \frac{1}{2}$, with observation error $\eta \sim\N(0,0.1^2)$ and $\theta_{ref} = [2\quad 2]^T$.

The UKS is initialized with $\theta_0 \sim \N(r_0, \Sigma_0)$. The posterior distributions obtained by the UKS at $t = 10$ with a time step $h = 5\times10^{-5}$ and Markov chain Monte Carlo method~(MCMC) with a step size $1.0$ and $5\times10^6$ samples (with a $10^6$ sample burn-in period) are presented in~\cref{fig:UKS-nonlinear}. The estimated posterior distributions are in 
reasonably good agreement, but of course not as accurate as in the linear setting
in the previous subsection, because of a Gaussian approximation being made to a non-Gaussian
distribution.
Specifically, the posterior mean and covariance estimated by the UKS are 
$$
[1.41 \quad 1.20]^T \quad \textrm{ and } \quad
\begin{bmatrix}
0.526 & -0.235 \\ 
-0.235 & 0.884 \\
\end{bmatrix},
$$
whilst the posterior mean and covariance estimated by the MCMC are 
$$[1.62 \quad 1.31]^T \quad \textrm{ and } \quad
\begin{bmatrix}
0.619 & -0.254 \\ 
-0.254 & 1.00 \\
\end{bmatrix}.
$$
\begin{figure}[ht]
\centering
    \includegraphics[width=0.49\textwidth]{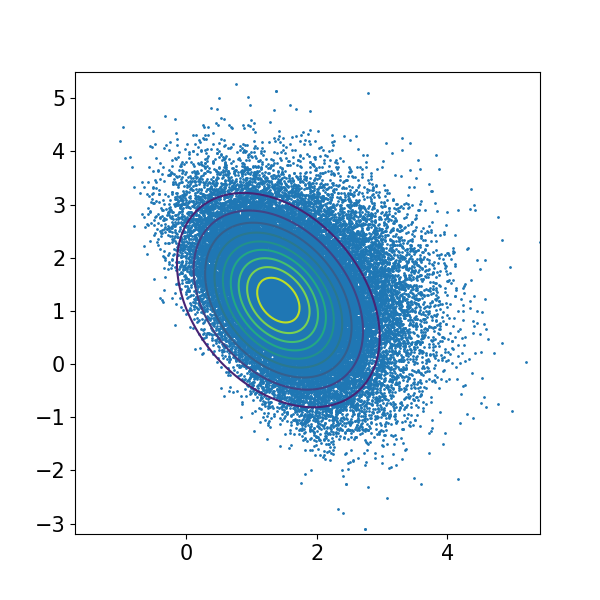}
    \caption{Contour plot: posterior distributions obtained by UKS at $t=10$; blue dots: reference posterior distribution obtained by MCMC for the nonlinear 2-parameter model problem. x-axis is for $\theta_{(1)}$ and y-axis is for $\theta_{(2)}$.}
    \label{fig:UKS-nonlinear}
\end{figure}

\bibliographystyle{unsrt}
\bibliography{references}
\end{document}